\documentclass[12pt,reqno]{amsart}
\usepackage[cp1251]{inputenc}
\usepackage{amsmath,eucal,verbatim,amsopn,amsthm,amsfonts,amssymb,mathdots,mathrsfs,esint,mathtools,enumitem,bbm,mathdots,url}

\usepackage[margin=1in]{geometry}
\let\mathcal=\CMcal
\allowdisplaybreaks[4]

\usepackage[hyperfootnotes=false,linktocpage=true,colorlinks=true,allcolors=blue]{hyperref}

\def\Hom{\operatorname{Hom}}
\def\Bil{\operatorname{Bil}}

\def\diag{\mathrm{diag}}
\def\Mat{\mathrm{Mat}}
\DeclareMathOperator{\tr}{tr}
\DeclareMathOperator{\GL}{GL}
\DeclareMathOperator{\SL}{SL}
\DeclareMathOperator{\SO}{SO}
\DeclareMathOperator{\Orth}{O}
\DeclareMathOperator{\GSpin}{GSpin}
\DeclareMathOperator{\Spin}{Spin}
\DeclareMathOperator{\Sp}{Sp}
\DeclareMathOperator{\Real}{Re}

\def\Ind{\operatorname{Ind}}
\def\ind{\operatorname{ind}}

\newcommand{\A}{\ensuremath{\mathbb{A}}}
\newcommand{\C}{\ensuremath{\mathbb{C}}}
\newcommand{\R}{\ensuremath{\mathbb{R}}}

\newcommand{\Z}{\ensuremath{\mathbb{Z}}}
\newcommand{\cD}{\ensuremath{\mathcal{D}}}

\newcommand{\cF}{{\mathcal{F}}}
\newcommand{\Sym}{\operatorname{Sym}}
\newcommand{\cE}{{\mathcal{E}}}
\newcommand{\Fre}{{Fr\'{e}chet \,}}
\newcommand{\hot}{\widehat{\otimes}}
\newcommand{\Span}{{\operatorname{Span}}}
\newcommand{\alp}{{\alpha}}
\newcommand{\Supp}{\mathrm{Supp}}
\newcommand{\cO}{{\mathcal{O}}}

\newtheorem{theorem}{Theorem}[section]

\newtheorem{lemma}[theorem]{Lemma}
\newtheorem{proposition}[theorem]{Proposition}
\newtheorem{corollary}[theorem]{Corollary}

\newtheorem{remark}[theorem]{Remark}
\newtheorem{example}[theorem]{Example}
\newtheorem{definition}[theorem]{Definition}
\newtheorem{notn}[theorem]{Notation}
\numberwithin{equation}{section}
\newtheorem{theo}{Theorem}[]

\usepackage{lipsum}

\makeatletter
    \def\@pnumwidth{2em}
  \makeatother

\begin{document}
\title[Multiplicity one theorems for generalized doubling]{Multiplicity one theorems for the generalized doubling method}

\address{Avraham Aizenbud,
Faculty of Mathematics and Computer Science,
Weizmann Institute of Science,
POB 26, Rehovot 76100, Israel}
\email{aizenr@gmail.com}
\urladdr{\url{http://www.wisdom.weizmann.ac.il/~aizenr}}

\author[D. Gourevitch]{Dmitry Gourevitch}
\address{Dmitry Gourevitch,
Faculty of Mathematics and Computer Science,
Weizmann Institute of Science,
POB 26, Rehovot 76100, Israel }
\email{dmitry.gourevitch@weizmann.ac.il}
\urladdr{\url{http://www.wisdom.weizmann.ac.il/~dimagur}}

\author[E. Kaplan]{Eyal Kaplan}

\address{Eyal Kaplan,
Department of Mathematics, Bar Ilan University, Ramat Gan 5290002, Israel}
\email{kaplaney@gmail.com}
\urladdr{\url{http://www.math.biu.ac.il/~kaplanea}}

\thanks{This research was supported by the ERC, StG grant number 637912 (Gourevitch),
and by the ISRAEL SCIENCE FOUNDATION, grant numbers 249/17 (Gourevitch) and
421/17 (Kaplan).}



\begin{abstract}
In this work we prove the local multiplicity at most one theorem underlying the definition and theory of local
$\gamma$-, $\epsilon$- and $L$-factors, defined by virtue of the generalized doubling method, over any local field of characteristic $0$.
We also present two applications: one to the existence of local factors for genuine representations of covering groups, the other to the global unfolding argument of the doubling integral.
\end{abstract}
\maketitle

The doubling method of \cite{PSR,CFGK2,CFK} constructs an integral representation for the tensor product of a pair of irreducible cuspidal automorphic representations of $G(\A)$ and $\GL_k(\A)$, for a range of reductive groups $G$ defined over a number field $F_0$ with a ring of adeles $\A$. One of the advantages of this method, is that it does not rely on the existence of a model (or a nonzero Fourier coefficient) for the representation of $G(\A)$; it is applicable to any cuspidal representation.
The family of local integrals can be used to define local $\gamma$-, $\epsilon$- and $L$-factors. These factors are defined for arbitrary irreducible admissible representations, and as such, generalize the corresponding (tensor) factors defined by Shahidi \cite{Sh3} for irreducible generic representations, using his method of local coefficients. We prove the local multiplicity at most one theorem underlying the definition of the local factors.

Let $F$ be a local field of characteristic $0$. Let $G$ be one of the split groups $\Sp_c$, $\SO_c$, $\GSpin_c$ or $\GL_c$, where in the symplectic case $c$ is even. For an integer $k$, let $H$ be the split group of the same type as $G$, which is either $\Sp_{2kc}$, $\SO_{2kc}$, $\GSpin_{2kc}$ or $\GL_{2kc}$. There is a unipotent subgroup $U$ of $H$, and a character $\psi_U$ of $U$ which is generic with respect to the unipotent orbit $((2k-1)^c1^c)$ associated with $H$, such that $G\times G$ can be mapped into the normalizer of $U$ and stabilizer of $\psi_U$.
We denote the image of $G\times G$ under this map by $(G,G)$ and let $D$ be the subgroup $U\rtimes (G,G)$ of $H$.

We identify $F$-groups with their $F$-points. The underlying principle of the doubling construction is the multiplicity at most one property of the restriction to $D$, of representations of $H$ parabolically induced from certain degenerate representations of $\GL_{kc}$.

A representation $\rho$ of $\GL_{kc}$ is called a $(k,c)$ representation if its wave-front set contains $(k^c)$ as the unique maximal orbit, and its degenerate Whittaker model with respect to this orbit is unique. The simplest examples are the representation $\det$ of $\GL_c$ or its twist by a quasi-character of $F^*$, which is a $(1,c)$ representation, or irreducible generic representations of $\GL_k$, which are $(k,1)$. The generalized Speh representation $\rho_c(\tau)$ of $\GL_{kc}$ attached to $c$ copies of an irreducible unitary representation $\tau$ of $\GL_k$ is $(k,c)$ (\cite[Theorem~5]{CFK}).

Let $P$ be a maximal parabolic subgroup of $H$ which is a Siegel parabolic subgroup if $G\ne\GL_{kc}$.
Denote the Levi part of $P$ by $M_P$. If
$M_P=\GL_{kc}$, let $\rho$ be a $(k,c)$ representation of $M_P$. For a complex parameter $s$, consider the space $V(s,\rho)$ of the representation of $H$ parabolically induced from
$|\det|^{s-1/2}\rho$ and $P$ to $H$. In the other cases of subgroups $M_P$ the representation $V(s,\rho)$ slightly varies: for $M_P=\GL_{kc}\times\GL_1$ we induce from $\rho\otimes\eta$ with a quasi-character $\eta$ of $F^*$, and for $M_P=\GL_{kc}\times\GL_{kc}$, $\rho$ is the (exterior) tensor product of two $(k,c)$ representations of $\GL_{kc}$.

\begin{theo}\label{theorem A}(see Theorem~\ref{theorem:uniqueness})
Let $\pi_1$ and $\pi_2$ be irreducible admissible representations of $G$, and $\rho$ be an admissible finite length $(k,c)$ representation of $\GL_{kc}$. Outside a discrete subset $\mathcal{B}\subset\C$ of $s$,
\begin{align*}
\dim\Hom_{D}(V(s,\rho),\psi_U^{-1}\otimes\pi_1\otimes\pi_2)\leq\dim\Hom_{G}(\pi_1^{\vee},\pi_2^{\iota}).
\end{align*}
Over non-archimedean fields, for supercuspidal representations $\pi_1$ and $\pi_2$ the result holds for all $s$, under certain additional conditions.
\end{theo}
Here $\iota$ is a certain involution of $G$; when $G=\Sp_c$, $\pi^{\iota}=\pi^{\vee}$, and for $\GL_c$, $\iota$ is trivial.
As usual, if the field is non-archimedean and its residue field contains $q$ elements, the set $\mathcal{B}$ consists of finitely many values of $q^{-s}$. For the stronger statement for supercuspidal representations we must exclude minimal rank cases where $G$ does not contain nontrivial unipotent subgroups, and for $\GL_c$ there is an additional condition on $\rho$. In the setup of the doubling method of \cite{CFGK2,CFK}, $(\pi_1,\pi_2)=(\pi^{\vee},\pi^{\iota})$ and the dimension is precisely $1$ outside a discrete subset of $s$. Note that there is no canonical isomorphism between the spaces appearing in the theorem.
For the definition of all objects and notation, and for the more precise statement, see
\S~\ref{preliminaries}, in particular \S~\ref{Doubling setup} where we recall the generalized doubling setup, and Theorem~\ref{theorem:uniqueness} (e.g., $\iota$ is defined in \S~\ref{Doubling setup}).

The case $k=1$ of the theorem for supercuspidal representations was proved by Harris \textit{et al.} \cite[\S~4]{HKS},
but the general setting of the theorem (even for $k=1$) has not been studied. In this sense we close a historical gap.

Theorem~\ref{theorem A} is the local counterpart of the global unfolding argument in \cite{CFGK2}.
We briefly recall the global result, focusing on the parts relevant to us here. For more detail on the global setting see \S~\ref{Global unfolding}.

Let $\tau$ be an irreducible cuspidal automorphic representation of $\GL_k(\A)$, and $\mathcal{E}_{\tau}$ be the generalized Speh representation of $\GL_{kc}(\A)$ corresponding to $c$ copies of $\tau$, defined by Jacquet \cite{Jac4}. The representation $\mathcal{E}_{\tau}$ is a global $(k,c)$ representation, in the sense that it does not support any Fourier coefficient along an orbit greater than or non-comparable with $(k^c)$, it supports a Fourier coefficient along $(k^c)$, and all of its local components are $(k,c)$ (\cite{G2,JL2013,CFGK2,CFK}). Let $E(h;s,f)$ denote the Eisenstein series attached to a suitable section $f$ in the space of the representation of $H(\A)$ parabolically induced from $|\det|^{s-1/2}\mathcal{E}_{\tau}$ and $P(\A)$. One can consider the Fourier coefficient
$E^{U,\psi_U}(h;s,f)$ of $E(h;s,f)$ along $(U,\psi_U)$ (see \eqref{eq:U psi U coefficient of the series}) as an automorphic function on $G(\A)\times G(\A)$. The global integral was defined
in \cite{CFGK2} by integrating $E^{U,\psi_U}(h;s,f)$ against two cusp forms in the space of a unitary irreducible cuspidal automorphic representation $\pi$ of $G(\A)$. In a right half plane $\Real(s)\gg0$, one can rewrite the integral as a sum (of integrals) parameterized by representatives of $P(F_0)\backslash H(F_0)/D(F_0)$. All summands but one vanish, and the remaining summand was shown to produce an Eulerian integral.

There are $3$ methods for showing the vanishing of a summand. The first is by finding a subgroup $U'<U$ such that
$\psi_U$ is nontrivial on $U'(\A)$, and showing that the summand admits an inner integral of $\psi_U$ over $U'(F_0)\backslash U'(\A)$, which is then zero. Second, if the summand admits an inner integral which constitutes a Fourier coefficient of $\mathcal{E}_{\tau}$, that is
greater than or non-comparable with $(k^{c})$. This summand vanishes because $\mathcal{E}_{\tau}$ is $(k,c)$.
Third, if one can obtain an inner integral of one of the cusp forms along a unipotent radical $U_R$ of a parabolic subgroup $R=M_R\ltimes U_R$ of $G$, then the summand vanishes because $\pi$ is cuspidal.

Our local result is, in some sense, parallel to the global unfolding. We consider distributions on
the orbits of $P\backslash H/D$. The argument involving $\psi_U$ can be applied locally. The second case, where we use the $(k,c)$ representation alone, is not difficult to carry out in the non-archimedean setting, using the local ``exchange of roots" arguments of Ginzburg \textit{et al.} \cite{GRS5} and the theory of derivatives of Bernstein and Zelevinsky \cite{BZ1,BZ2}. Results involving the Jacquet functor are in general more difficult and subtle over archimedean fields. Fortunately, we are able to benefit from the recent (partial) extension of the theory of derivatives to archimedean fields by Aizenbud \textit{et al.} \cite{AGS2015a,AGS2015b}. In fact, our argument in this case is greatly simplified and streamlined using the precise reformulation of
Gomez \textit{et al.} \cite{GGS} of the connection between the wave-front set and the theory of derivatives, over both archimedean and non-archimedean fields.

The class of double cosets, where the global vanishing follows using the fact that the representations of $G$ are cuspidal, require a different approach. The difficulty arises because in the local setting the class of representations of $G$ must not be restricted to supercuspidal ones. This is where we lose the subset $\mathcal{B}$.

In more detail, if the global summand was treated using, say, $U_R<R<G$, the corresponding orbit should be handled by analyzing the action of the center $C_{M_R}$ of $M_R$. Consider the non-archimedean setting. Following the method of Jacquet \textit{et al.} \cite{JPSS}, if the local representations (which are usually Jacquet modules of $\pi_i$ and $\rho$) restrict to finite length representations of $M_R$, the action of $C_{M_R}$ is filtered by a finite sequence of quasi-characters, combined with a quasi-character determined by $|\det|^s$. This produces a compatibility condition that rules out a discrete subset of $s$, i.e., there are no distributions on the orbit unless $s$ belongs to a discrete set. This argument was carried out by several authors, including \cite{GPR,Soudry,GRS5,me5}. The main obstacle was showing that indeed the representations involved restrict to finite length representations of a Levi subgroup, or more precisely in those works, of the reductive part of the appropriate mirabolic subgroup, and the key tool in the proofs was the theory of derivatives \cite{BZ1,BZ2}. In contrast, here the Jacquet modules of $\rho$ that occur in the analysis do not afford a representation of the mirabolic subgroup, but we are still able to show that they restrict to finite length representations of $M_R$.
Moreover, while in the previous aforementioned works the number of double cosets was finite, here there are in general uncountably many (unless $k=1$). Therefore we must be careful to apply this argument to only finitely many representatives.

In fact, treating uncountably many orbits is another difficulty. In the non-archimedean case, in principle if there are no distributions (satisfying certain equivariance properties) on the orbits, there are no global distributions (see e.g., \cite[\S~6]{BZ1}).
Over archimedean fields this is considerably more complicated. Kolk and Varadarajan \cite{KV} extended parts of the archimedean Bruhat theory to this case using transverse symbols. In Appendix~\ref{Appendix} by the first named author and Avraham Aizenbud, we will present a generalization of the main result of \cite{KV}, which is sufficiently strong for our application and is of independent interest, using tools from functional analysis.

The proof here clarifies several arguments of \cite{CFGK2}, and is applicable to a wide class of groups, in particular all groups treated in \cite{CFK} (the unfolding argument in \cite{CFGK2} was presented only for symplectic groups). The original doubling method of \cite{PSR} was stated for a slightly different class of groups, e.g., the full orthogonal groups $\Orth_c$, and also unitary groups; general spin groups, as well as the double cover of the symplectic group, were mentioned in \cite[\S~4.3]{PSR} but not treated in any way. The extension of the present proof to other classes of groups would be straightforward.

Note that the aforementioned proof in \cite[\S~4]{HKS} does not involve $\psi_U$ nor the structure of $\rho$ (at any rate $\rho$ for $k=1$
is plainly a character), and involves only finitely many orbits. Once the proof here is reduced to the open orbit, it ``mirrors" the arguments of \textit{loc. cit.} when the representations are supercuspidal, but again in the general case it is more difficult and relies on the deep properties of $\rho$.

Our work has two immediate applications. The first concerns covering groups (topological central extensions by finite cyclic groups). The classical doubling method ($k=1$) was extended by Gan \cite{Gan} to the double cover of the symplectic group. In the recent work \cite{me12}, the doubling method was extended to $m$-fold coverings $\Sp_c^{(m)}(\A)$ of $\Sp_c(\A)$ (defined by \cite{Mats}) for all $m$, and any $k$, providing an integral representation for the tensor product of a pair of genuine irreducible cuspidal automorphic representations $\pi$ of $\Sp_c^{(m)}(\A)$ and $\tau$ of $\widetilde{\GL}_k(\A)$, where $\widetilde{\GL}_k(\A)$ is a covering group of $\GL_k(\A)$ defined by restriction from $\Sp_{2k}^{(m)}(\A)$. Alongside, the local doubling construction for $\widetilde{\GL}_c$ was developed as well (for all $m$ and $k$). The construction of \cite{me12} is still subject to a local and global conjecture regarding generalized Speh representations, but the local theory for unramified data
over non-archimedean fields does not depend on these conjectures.

Theorem~\ref{theorem A} can be reformulated for covering groups, granted certain conditions hold (see \S~\ref{Covering groups} for details). In the particular cases of $\Sp_c^{(m)}$ and $\widetilde{\GL}_c$, Theorem~\ref{theorem A} is applicable and as a consequence, we can define local factors using uniqueness, at least when data are unramified. Specifically, define the $\gamma$-factor as the proportionality between two integrals, then use it to define $\epsilon$- and $L$-factors; see \eqref{eq:gamma factor} and the explanation below in the linear setting. To the best of our knowledge, at present no other method for an analytic definition of these factors is known (see below; of course, for a formal abstract definition of local factors for unramified data one can use the Satake parametrization). Moreover, granted the conjectures of \cite{me12}, Theorem~\ref{theorem A} is expected to imply the existence of local $\gamma$-, $\epsilon$- and $L$-factors in general, i.e., also in the ramified case, for $\Sp_c^{(m)}\times \widetilde{\GL}_k$ (and additional covering groups).

As explained above, the doubling method does not rely on the existence of a model for the representation of $G$. This is advantageous for linear groups, but even more so when considering covering groups. As a rule, Whittaker models are not unique for representations of covering groups (the double cover of $\Sp_c$ is an exception), first and foremost, for genuine irreducible unramified principle series representations. This means that Shahidi's theory of local coefficients is no longer applicable, even in the unramified setting. The fact that one can still define local factors using analytic methods and uniqueness, is perhaps a surprise.

We note that the number of Whittaker models is still finite. In recent works, Gao \textit{et al.} \cite{GaoShahidiSzpruch2018,GaoShahidiSzpruch2019} and Szpruch \cite{Szpruch2018} studied generalizations of Shahidi's local coefficients, namely a local coefficients matrix and a scattering matrix, and extracted interesting representation theoretic invariants.

The second application is global. Since the local components of $\mathcal{E}_{\tau}$ are $(k,c)$ representations, our local analysis expresses the Fourier coefficient of the Eisenstein series as a sum over a finite number of cosets in (the infinite space) $P(F_0)\backslash H(F_0)/D(F_0)$. Then it is visible, that the integral of this coefficient against two cusp forms reduces to a single summand, explicating the unfolding process. In this sense we fill out the gap of the unfolding for the cases of groups considered in \cite{CFK}; although our arguments also readily globalize.

The original doubling method of Piatetski-Shapiro and Rallis \cite{PSR} produced an integral representation for the
standard automorphic $L$-function of an irreducible cuspidal automorphic representation of a classical group, or its rank-$1$ twists, which is the case $k=1$. The local theory for $k=1$ was fully developed by Lapid and Rallis \cite{LR}.
The doubling construction was extended to arbitrary $k$ in \cite{CFGK2}, and the corresponding theory of local factors was developed in \cite{CFK}.

We briefly explain how Theorem~\ref{theorem A} is used for the definition of the local factors. Fix a nontrivial additive character $\psi$ of $F$. Let $\pi$ be an irreducible admissible representation of $G$, and $\tau$ be an irreducible admissible and generic representation of $\GL_k$. If $\tau$ is unitary, the representation
$\rho_c(\tau)$ was introduced above, in general $\rho_c(\tau)$ is defined using Langlands' classification and the tempered case.
The local doubling integral $Z(s,\omega,f)$ is defined for a matrix coefficient $\omega$ of $\pi^{\vee}$ and a holomorphic section $f$ of $V(s,\rho_c(\tau))$. In its domain of absolute convergence (a right half plane), $Z(s,\omega,f)$ defines a morphism in
\begin{align}\label{space:homspace for proportionality intro}
\Hom_{D}(V(s,\rho_c(\tau)),\psi_U^{-1}\otimes\pi^{\vee}\otimes\pi^{\iota}).
\end{align}
Applying a standard intertwining operator
\begin{align*}
M(s,w):V(s,\rho_c(\tau))\rightarrow V(1-s,{}^w\rho_c(\tau)),
\end{align*}
where ${}^w\rho_c(\tau)=\rho_c(\tau^{\vee})$
for $G=\Sp_c,\SO_c$, we obtain a second integral $Z(s,\omega,M(s,w)f)$, absolutely convergent in a left half plane, which still defines a morphism in \eqref{space:homspace for proportionality intro}. In fact $M(s,w)$ is further normalized using a second functional equation,
$M^*(s,w)=C(s,c,\tau,\psi)M(s,w)$, where $C(s,c,\tau,\psi)$ is a meromorphic function of $s$ (see \cite[\S~4]{CFK}).
By Theorem~\ref{theorem A} we can define the $\gamma$-factor
$\gamma(s,\pi\times\tau,\psi)$ by
\begin{align}\label{eq:gamma factor}
\gamma(s,\pi\times\tau,\psi)Z(s,\omega,f)=Z^*(s,\omega,f),\qquad Z^*(s,\omega,f)=Z(1-s,\omega,M^*(s,w)f).
\end{align}
The main local result of \cite{CFK} was the characterization of this factor, according to the prescribed list of properties formulated
by Shahidi in the context of generic representations \cite[Theorem~3.5]{Sh3} (see also \cite[Theorem~4]{LR}). In turn, the $\gamma$-factor was used in \cite{CFK} to define the local $\epsilon$- and $L$-factors, following Shahidi's method \cite[\S~7]{Sh3} (see also \cite{Sh2,Sh4,Shahidi1983,Shahidi1985}). The main motivation of \cite{CFK} was to find a new proof of global functoriality from $G(\A)$ to the appropriate general linear group, via the Converse Theorem of Cogdell and Piatetski-Shapiro \cite{CPS3,CPS1999}, thereby extending the global result of \cite{CKPS2,CKPS,AsgSha} from globally generic representations to arbitrary cuspidal ones. Of course the endoscopic functorial transfer for quasi-split orthogonal or symplectic groups was obtained by Arthur \cite{Arthur2013} using the twisted stable trace formula, and extended to quasi-split unitary groups by Mok \cite{Mok2015}.

We mention that if $\pi$ is supercuspidal (and certain additional assumptions), our uniqueness results imply, using Bernstein's continuation principle (\cite{Banks}), that the integral is holomorphic (see Corollary~\ref{coro:meromorphic continuation for doubling 2}). Using the fact that the integral can always be made constant, it follows that
the only poles appearing in $\gamma(s,\pi\times\tau,\psi)$ are poles of $M^*(s,w)$. This observation hints that
a ``g.c.d. definition" of the $L$-function using the generalized doubling method must involve ``good sections"
(see e.g., \cite[589--590]{me4}). Indeed this was the approach of Yamana \cite{Yamana}, who studied this definition of the $L$-function for $k=1$.

Similar multiplicity at most one theorems exist in the literature. In the context of Rankin--Selberg integrals for representations of $G\times\GL_k$ admitting unique Whittaker models, where $G$ is a classical group, see \cite{GPR,Soudry,Soudry3,GRS5,me5}. See also \cite{AGRS,JSZ,GGP,MW,Sun2012,SunZhu2012,Waldspurger2012,LiuSun2013} who proved strong general uniqueness results (which in particular imply multiplicity one for the same Rankin--Selberg constructions with irreducible generic representations). Our proof technique resembles Soudry's \cite[\S~8]{Soudry} and \cite{Soudry3}.

In a more general context, for a representation $\rho$ of an arbitrary group $H$, a subgroup $D<H$ and a representation $\xi$ of $D$,
one can consider the space $\Hom_D(\rho,\xi)$. Typical questions involve the multiplicity of this space, or the structure of $\xi$ for which $\Hom_D(\rho,\xi)\ne0$. In certain cases, the nonvanishing is related to special values of $L$-functions. Globally, one is often interested in a period integral of an automorphic form on $H(\A)$, over $D(F_0)\backslash D(\A)$ (with $\xi=1$). There is a vast amount of studies of such problems, let us mention \cite{Z3,Jac2,JR,FJ,Offen,OS2,OS,Offen2,Jac3,Matringe4,Matringe3,Matringe2,Matringe,FLO,Matringe5,Yamana2}.

For other works on the doubling method see, e.g., \cite{KudlaRallis1989,HKS,BochererSchmidt2000,HarrisLiSkinner2005,HarrisLiSkinner2006,WGS}.
The doubling method is not the only integral representation to lift the barrier of globally generic representations: other constructions of similar generality were developed, thus far without complete local theory, in \cite{GPSR,BS,GJRS,JZ,Soudry7,Soudry8}. While the local ramified theory is probably within reach (see e.g., \cite[Theorem~4.2]{Soudry7}), it will require an abundance of work. E.g., these integrals are far less uniform than the doubling method, and for orthogonal groups one uses the Bessel model of $\pi$ which involves an auxiliary representation.

The rest of this work is organized as follows. In \S~\ref{preliminaries} we provide some general preliminaries, define $(k,c)$ representations and recall the doubling construction of \cite{CFGK2,CFK}. The proof of our main result is given in \S~\ref{Uniqueness results}.
Section~\ref{Applications} contains our main applications.

Parts of the non-archimedean version of Theorem~\ref{theorem A} for supercuspidal representations appear in \cite{CaiTwisted}; him and the authors were working independently.

\addtocontents{toc}{\protect\setcounter{tocdepth}{1}}

\subsection*{Acknowledgements}
We are grateful to the referees for their interest in this work and
helpful remarks, which helped improve the presentation.

\addtocontents{toc}{\protect\setcounter{tocdepth}{2}}

\makeatletter
  \def\@pnumwidth{2em}
  \def\@tocrmarg {3.5em}
\makeatother

\tableofcontents
\section{Preliminaries}\label{preliminaries}
\subsection{The groups}\label{the groups}

Let $l\geq1$ be an integer. Let $B_{\GL_l}=T_{\GL_l}\ltimes N_{\GL_l}$ denote the Borel subgroup of upper triangular invertible matrices, where $N_{\GL_l}$ is its unipotent radical. The standard parabolic subgroups of
$\GL_l$ can be identified with the set of compositions $\beta=(\beta_1,\ldots,\beta_a)$ of $l$ ($\beta_i\geq0$, $a\geq1$), where $P_{\beta}=M_{\beta}\ltimes V_{\beta}$ denotes the parabolic subgroup with $M_{\beta}=\GL_{\beta_1}\times\ldots\times \GL_{\beta_a}$ and $V_{\beta}<N_{\GL_l}$.
Let $J_l$ be the permutation matrix with $1$ on the anti-diagonal and $0$ otherwise. For $g\in\GL_l$, ${}^tg$ denotes the transpose of $g$, and $g^*=J_l{}^tg^{-1}J_l$.

For $x\in \R$, $\lfloor x \rfloor$ (resp., $\lceil x \rceil$) denotes the largest integer smaller (resp., greater) than or equal to $x$.

For an even $l$, define
\begin{align*}
\Sp_{l}=\{g\in\GL_l:{}^tg\left(\begin{smallmatrix}&J_{l/2}\\-J_{l/2}\end{smallmatrix}\right)g=\left(\begin{smallmatrix}&J_{l/2}\\-J_{l/2}\end{smallmatrix}\right)\}.
\end{align*}
Let $B_{\Sp_{l}}=\Sp_{l}\cap B_{\GL_l}$. For any $l$, let $\SO_l=\{g\in\SL_l:{}^tgJ_lg=J_l\}$ and fix
$B_{\SO_{l}}=\SO_{l}\cap B_{\GL_l}$.
Let $\Spin_l$ be the algebraic double cover of $\SO_l$, with the Borel subgroup which is the preimage of $B_{\SO_l}$. This defines the set of simple roots $\alpha_0,\ldots,\alpha_{\lfloor l/2\rfloor-1}$ where
$\alpha_i=\epsilon_i-\epsilon_{i+1}$ for $0\leq i<\lfloor l/2\rfloor-1$, and $\GSpin_l$ can be defined as the Levi subgroup of
$\Spin_{l+2}$ obtained by removing $\alpha_0$. For $l=0,1$, $\GSpin_l=\GL_1$, and $\GSpin_2=\GL_1\times\GL_1$.

Henceforth we fix one of the families of groups $\GL_l$, $\Sp_{l}$ (when $l$ is even), $\SO_l$ or $\GSpin_l$, and for a given $l$ denote the member by $\mathcal{G}_l$, e.g., $\mathcal{G}_l=\Sp_l$. Write the Borel subgroup in the
form $B_{\mathcal{G}_l}=T_{\mathcal{G}_l}\ltimes N_{\mathcal{G}_l}$, where $N_{\mathcal{G}_l}$ is the unipotent radical.
For a parabolic subgroup $R<\mathcal{G}_l$, $\delta_R$ denotes its modulus character, and we write $R=M_R\ltimes U_R$ where $M_R$ is the Levi part and $U_R$ is the unipotent radical. If $U<\mathcal{G}_l$ is a unipotent subgroup,
$U^-$ denotes the opposite subgroup. The Weyl group of $\mathcal{G}_l$ is denoted $W(\mathcal{G}_l)$, and similar notation is used for any reductive group. The center of an algebraic group $X$ is denoted $C_X$, and its connected component by $C_X^{\circ}$.

The unipotent subgroups of $\GSpin_l$ are isomorphic (as algebraic groups) to the unipotent subgroups of $\SO_l$, and $W(\GSpin_l)$ is isomorphic to $W(\SO_l)$. Also $C_{\GSpin_{2l+1}}$ is connected and for $l>2$,
$C_{\GSpin_{l}}^{\circ}\cong\GL_1$.

Let $F$ be a local field with characteristic $0$. Throughout, we identify $F$-groups with their $F$-points, e.g.,  $\mathcal{G}_l=\mathcal{G}_l(F)$. The additive group of $l\times l'$ matrices (over $F$) is denoted $\Mat_{l\times l'}$ and $\Mat_l=\Mat_{l\times l}$.
The trace map is denoted $\tr$. If $F$ is non-archimedean, we let $q$ denote the cardinality of its residue field. When we say that a property holds outside a discrete subset of $s$, over a non-archimedean field we mean for all but finitely many values of $q^{-s}$.
For any group $X$, $x,y\in X$ and $Y<X$, ${}^xy=xyx^{-1}$ and ${}^xY=\{{}^xy:y\in Y\}$.

\subsection{Representations}\label{representations}
We describe the general notation involving representations that appear in this work. In this section $\mathcal{G}_l$ can be replaced with any reductive algebraic group. By a representation of a closed subgroup of $\mathcal{G}_l$ we always mean a smooth representation on a complex vector space.
Over archimedean fields, an admissible representation is understood to be admissible Fr\'{e}chet of moderate growth.
If $\pi$ is a representation of a closed subgroup $Y<\mathcal{G}_l$, $\pi^{\vee}$ is the representation contragredient to $\pi$, and for $x\in\mathcal{G}_l$, ${}^x\pi$ denotes the representation of ${}^xY$ on the same space of $\pi$, with the action given by ${}^x\pi(y)=\pi({}^{x^{-1}}y)$. Parabolic induction is normalized. Morphisms are continuous and induction is smooth, and $\otimes$ is the complete tensor product, over archimedean fields.

In this work supercuspidal representations are not automatically irreducible (or unitary). When the field is non-archimedean, a representation of a group which does not have unipotent subgroups is also (trivially) supercuspidal. By definition, supercuspidal representations only exist over non-archimedean fields.

For a closed unipotent subgroup $U<\mathcal{G}_l$, denote the set of (unitary) characters of $U$ by $\widehat{U}$. Let $\pi$ be a representation of $U$ on a space $\mathcal{V}$. For $\psi\in\widehat{U}$,
let $\mathcal{V}(U,\psi)\subset \mathcal{V}$ be the subspace spanned by the vectors $\pi(u)\xi-\psi(u)\xi$ for all
$u\in U$ and $\xi\in\mathcal{V}$ over non-archimedean fields, and over archimedean fields $\mathcal{V}(U,\psi)$ is the closure of this subspace. The Jacquet module $J_{U,\psi}(\pi)$ is the quotient $\mathcal{V}(U,\psi)\backslash\mathcal{V}$. Assume
$R<\mathcal{G}_l$ is a closed subgroup containing $U$. Denote the normalizer of $U$ in $R$ by $N_R(U)$. If
$\pi$ is a representation of $R$, $J_{U,\psi}(\pi)$ is a representation of the subgroup of $N_R(U)$ which stabilizes $\psi$. We do not twist the action, i.e., we do not multiply by a modulus character. For any $r\in R$, we have an isomorphism ${}^rJ_{U,\psi}(\pi)\cong J_{{}^rU,{}^r\psi}(\pi)$ of representations of ${}^rU$ (use $\xi\mapsto\pi(r)\xi$). In particular if $r\in N_R(U)$, ${}^rJ_{U,\psi}(\pi)\cong J_{U,{}^r\psi}(\pi)$.

Over non-archimedean fields, if $U$ is abelian and $N_R(U)$ acts on $\widehat{U}$ with finitely many orbits, by
\cite[5.9--5.12]{BZ1} if $J_{U,\psi'}(\pi)=0$ when $\psi'$ varies over a complete set of representatives for the nontrivial orbits, $U$ acts trivially on the space of $\pi$, i.e., $\pi=J_{U,1}(\pi)$.

Let $J_{U,\psi}(\pi)^*$ be the algebraic dual of $J_{U,\psi}(\pi)$ over a non-archimedean field, and
the continuous dual over archimedean fields. By definition $\Hom_{U}(\pi,\psi)=J_{U,\psi}(\pi)^*$.

Over archimedean fields we will also need the notion of generalized Jacquet modules. Let $\pi$ be a representation of $\mathcal{G}_l$, and $R=M_R\ltimes U_R<\mathcal{G}_l$ be a parabolic subgroup. Denote the Lie algebra of $U_R$ by $\mathfrak{u}_R$. For any positive integer $i$, we call $\pi/\overline{\mathfrak{u}^i\pi}$ the $i$-th generalized Jacquet module of $\pi$.

\begin{lemma}\label{lem:adm}
If $\pi$ is an admissible finite length representation of $\mathcal{G}_l$,
the $i$-th generalized Jacquet module is an admissible finite length representation of $M_R$.
\end{lemma}
This lemma is proven in the same way as the classical case ($i=1$), see
Wallach \cite[Lemma~4.3.1]{Wal88}.

\begin{lemma}\label{lem:adm 2 discrete}
Assume $\pi$ is an admissible finite length representation of $\mathcal{G}_l$.
The set of central exponents of $\pi/\mathfrak{u}^i\pi$, i.e., the central characters of the irreducible constituents of $\pi/\mathfrak{u}^i\pi$ as a representation of $M_R$, where $i$ varies over the positive integers, belong in a discrete set.
\end{lemma}
\begin{proof}
Let $V$ denote the Harish-Chandra module of $\pi$, i.e., the space of $K$-finite vectors, where $K\subset \mathcal{G}_l$ is a maximal compact subgroup. By \cite[Proposition 2.2]{Cass}, $V$ is dense in $\pi$, and by \cite[Proposition  5.1 and Lemma 5.3]{Cass},  $V/\mathfrak{u} V$ has finitely many central exponents. In other words, for any $X$ in the Lie algebra of the center of  $\mathcal{G}_l$ there exists a polynomial $p$ such that $p(X)$ acts by zero on $V/\mathfrak{u} V$.

Now, $V/\mathfrak{u}^iV$ is filtered by the modules $\mathfrak{u}^jV/\mathfrak{u}^{j+1}V$, $0\leq j<i$, and each of these is a quotient of $\mathfrak{u}^j\otimes V/\mathfrak{u}V$.
When $i$ varies, the set of central exponents of
$V/\mathfrak{u}^iV$ is contained in the set of central exponents of $\mathfrak{u}^j\otimes V/\mathfrak{u}V$, $j\geq0$.
Regarding $\mathfrak{u}^j$, its central exponents can be computed using the adjoint action, and when $j$ varies they belong in a lattice. Since
$V/\mathfrak{u}V$ admits only finitely many central exponents, the central exponents of $V/\mathfrak{u}^iV$ for all $i$ lie in a finite union of lattices.

Finally, note that for any $i$, the set of central exponents of $\pi/\mathfrak{u}^i\pi$ lies in the set of central exponents of $V/\mathfrak{u}^iV$. Indeed, if $p(X)$ acts by zero on $V/\mathfrak{u}^iV$ then it acts by zero on $\pi/\mathfrak{u}^i\pi$, since $V$ is dense in $\pi$.
\end{proof}
\begin{remark}
In particular, the set of central exponents of the $i$-th generalized Jacquet modules of $\pi$, where $i$ varies over the positive integers, belongs in a discrete set.
\end{remark}
Let $\psi$ be a nontrivial additive character of $F$. For $v\in V_{(c^l)}$, write $v=(v_{i,j})_{1\leq i,j\leq l}$ with $v_{i,j}\in\Mat_c$. Denote
\begin{align*}
\psi_{l}(v)=\psi(\sum_{i=1}^{l-1}\tr(v_{i,i+1})).
\end{align*}

For a representation $\pi$ of $\GSpin_l$ which admits a central character, let $\chi_{\pi}$ be the restriction of the central character of $\pi$ to $C_{\GSpin_l}^{\circ}$.

\subsection{Distribution vanishing theorem}\label{sec:KV}

Let a real algebraic group $C$ act on a real algebraic manifold $X$. Let $E$ be a  smooth representation of $C$ in a Fr\'{e}chet space. Assume
the actions of $C$ on $X$ and on $E$ extend to a Lie group $A$, which contains $C$ as a closed normal subgroup.

Let $Z\subset X$ be a closed subset which is a union of finitely-many locally closed $A$-orbits.
For any $\nu \in\Z_{>0}$ and $z\in Z$ let $\Lambda_z^{\nu}$ be the symmetric $\nu$-th power of the conormal space at $z$ to the orbit $Cz$ in $X$. Let $C_z$ denote the stabilizer of $z$ in $C$, and $\delta$ be the ratio of modular functions of $C$ and $C_z$.

Denote the space of $E$-valued distributions on $X$, i.e., functionals on the space of compactly supported smooth $E$-valued functions on $X$, by $\cD'(X,E)$, and let $\cD'_{Z}(X,E)\subset \cD'(X,E)$ denote the subspace of distributions supported on $Z$. For a smooth character $\chi$ of $Z$,
let $\cD'_{Z}(X,E)^{C,\chi}\subset \cD'_{Z}(X,E)$ be the subspace of $(C,\chi)$-equivariant distributions.

The following theorem follows from Theorem~\ref{theorem:convenient vanishing} in the appendix:
\begin{theorem}\label{thm:AG}
Assume that for any $z\in Z$, the set $\{\chi^a|_{C_z}:a\in A\}$ is a union of finitely many locally closed orbits, under
the action of the stabilizer $A_z$ of $z$ in $A$. Suppose also that for any $z\in Z$ and any $\nu\geq 0$,
\begin{align}\label{=KV}
((E\otimes \Lambda_{\nu}\otimes \delta)^*)^{C_z,\chi}=0.
\end{align}
Then $\cD'_{Z}(X,E)^{C,\chi}=0$.
\end{theorem}

\begin{remark}
If $\chi$ is trivial or $A=C$, the theorem already follows from \cite[Theorem 3.15, Cases (i,ii)]{KV}. Note that in both cases $\chi^a=\chi$ for any $a\in A$.
\end{remark}

\begin{remark}
If $A,X$ and the action of $A$ on $X$ are semi-algebraic, the $A$-orbits in $X$ are automatically locally closed. If in addition $C$ is semi-algebraic and $C_z$ is unipotent, the condition \eqref{=KV} is equivalent to $(E^*)^{C_z,\chi}=0$, independent of $\nu$ (see \cite{Sun}).
\end{remark}
In order to check the conditions of the theorem we will need the following lemma.

\begin{lemma}\label{lem:Uhat}
Let $H$ be a real reductive group and $Q<H$ be a parabolic subgroup, with a unipotent radical $U=U_Q$. The set $\widehat{U}$ (the unitary characters of $U$) is a finite union of locally closed $Q$-orbits.
\end{lemma}
\begin{proof}
Let $\mathfrak{u}$ denote the Lie algebra of $U$. There exists a
hyperbolic semi-simple element $S\in H$ such that $\mathfrak{u}$ is the sum of positive eigenspaces of the adjoint action $\mathrm{ad}(S)$. The eigenspace $\mathfrak{u}_1$ corresponding to the smallest positive eigenvalue of $\mathrm{ad}(S)$ is called the first internal Chevalley module of $Q$. Clearly, $\mathfrak{u}_1$ projects onto (and in fact identifies with) the space of characters of $\mathfrak{u}$, which in turn identifies with $\widehat{U}$ by multiplying by $i$ and exponentiation. By \cite[Theorem E']{Ric}, $Q$ has finitely many orbits on $\mathfrak{u}_1$, and each orbit is locally closed since the action is algebraic.
\end{proof}

\subsection{Representations of type $(k,c)$}\label{kc representations}
Let $k$ and $c$ be positive integers.
For a partition $\sigma$ of $kc$, let $V(\sigma)<N_{\GL_{kc}}$ denote the corresponding unipotent subgroup, and
$\widehat{V}(\sigma)_{\mathrm{gen}}$ denote the set of generic characters. If $\sigma'$ is another partition of $kc$, write $\sigma'\succsim\sigma$ if $\sigma'$ is greater than or non-comparable with $\sigma$, with respect to the natural partial ordering on
partitions. See \cite{G2}, \cite[\S~5]{CM} and \cite{Cr} for details on these notions. For convenience, we provide the definition of $V(\sigma)$. Identify $\sigma$ with an $l$-tuple of integers $(a_1,\ldots,a_l)$ such that $a_1 \geq\ldots\geq a_l>0$. Let $p_{\sigma}$ be the $kc$-tuple of integers obtained by arranging the multi-set $\{a_i-2j+1:1\leq i \leq l,\,1\leq j\leq a_i\}$ in decreasing order. For any $x\in F^*$, put $x^{p_{\sigma}}=\diag(x^{p_{\sigma}(1)},\ldots,x^{p_{\sigma}(kc)})\in T_{\GL_{kc}}$. The one-parameter subgroup
$\{x^{p_{\sigma}}:x\in F^*\}$ acts on the Lie algebra of $N_{\GL_{kc}}$ by conjugation, and $V(\sigma)$ is the subgroup generated by the weight subspaces of weight at least $2$.

For the orbit $(k^c)$, $V((k^c))=V_{(c^k)}$, the group $M_{(c^k)}$ acts transitively on the set
$\widehat{V}((k^c))_{\mathrm{gen}}$, and $\psi_k\in \widehat{V}((k^c))_{\mathrm{gen}}$. The stabilizer of
$\psi_k$ in $M_{(c^k)}$ is then the diagonal embedding $\GL_c^{\triangle}$ of $\GL_c$ in $M_{(c^k)}$.

Let $\rho$ be a representation of $\GL_{kc}$. We say that $\rho$ is a $(k,c)$ representation if
$\Hom_{V(\sigma)}(\rho,\psi')=0$ for all $\sigma\succsim(k^c)$ and $\psi'\in\widehat{V}(\sigma)_{\mathrm{gen}}$, and
$\dim\Hom_{V_{(c^k)}}(\rho,\psi_{k})=1$.

We briefly recall the definition of the wave-front set (see e.g., \cite[\S~4.1]{GourevitchSahi2019} for some more details).
When $\rho$ is admissible of finite length, its character defines a distribution on a neighborhood of $0$
in the Lie algebra of $\GL_{kc}$. This distribution (in the non-archimedean case) or the leading term of its asymptotic expansion near $0$ (archimedean case) is a combination of Fourier transforms of Haar measures of nilpotent coadjoint orbits (\cite{Howe1974}, \cite[p.~180]{Harish-Chandra1977}, \cite[Theorems~1.1 and 4.1]{BV80}). For a nilpotent orbit $\mathcal{O}$, let
$c_{\mathcal{O}}$ denote its coefficient in this expansion (for a suitable normalization of the measures).
The wave-front set $\mathrm{WF}(\rho)$ of $\rho$ is defined to be the set of orbits $\mathcal{O}$ such that
$c_{\mathcal{O}}\ne0$ and for each orbit $\mathcal{O}'$ containing $\mathcal{O}$ in its closure, $c_{\mathcal{O}'}=0$.

In this case, an equivalent definition of a $(k,c)$ representation can be given in terms of $\mathrm{WF}(\rho)$:
Now $\rho$ is $(k,c)$ if $(k^c)$ is the unique maximal orbit in $\mathrm{WF}(\rho)$ and
the dimension of the space of degenerate Whittaker functionals on $\rho$ with respect to $V_{(c^k)}$ and $\psi_{k}$ is $1$ (see \cite[Theorem~E]{GGS}).

For $c=1$, a representation is $(k,1)$ if and only if it affords a unique Whittaker model. On the other end,
a representation is $(1,c)$ if and only if $\dim\Hom_{V_{(c)}}(\rho,1)=1$, equivalently $\rho$ is a character
($V_{(c)}$ is the trivial group).

For a $(k,c)$ representation $\rho$, $\dim J_{V_{(c^k)},\psi_k}(\rho)^*=1$, hence $\dim J_{V_{(c^k)},\psi_k}(\rho)=1$ so that $\SL_c^{\triangle}$ acts trivially on $J_{V_{(c^k)},\psi_k}(\rho)$ and
$\GL_c^{\triangle}$ acts on $J_{V_{(c^k)},\psi_k}(\rho)$ by a character.

We recall the map $\rho_c$ defined (implicitly) in \cite[\S~2.2]{CFK} from irreducible generic representations of $\GL_k$ to admissible
finite length $(k,c)$ representations of $\GL_{kc}$. For an irreducible tempered representation $\tau$ of $\GL_k$, $\rho_c(\tau)$ is the generalized Speh representation, i.e., the unique irreducible quotient of
$\Ind_{P_{(k^c)}}^{\GL_{kc}}((\tau\otimes\ldots\otimes\tau)\delta_{P_{(k^c)}}^{1/(2k)})$ (see \cite{Jac4,MW4}). Then if $\tau=\Ind_{P_{\beta}}^{\GL_k}(\otimes_{i=1}^d|\det|^{a_i}\tau_i)$ where $\beta$ is a composition of $d$ parts of $k$, $a_1>\ldots>a_d$ and each $\tau_i$ is tempered, $\rho_c(\tau)=\Ind_{P_{\beta c}}^{\GL_{kc}}(\otimes_{i=1}^d|\det|^{a_i}\rho_c(\tau_i))$. By \cite[Theorem~5]{CFK} the representation $\rho_c(\tau)$ is $(k,c)$. The definition of $\rho_c(\tau)$ was also extended to unramified principal series
$\Ind_{B_{\GL_k}}^{\GL_k}(\otimes_{i=1}^k|\det|^{a_i}\tau_i)$, where $\tau_i$ are unramified unitary quasi-characters of $F^*$ and $a_1\geq\ldots\geq a_k$, again by letting $\rho_c(\tau)=\Ind_{P_{(c^k)}}^{\GL_{kc}}(\otimes_{i=1}^k|\det|^{a_i}\rho_c(\tau_i))$ (note that
$\rho_c(\tau_i)=\tau\circ\det_{\GL_c}$). While $\rho_c(\tau)$ might be reducible in the general case, it is still admissible, of finite length and admits a central character.
Also note that (over any local field) $\GL_c^{\triangle}$ acts on
$J_{V_{(c^k)},\psi_k}(\rho_c(\tau))$ by $g\mapsto\tau((\det g)I_k)$ (\cite[Lemma~14]{CFK}).

We mention that over non-archimedean fields, certain structural properties of irreducible $(k,c)$ representations
follow from \cite[\S~II.2]{MW3}. For principal series representations, irreducible or not, over any local field, a representation is $(k,c)$ if and only if it takes the form $\Ind_{P_{(c^k)}}^{\GL_{kc}}(\otimes_{i=1}^k\chi_i\det_{\GL_c})$ for quasi-characters $\chi_i$ of $F^*$. This follows from
\cite{AGS2015a,AGS2015b,GGS} (their focus was archimedean; the non-archimedean case essentially follows from \cite{BZ1,MW3}).
\subsection{Doubling setup}\label{Doubling setup}
We define the basic setup for the doubling method: the groups $G$ and $H$, the image of $G\times G$ in $H$, and the definition of the local integral. The precise details depend on $G$.

Let $c,k\geq1$ be integers, $G=\mathcal{G}_c$ and $H=\mathcal{G}_{2kc}$ (if $G=\Sp_c$, $c$ must be even). Let $n=\lfloor c/2\rfloor$ if $G\ne\GL_c$, otherwise
$n=c$. Also set $\epsilon_0=-1$ for $G=\Sp_{c}$ and $\epsilon_0=1$ otherwise, and if $G=\SO_c,\GSpin_c$ and $c$ is odd, define
$(\epsilon_1,\epsilon_2)=(1,1/2)$ if $k$ is even and $(\epsilon_1,\epsilon_2)=(1/2,1)$ if $k$ is odd.

Recall $B_H=T_H\ltimes N_H$ is our fixed Borel subgroup in $H$ (see \S~\ref{the groups}).
Set $H_0=\mathcal{G}_{2c}$. Let $Q=M_Q\ltimes U_Q$ be the standard parabolic subgroup of $H$ such that its Levi part $M_Q$ is isomorphic to $\GL_c\times\ldots\times\GL_c\times H_0$ if $H\ne\GL_{2kc}$,
otherwise $Q=P_{(c^{k-1},2c,c^{k-1})}$. Denote $U=U_Q$. We construct the following character $\psi_U$ of $U$.

For $k>1$, denote the middle $4c\times 4c$ block of an element in $U$ by
\begin{align}\label{matrix:middle 4c block of u}
\left(\begin{smallmatrix}I_c&u&v\\&I_{2c}&u'\\&&I_c\end{smallmatrix}\right).
\end{align}
Let $u^{1,1}$ be the top left $n\times n$ block of $u$, and if $H\ne\GL_{2kc}$, denote the bottom right $n\times n$ block of $u$ by $u^{2,2}$.
For $H=\GL_{2kc}$, $u^{2,2}$ is defined to be the top $c\times c$ block of $u'$. If $H=\SO_{2kc},\GSpin_{2kc}$ and $c$ is odd, denote the middle two coordinates of row $n+1$ of $u$ by $(u^3,u^4)\in\Mat_{1\times2}$.

If $H\ne\GL_{2kc}$ and $k>1$, the character $\psi_U$ restricts to $\psi_{k-1}$ on the group $V_{(c^{k-1})}$, identified with a subgroup of $U$ via the embedding $v\mapsto\diag(v,I_{2c},v^*)\in U$. For $H=\GL_{2kc}$ and $k>1$, $\psi_U$ restricts to $\psi_{k-1}^{-1}$ on each of the two copies of $V_{(c^{k-1})}$, embedded in $U$ via $(v_1,v_2)\mapsto\diag(v_1,I_{2c},v_2)$ ($v_1,v_2\in V_{(c^{k-1})}$).
The character $\psi_U$ is given on \eqref{matrix:middle 4c block of u} by
\begin{align*}
\begin{cases}
\psi(\tr(-u^{1,1}+u^{2,2}))&H=\GL_{2kc},\\
\psi(\tr(u^{1,1}+u^{2,2}))&H=\Sp_{2kc},\SO_{2kc},\GSpin_{2kc},\text{ even $c$},\\
\psi(\tr(u^{1,1}+u^{2,2})+\epsilon_1u^3-\epsilon_2u^4)&H=\SO_{2kc},\GSpin_{2kc},\text{ odd $c$}.
\end{cases}
\end{align*}
For $k=1$, $U$ and thereby $\psi_U$ are trivial.

Now consider the case $H\ne\GSpin_{2kc}$. In this case $G\times G$ is embedded in the stabilizer of $\psi_U$ in $M_Q$.
Explicitly, assume $k\geq1$ and $g_1,g_2\in G$. If $H=\Sp_{2kc},\SO_{2kc}$ with an even $c$, write
$g_1=\left(\begin{smallmatrix}g_{1,1}&g_{1,2}\\g_{1,3}&g_{1,4}\end{smallmatrix}\right)$, $g_{1,i}\in \Mat_n$, then
\begin{align*}
(g_1,g_2)=
\diag(g_1,\ldots,g_1,\left(\begin{smallmatrix} g_{1,1}&&g_{1,2}\\ &g_2&\\ g_{1,3}&&g_{1,4}\end{smallmatrix}\right),g_1^*,\ldots,g_1^*),
\end{align*}
where $g_1^*$ appears $k-1$ times. For $H=\GL_{2kc}$,
\begin{align*}
(g_1,g_2)=
\diag(g_1,\ldots,g_1,g_1,g_2,g_1,\ldots,g_1).
\end{align*}
Here $g_1$ appears $k$ times on the left of $g_2$ and $k-1$ on the right.

For odd $c$ and $H=\SO_{2kc}$, take column vectors $e_{\pm i}$, $1\leq i\leq c$, whose Gram matrix is $J_{2c}$ (i.e., ${}^te_{i}e_{-j}=\delta_{i,j}$). Let
\begin{align*}
&b=(e_1,\ldots,e_{c-1},\epsilon_1e_{c}-\epsilon_2e_{-c},\epsilon_1e_{c}+\epsilon_2e_{-c},e_{-c+1},\ldots,e_{-1}),\\
&b_1=(e_1,\ldots,e_{n},\epsilon_1e_{c}-\epsilon_2e_{-c},e_{-n},\ldots,e_{-1}),\\
&b_2=(e_{n+1},\ldots,e_{c-1},\epsilon_1e_{c}+\epsilon_2e_{-c},e_{-c+1},\ldots,e_{-n-1}),\\
&m=\diag(I_{c-1},\left(\begin{smallmatrix}\epsilon_1&\epsilon_1\\-\epsilon_2&\epsilon_2\end{smallmatrix}\right),I_{c-1}).
\end{align*}
The Gram matrices of $(b,b_1,b_2)$ are $(J_{2c},\diag(I_n,-1,I_n)J_{c},J_{c})$. The left (resp., right) copy of $\SO_c$ acts
on the subspace spanned by $b_1$ (resp., $b_2$); the left copy is defined by
\begin{align*}
\{g_1\in\SL_c:{}^tg_1\diag(I_n,-1,I_n)J_{c}g_1=\diag(I_n,-1,I_n)J_{c}\},
\end{align*}
and the right copy is defined using the convention of \S~\ref{the groups} (the Gram matrix of $b_2$ is $J_c$).
Extend $g_i$ by letting it fix the vectors of $b_{3-i}$, then write this extension as a matrix $g_i'\in\SO_{2c}$ with
respect to $b$, $i=1,2$. The matrices ${}^mg_1'$ and ${}^mg_2'$ commute and the embedding is given by
\begin{align*}
(g_1,g_2)=\diag(g_1,\ldots,g_1,{}^mg_1'\,{}^mg_2',g_1^*,\ldots,g_1^*).
\end{align*}

The notation $(1,g)$ or $(g,1)$ is used for the embedding of one of the copies of $G$ in $H$, where $1$ denotes the identity element of $G$.

\begin{example}\label{example:odd c}
Here are a few examples for the embedding in the odd orthogonal case, adapted from \cite[Example~15]{CFK}.
Consider the standard Siegel parabolic subgroup $R$ of $G$. For $a,b\in\GL_n\cong M_R$,
\begin{align*}
(a,b)=\diag(\diag(a,1,a^*)^{\triangle'},\diag(a,b,I_2,b^*,a^*)),
\end{align*}
where $\triangle'$ denotes the diagonal embedding of $\GL_c$ in $\GL_{(k-1)c}$, and we omitted, here and below, the bottom right $(k-1)c\times(k-1)c$ block
of $(a,b)$ (it is uniquely defined by the given blocks and $H$).
The images of $(U_R,1)$ and $(1,U_R)$ take the form
\begin{align*}
&\diag(\left(\begin{smallmatrix}I_n&x&y\\&1&x'\\&&I_n\end{smallmatrix}\right)^{\triangle'},
\left(\begin{smallmatrix}I_n&&\epsilon_2x&-\epsilon_1x&&y\\&I_{n}\\&&1&&&\epsilon_1x'\\&&&1&&-\epsilon_2x'\\&&&&I_n\\&&&&&I_n\end{smallmatrix}\right)),\\
&\diag(I_{(k-1)c},
\left(\begin{smallmatrix}I_n\\&I_n&\epsilon_2x&\epsilon_1x&y\\&&1&&\epsilon_1x'\\&&&1&\epsilon_2x'\\&&&&I_n\\&&&&&I_n\end{smallmatrix}\right)),
\end{align*}
where $x'$ is uniquely determined given $x$ and $H$. We also note that
\begin{align*}
&(\left(\begin{smallmatrix}I_{n-1}\\&&&1\\&&-1\\&1\\&&&&I_{n-1}\end{smallmatrix}\right),1)=
\diag(\left(\begin{smallmatrix}I_{n-1}\\&&&1\\&&-1\\&1\\&&&&I_{n-1}\end{smallmatrix}\right)^{\triangle'},
\left(\begin{smallmatrix}I_{n-1}\\&&&&&&1\\&&I_n\\&&&&2\epsilon_1^2\\&&&2\epsilon_2^2\\&&&&&I_n\\&1\\&&&&&&&I_{n-1}\end{smallmatrix}\right)),\\
&(1,\left(\begin{smallmatrix}I_{n-1}\\&&&1\\&&-1\\&1\\&&&&I_{n-1}\end{smallmatrix}\right))=
\diag(I_{(k-1)c},
\left(\begin{smallmatrix}I_{2n-1}\\&&&&1\\&&&-2\epsilon_1^2\\&&-2\epsilon_2^2\\&1&&\\&&&&&I_{2n-1}\end{smallmatrix}\right)).
\end{align*}
\end{example}
The case of $H=\GSpin_{2kc}$ is slightly more complicated: we have an embedding of
\begin{align}\label{embedding G G in H GSpin}
\{(z,z):z\in C_G^{\circ}\}\backslash G\times G
\end{align}
in $M_Q$, in the stabilizer of $\psi_U$. Here with a minor abuse of notation $(z,z)$ is regarded as an element of $G\times G$. For details see \cite[\S~3.5]{CFK}.

We define the space of the induced representation of $H$, which is used for the construction of the integral.
First assume $H=\Sp_{2kc},\SO_{2kc}$. Let $P=M_P\ltimes U_P$ be a standard maximal parabolic subgroup of $H$ such that $M_P\cong\GL_{kc}$
and $M_P<M_{(kc,kc)}$. Let $\rho$ be a representation of $\GL_{kc}$. For a complex parameter $s$, let $V(s,\rho)$ be the space of $\Ind_{P}^{H}(|\det|^{s-1/2}\rho)$. For $H=\GSpin_{2kc}$ we take the standard parabolic subgroup $P$ obtained by removing the simple root $\alpha_{kc}$, then $M_P\cong\GL_{kc}\times\GL_1$ and note that $\GL_1$ is identified with $C_H^{\circ}$. Let $\rho$ be as above, and $\eta$ be a quasi-character of $F^*$. Then
$V(s,\rho\otimes\eta)$ is the space of the induced representation $\Ind_{P}^{H}(|\det|^{s-1/2}\rho\otimes\eta)$.
For $H=\GL_{2kc}$ we take $P=P_{(kc,kc)}$, $\rho=\rho_1\otimes\rho_2$ for two representations $\rho_1$ and $\rho_2$ of $\GL_{kc}$,
and $V(s,\rho)$ denotes the space of $\Ind_{P}^{H}(|\det|^{s-1/2}\rho_1\otimes|\det|^{-s+1/2}\rho_2)$.

Assume $H\ne\GSpin_{2kc}$.
Take $\delta_0\in H$ satisfying ${}^{\delta_0}U_P=U_P^-$ if $kc$ is even or $H=\GL_{2kc}$; otherwise take
$\delta_0'\in\Orth_{2kc}$ with ${}^{\delta_0'}U_P=U_P^-$ and let $\delta_0$ be the product of $\delta_0'$ and a representative of the transposition in $\Orth_{2kc}$ which normalizes $N_H$ (to obtain $\det\delta_0=1$). Let $\delta_1\in H_0\cap U_P$ ($H_0<M_Q$) be such that its natural identification with a matrix in $\Mat_c$ is of rank $c$, unless $H=\SO_{2kc}$ and $c$ is odd, in which case the rank is $c-1$ (this is the maximal rank). Put $\delta=\delta_0\delta_1$.
Then let $\iota$ be an involution of $G$ such that ${}^{\delta}\{(g,{}^{\iota}g):g\in G\}<M_P$.

One concrete choice of $\delta_0,\delta_1$ and $\iota$ was given in \cite{CFK}:
\begin{align*}
&\delta_0=\begin{dcases}\left(\begin{smallmatrix} &I_{kc}\\ \epsilon_0 I_{kc}\end{smallmatrix}\right)&\text{$H\ne\SO_{2kc}$ or $c=2n$},\\
\left(\begin{smallmatrix} &I_{kc}\\ I_{kc}\end{smallmatrix}\right)
\diag(I_{(k-1)c},\left(\begin{smallmatrix}I_n\\&&(-1)^{k}\\&I_n\end{smallmatrix}\right),
\left(\begin{smallmatrix}&I_n\\(-1)^{k}&&\\&&I_n\end{smallmatrix}\right),I_{(k-1)c})\jmath_{kc}&\text{$H=\SO_{2kc}$, $c=2n+1$},
\end{dcases}
\end{align*}
where $\jmath_{kc}=\diag(I_{kc-1},\left(\begin{smallmatrix} &1\\ 1\end{smallmatrix}\right)^{kc},I_{kc-1})$,
$\delta_1=\diag(I_{(k-1)c},\left(\begin{smallmatrix}I_c &A \\ &I_c\end{smallmatrix}\right),I_{(k-1)c})$ with
\begin{align*}
A=\begin{dcases}
I_c&H=\Sp_{2kc},\GL_{2kc},\\
\left(\begin{smallmatrix}-I_{n}\\&I_{n}\end{smallmatrix}\right) &\text{$H=\SO_{2kc}$, $c=2n$},\\
\left(\begin{smallmatrix}&-I_{n}\\&&I_{n}\\0\end{smallmatrix}\right) &\text{$H=\SO_{2kc}$, $c=2n+1$},\\
\end{dcases}
\end{align*}
and
\begin{align*}
\iota=\begin{dcases}
\left(\begin{smallmatrix}&I_{n}\\-\epsilon_0I_{n}\end{smallmatrix}\right)&\text{$H=\Sp_{2kc},\SO_{2kc}$, $c=2n$},\\
I_c&H=\GL_{2kc},\\
\left(\begin{smallmatrix}&&I_{n}\\&I_2\\I_{n}\end{smallmatrix}\right)&\text{$H=\SO_{2kc}$, $c=2n+1$, $k$ is odd},\\
\left(\begin{smallmatrix}&&I_{n}\\&\begin{smallmatrix}&-2\epsilon_1^2\\-2\epsilon_2^2\end{smallmatrix}\\I_{n}\end{smallmatrix}\right)&
\text{$H=\SO_{2kc}$, $c=2n+1$, $k$ is even}.
\end{dcases}
\end{align*}
Also set $U_0=U\cap {}^{\jmath_{kc}}U_P$, then $\psi_U$ is a character of $U_0$ by restriction.

For $H=\GSpin_{2kc}$, $\delta_0\in H$ is defined using the isomorphism $W(H)\cong W(\SO_{2kc})$, $\delta_1$ is the element taken for $\SO_{2kc}$, and $\iota$ satisfies the same condition as above (the concrete examples of $\iota$ extend to involutions of $\GSpin_c$ as well). The important observation for us here concerning $\GSpin_{2kc}$, is that the images of unipotent subgroups, and subgroups $\GL_l$ occurring as direct factors of standard Levi subgroups, can be read off the corresponding orthogonal cases.

For any representation $\pi$ of $G$, $\pi^{\iota}$ is the representation on the space of $\pi$,
with the action defined by $\pi^{\iota}(g)=\pi({}^{\iota}g)$. The definitions imply
$(\pi^{\iota})^{\vee}=(\pi^{\vee})^{\iota}$.

We define the (local) doubling integral.
Let $\pi$ be an irreducible admissible representation of $G$.
If $H\ne\GL_{2kc}$, let $\rho$ be an admissible finite length $(k,c)$ representation of $\GL_{kc}$ which admits a central character.
Otherwise $\rho=\rho_1\otimes\chi^{-1}\rho_2$ where $\rho_1$ and $\rho_2$ are
admissible finite length $(k,c)$ representations of $\GL_{kc}$, each admitting a central character, and such that the central character of $\rho_1$ is the inverse of the central character of $\rho_2$, and $\chi$ is a quasi-character of $F^*$.

Let $\omega$ be a matrix coefficient of $\pi^{\vee}$.
Let $f$ be a holomorphic section of $V(s,\rho)$ if $H\ne\GSpin_{2kc}$, and for $\GSpin_{2kc}$, $f$ is a holomorphic section of $V(s,\rho\otimes\chi_{\pi})$. The doubling integral for $\pi\times\rho$ is defined by
\begin{align}\label{local integral}
Z(s,\omega,f)=\int\limits_{G}\int\limits_{U_0}
\omega(g)f(s,\delta u_0(1,{}^{\iota}g))\,\psi_U(u_0)\,du_0\,dg.
\end{align}
Here if $H=\GSpin_{2kc}$, the domain of integration is $C_G^{\circ}\backslash G$ instead of $G$.

\begin{theorem}\label{theorem:basic props of doubling integrals}\cite[Propositions~17, 20, 21]{CFK}
Integral \eqref{local integral} enjoys the following properties.
\begin{enumerate}[leftmargin=*]
  \item\label{props 1}Formally, it belongs to the space
  \begin{align}\label{eq:doubling Hom}
  \Hom_{(G,G)}(J_{U,\psi_U^{-1}}(V(s,\rho\otimes\chi_{\pi})),\chi^{-k}\pi^{\vee}\otimes\pi^{\iota}).
  \end{align}
  Here $\chi_{\pi}$ and $\chi$ are omitted for the cases where they are undefined.
  \item\label{props 2}It is absolutely convergent for $\Real(s)\gg0$, independent of the data $(\omega,f)$.
  \item\label{props 3}Over non-archimedean fields there is data $(\omega,f)$, where $f$ is a polynomial section in $q^{\mp s}$, such that $Z(s,\omega,f)$ is absolutely convergent in $\C$ and equals a nonzero constant (independent of $s$). Over archimedean fields for each $s$ there is $\omega$ and a smooth section $f$ such that the integral is nonzero at $s$.
\end{enumerate}
\end{theorem}
\begin{proof}
The theorem was proved in \textit{loc. cit.}, for the representation $\rho=\rho_c(\tau)$ ($H\ne\GL_{2kc}$) or $\rho=\rho_c(\tau)\otimes\chi^{-1}\rho_c(\tau^{\vee})$ ($\rho_c(\tau)$ was defined in \S~\ref{kc representations}). However, the proofs of these statements remain valid when we take the more general representation $\rho$ as described above.
\end{proof}

Over non-archimedean fields, once we prove that \eqref{eq:doubling Hom} is at most one-dimensional outside a discrete subset of $s$, Theorem~\ref{theorem:basic props of doubling integrals} together with Bernstein's continuation principle (in \cite{Banks}) imply that for a rational section $f$, $Z(s,\omega,f)$ admits meromorphic continuation to a rational function in $q^{-s}$.
Over archimedean fields for the choice of $\rho$ in \cite{CFK}, the meromorphic continuation of the integral and continuity of this continuation in the input data were proved in \cite[\S~6.13]{CFK}.

\section{Uniqueness results}\label{Uniqueness results}
\subsection{Outline of the proof of Theorem~\ref{theorem A}}\label{outline}
Let $\pi_1$ and $\pi_2$ be admissible finite length representations of
$G$. If $H\ne\GL_{2kc}$, let $\rho$ be an admissible finite length $(k,c)$ representation of $\GL_{kc}$. For $H=\GL_{2kc}$ put
$\rho=\rho_1\otimes\rho_2$ where each $\rho_i$ is an admissible finite length $(k,c)$ representation of $\GL_{kc}$, and let $\chi_0$ be the quasi-character of $F^*$ such that the diagonal action of $\GL_c^{\triangle}$ on $J_{V_{(c^k)},\psi_k}(\rho_1)\otimes J_{V_{(c^k)},\psi_k}(\rho_2)$ is given by $g\mapsto\chi_0(\det g)$ ($g\in \GL_c$).
If $H=\GSpin_{2kc}$, assume in addition that $\chi_{\pi_1},\chi_{\pi_2}$ exist and $\chi_{\pi_1}^{-1}=\chi_{\pi_2}$, and put $\eta=\chi_{\pi_1}^{-1}$.
To preserve uniform notation the characters $\chi_0,\chi_{\pi_i}$ and $\eta$ are simply ignored in all other cases.

Let $D=U\rtimes (G,G)<H$. We will prove our main result by analyzing distributions on the orbits of the right action of $D$
on the homogeneous space $P\backslash H$. The space $P\backslash H/D$ is finite if $k=1$ (see \cite[Lemma~2.1]{PSR}) or $c=1$ (then either $n=0$ and $U=N_H$, or $G=\GL_1$ and $U$ contains all the roots of $N_H$ but one, on which $(G,G)$ acts with $2$ orbits). Otherwise it is infinite, even uncountable (e.g., for $H=\Sp_{2kc}$ and $k>2$), but contains a unique Zariski open orbit which is $P\delta D$. This follows by showing that the dimension of $P\delta D$ is equal to the dimension of $PU_P^-$. For $h,h'\in H$, write $h\sim h'$ if $PhD=Ph'D$, otherwise $h\not\sim h'$.

Regard $\psi_U\otimes \pi_1^{\vee}\otimes\pi_2^{\vee}$ as a representation of $D$. For $H=\GSpin_{2kc}$,
$(G,G)$ is a homomorphic image of $G\times G$ (see \eqref{embedding G G in H GSpin}), and the condition $\chi_{\pi_1}^{-1}=\chi_{\pi_2}$ above implies that $\pi_1^{\vee}\otimes\pi_2^{\vee}$ is a representation of $(G,G)$.

Consider the space
\begin{align}\label{eq:Hom L 0}
\Hom_{(G,G)}(J_{U,\psi_U^{-1}}(V(s,\rho\otimes\eta)),\pi_1\otimes\pi_2)\cong
\Hom_{D}(V(s,\rho\otimes\eta)\otimes\psi_U\otimes \pi_1^{\vee}\otimes\pi_2^{\vee},1),
\end{align}
which is isomorphic to
\begin{align}\label{eq:Hom L}
\Hom_{D}(\Ind_{P\times D}^{H\times D}\left((|\det|^{s-1/2}\rho\otimes\eta)\otimes(\psi_U\otimes \pi_1^{\vee}\otimes\pi_2^{\vee})\right),1).
\end{align}
Here the action of $D$ on the space of functions $\xi$ on $H\times D$ is given by $d\cdot\xi(h',d')=\xi(h'd,d'd)$; and if $H=\GL_{2kc}$, $|\det|^{s-1/2}\rho$ is short for $|\det|^{s-1/2}\rho_1\otimes|\det|^{-s+1/2}\rho_2$.

For any $h\in H$, denote $P_{h}={}^{h^{-1}}P\cap D$. We will study \eqref{eq:Hom L} by considering the following spaces of distributions
on the orbits $PhD$ (this is well defined, see below):
\begin{align}\label{space:Hom orbit}
\Hom_{D}(\ind_{P_{h}}^D\left({}^{h^{-1}}((|\det|^{s-1/2}\rho\otimes\eta)\delta_{P}^{1/2})\otimes(\psi_U\otimes \pi_1^{\vee}\otimes\pi_2^{\vee}\otimes\Lambda_{\nu})\right),1).
\end{align}
Here over non-archimedean fields $\ind$ denotes the compact non-normalized induction, while for archimedean fields $\ind$ is the Schwartz induction of \cite[\S~2]{duCloux1991} (see also \cite[\S~2.3]{GGS2}); and $\Lambda_0$ is the trivial character. If the field is non-archimedean
or $h\sim\delta$, we only have $\nu=0$. Over archimedean fields when $h\not\sim\delta$, we further have for each integer $\nu>0$, a finite dimensional algebraic representation $\Lambda_{\nu}$ which is the algebraic dual of the symmetric $\nu$-th power of the normal bundle to the double coset. Note that when $h\sim\delta$, i.e., for the open orbit $P\delta D$, the tangent space to the double coset coincides with the total tangent space, and thus the normal space is trivial.

By the Frobenius reciprocity \eqref{space:Hom orbit} is isomorphic to
\begin{align}\label{H(h)}
\mathcal{H}_{\nu}(h)=\Hom_{P_{h}}({}^{h^{-1}}(|\det|^{s-1/2}\rho\otimes\eta)\otimes(\psi_U\otimes \pi_1^{\vee}\otimes\pi_2^{\vee})\otimes\Lambda_{\nu},\theta_h).
\end{align}
Here $\theta_h(x)=\delta_{P_h}(x)\delta_D^{-1}(x)\delta_P^{-1/2}({}^hx)$ ($x\in P_h$).
We define $\mathcal{H}(h)=\mathcal{H}_0(h)$ if $F$ is non-archimedean or $h\sim\delta$, otherwise
$\mathcal{H}(h)=\oplus_{\nu}\mathcal{H}_{\nu}(h)$.

Our main result --- Theorem~\ref{theorem:uniqueness} below, is that \eqref{eq:Hom L 0} is at most one-dimensional outside a discrete subset of $s$. We will prove there is a discrete subset $\mathcal{B}\subset\C$, such that for all $s\notin\mathcal{B}$, $\mathcal{H}(h)=0$ for all $h\not\sim\delta$, and $\dim \mathcal{H}(\delta)\leq1$.

Over non-archimedean fields this already implies \eqref{eq:Hom L 0} is at most one-dimensional outside $\mathcal{B}$.
Indeed this follows from the theory of distributions on $l$-sheafs of \cite{BZ1}. In more detail,
let $\mathcal{F}$ be the $l$-sheaf of the induced representation in \eqref{eq:Hom L}. The right action of $D$ on
$P\backslash H$ is constructive, by \cite[Theorem~A]{BZ1} applied to $X(F)$ where $X$ is the algebraic $F$-variety $P\backslash H$. Each $\mathcal{H}(h)$ (see \eqref{space:Hom orbit}) is the
space of distributions on the restriction of $\mathcal{F}$ to the orbit $PhD$ (the orbits are locally closed, hence this restriction is well defined). Fix $s\notin\mathcal{B}$ and
let $\mathcal{T},\mathcal{T}'$ be nonzero distributions in \eqref{eq:Hom L}. Since $P\delta D$ is open,
by \cite[1.16]{BZ1} both $\mathcal{T}$ and $\mathcal{T}'$ restrict to distributions on $\mathcal{H}(\delta)$, which is one-dimensional, hence
there is $\alpha\in\C$ such that $\alpha\mathcal{T}|_{P\delta D}=\mathcal{T}'|_{P\delta D}$. Then
$\alpha\mathcal{T}-\mathcal{T}'$ is well defined on the quotient $l$-sheaf
$\mathcal{F}(P\delta D)\backslash\mathcal{F}$ (see \cite[1.16]{BZ1} for the definition and notation), which is an $l$-sheaf on
the complement of $P\delta D$ in $H$. Since there are no nonzero distributions on any $\mathcal{H}(h)$ for $h\not\sim\delta$, by \cite[Theorem~6.9]{BZ1} we deduce $\alpha\mathcal{T}-\mathcal{T}'$ vanishes on
$\mathcal{F}$, i.e., $\alpha\mathcal{T}=\mathcal{T}'$.

Over archimedean fields the argument also depends on the precise methods we use in order to handle
each $\mathcal{H}(h)$. We describe this below.
\subsubsection{Basic properties of $\mathcal{H}(h)$}\label{Basic properties}
In general every algebraic representation of a unipotent group is unipotent, i.e., admits a (finite) filtration such that the group acts trivially on each of its quotients. We can hence filter each $\Lambda_{\nu}$ and consider these quotients. If \eqref{H(h)} is nonzero, it is nonzero when $\Lambda_{\nu}$ is replaced by one of these quotients. Since we will prove
$\mathcal{H}_{\nu}(h)=0$ for all $\nu>0$, we can consider each of these quotients, re-denoted $\Lambda_{\nu}$ (at the cost of re-enumerating the index set of $\nu$), separately, so that we assume $\Lambda_{\nu}$ is a trivial representation of $U$ for all $\nu\geq0$.

In general if $Y<{}^hU\cap M_P$, then ${}^{h^{-1}}Y<P_h$ and
by definition any morphism in $\mathcal{H}(h)$ factors through $J_{Y,{}^{h}\psi_U^{-1}}(\rho)$. Indeed since
${}^{h^{-1}}Y<U$, for $y\in Y$ we have
\begin{align*}
{}^{h^{-1}}(|\det|^{s-1/2}\rho\otimes\eta)({}^{h^{-1}}y)=\rho(y),\qquad
(\psi_U\otimes \pi_1^{\vee}\otimes\pi_2^{\vee}\otimes\Lambda_{\nu})({}^{h^{-1}}y)=\psi_U({}^{h^{-1}}y),
\end{align*}
then if $\mathcal{T}\in\mathcal{H}_{\nu}(h)$ for some $\nu$, and $\xi_{\rho}\otimes\xi$ is a pure tensor in the space of
$\rho\otimes(\pi_1^{\vee}\otimes\pi_2^{\vee}\otimes\Lambda_{\nu})$,
\begin{align*}
\psi_U({}^{h^{-1}}y)\mathcal{T}(\rho(y)\xi_{\rho}\otimes\xi)=
\mathcal{T}(\psi_U({}^{h^{-1}}y)\rho(y)\xi_{\rho}\otimes\xi)=\mathcal{T}(\xi_{\rho}\otimes\xi).
\end{align*}
Thus
\begin{align}\label{eq:T on Jacquet}
\mathcal{T}((\rho(y)\xi_{\rho}-{}^{h}\psi_U^{-1}(y)\xi_{\rho})\otimes\xi)=0.
\end{align}
This means that $\mathcal{T}$ factors through $J_{Y,{}^{h}\psi_U^{-1}}(\rho)$, where in the archimedean case note that $\mathcal{T}$ is continuous, and because the argument is applicable to all $\nu$, we conclude that any morphism in $\mathcal{H}(h)$ factors through $J_{Y,{}^{h}\psi_U^{-1}}(\rho)$.

\subsubsection{The vanishing of $\mathcal{H}(h)$}\label{the vanishing 3 types}
One can prove the vanishing of $\mathcal{H}(h)$ using three types of arguments. First we have an incompatibility condition: assume
$h$ is such that
\begin{align}
\label{psi U nontrivial}
&\psi_U|_{U\cap {}^{h^{-1}}U_P}\ne1.
\end{align}
In this case we can take a subgroup $Y<U$ such that ${}^hY<U_P$ and $\psi_U|_Y\ne1$.
Then $Y<P_h$ and both ${}^{h^{-1}}(|\det|^{s-1/2}\rho\otimes\eta)$ and $\pi_1^{\vee}\otimes\pi_2^{\vee}\otimes\Lambda_{\nu}$ are trivial on
$Y$ (because ${}^hY<U_P$ and $Y<U$), hence the action on the left hand side in $\mathcal{H}_{\nu}(h)$ is given by
$\psi_U$ which is nontrivial by \eqref{psi U nontrivial}. However, the action on the right hand side is trivial, because it is given by a modulus character and $Y<U$. Thus
$\mathcal{H}_{\nu}(h)=0$ for all $\nu$, and $\mathcal{H}(h)=0$.

Note that while a priori \eqref{psi U nontrivial} depends on $h$, we will actually prove it only depends on the double coset
$PhQ$ (this is only important for the archimedean parts).

Second, if any morphism in $\mathcal{H}(h)$ factors through $J_{V(\sigma),\psi'}(\rho)$, where $\sigma\succsim(k^c)$ and $\psi'\in\widehat{V}(\sigma)_{\mathrm{gen}}$, then $J_{V(\sigma),\psi'}(\rho)=0$ because $\rho$ is $(k,c)$, a fortiori $\mathcal{H}(h)=0$.

Let us remark, that these two methods for proving vanishing will be applied to all but finitely many representatives. In fact, consider the Bruhat decomposition $H=\coprod_{w'}Pw'Q$ where $w'$ are representatives for Weyl elements of $H$, and let $w_0$ denote the representative of the longest reduced Weyl element. Then the orbit $Pw_0Q$ is open. The above arguments prove vanishing on $H-Pw_0Q$. The remaining orbit $Pw_0Q$ is the disjoint union of finitely many orbits $PhD$, namely $n+1$ orbits when $H\ne\GL_{2kc}$ and $(c+1)(c+2)/2$ orbits for $H=\GL_{2kc}$. In particular for $\delta$, as explained in \S~\ref{Doubling setup}, one can choose $\delta=\delta_0\delta_1$ with $\delta_0=w_0$ and $\delta_1\in N_{H_0}$, then $P\delta D \subset Pw_0Q$. The orbits in $Pw_0Q$ must be handled using the third method, which we turn to describe.

Third, assume there is a composition $\beta$ of $kc$ and a character $\psi$ of $V_{\beta}$, which may depend on $h$, such that any morphism in $\mathcal{H}(h)$ factors through $J_{V_{\beta},\psi}(\rho)$. The vanishing argument in this case will be applicable to all but a discrete subset of $s$.

We first describe the non-archimedean case. Assume there is a proper parabolic subgroup $R=M_R\ltimes U_R<G$, with $M_R$ containing $\GL_l$ as a direct factor, $l\geq1$, such that $J_{V_{\beta},\psi}(\rho)$ is a trivial representation of ${}^h(1,U_R)$. Therefore
any morphism in $\mathcal{H}(h)$ also factors through $J_{U_R}(\pi_2^{\vee})$ which is an admissible finite length representation of $M_R$
(if $\pi_2$ is supercuspidal, we immediately deduce $\mathcal{H}(h)=0$).
On each irreducible constituent of $J_{U_R}(\pi_2^{\vee})$, as a representation of $M_R$, $C_{\GL_l}<C_{M_R}$ acts by a character, and there are only finitely many such characters possible, depending only on $\pi_2$ and $U_R$ (thereby on $h$).

Also assume $J_{V_{\beta},\psi}(\rho)$ admits a finite length filtration as a representation of
${}^h(1,\GL_l)$, and on each of the (not necessarily irreducible) constituents,
${}^h(1,C_{\GL_l})$ acts by a character. Again this character belongs to a finite set, now depending only on $\rho$ and on the character $\psi$ (which depends on $h$).

If $0\ne\mathcal{T}\in\mathcal{H}(h)$, we can take constituents $\mathcal{V}$ of $J_{V_{\beta},\psi}(\rho)$ and $\mathcal{V}'$ of $J_{U_R}(\pi_2^{\vee})$, such that $\mathcal{T}$ is well defined and nonzero on $\mathcal{V}\otimes\pi_1^{\vee}\otimes\mathcal{V}'$. We then obtain a relation
\begin{align}\label{eq:relation for T with s}
\mu(a)|a|^{bs}\mathcal{T}(\xi)=\mathcal{T}(\left({}^{h^{-1}}(|\det|^{s-1/2}\rho\otimes\eta)\otimes(\psi_U\otimes \pi_1^{\vee}\otimes\pi_2^{\vee})\right)(1,a)\xi)=\theta_h((1,a))\mathcal{T}(\xi),
\end{align}
where $\mu$ is a quasi-character of $F^*$ which belongs to a finite set depending only on $(\pi_2,\eta,\rho,h)$, and $b$ is a constant which depends only on $h$, and we assume $b\ne0$.
We deduce
$\mu(a)|a|^{bs}=\theta_h((1,a))$ for all $a\in F^*$. This excludes at most a discrete subset of $s$, and if we apply this argument to only finitely many representatives $h$, the set of these values of $s$ can be taken to be our $\mathcal{B}$.

Now assume the field is archimedean. Let $\mathfrak{u}_R$ denote the Lie algebra of $U_R$. Assume
${}^h(1,\mathfrak{u}_R)$ acts locally nilpotently on $J_{V_{\beta},\psi}(\rho)^*$. Then there is a countable increasing filtration of
(closed subspaces) $\mathcal{W}_i$ of $J_{V_{\beta},\psi}(\rho)^*$ by the order of nilpotency. The orthogonal complements $\mathcal{V}_{i}=(\mathcal{W}_i)_{\bot}\subset J_{V_{\beta},\psi}(\rho)$ form a decreasing filtration of $J_{V_{\beta},\psi}(\rho)$, exhausting in the sense that $\bigcap_i \mathcal{V}_i=0$. For each $i$,
$J_{V_{\beta},\psi}(\rho)/\mathcal{V}_{i}$ is a quotient of a generalized Jacquet module of $\rho$ with respect to
${}^h(1,\mathfrak{u}_R)$. Since any morphism in $\mathcal{H}_{\nu}(h)$ lies in some $\mathcal{W}_i$, it is annihilated by $\mathfrak{u}_R^i$. Thus it factors through a generalized Jacquet module $\pi_2^{\vee}/\overline{\mathfrak{u}_R^i\pi_2^{\vee}}$. The latter is
an admissible finite length representation of $M_R$ by Lemma~\ref{lem:adm}, in particular admits a finite filtration such that
$C_{\GL_l}$ acts by a character on each constituent.

Assume in addition, that there exists a parabolic subgroup of $\GL_{kc}$, whose Levi part contains
${}^h(1,\GL_l)$ as a direct factor, such that the Lie algebra $\mathfrak{v}$ of its unipotent radical acts locally nilpotently on $J_{V_{\beta},\psi}(\rho)^*$. Repeating the argument in the last paragraph, any morphism in $\mathcal{H}_{\nu}(h)$  factors through
a generalized Jacquet module $\rho/\overline{\mathfrak{v}^j\rho}$, and the latter --- by Lemma~\ref{lem:adm} --- has a finite filtration with
${}^h(1,C_{\GL_l})$ acting by a character on its constituents.

Now if $0\ne\mathcal{T}\in\mathcal{H}_{\nu}(h)$, there are constituents $\mathcal{V}$ of $\rho/\overline{\mathfrak{v}^j\rho}$,
$\mathcal{V}'$ of $\pi_2^{\vee}/\overline{\mathfrak{u}_R^i\pi_2^{\vee}}$ and $\mathcal{V}''$ of $\Lambda_{\nu}$ (considering $\Lambda_{\nu}$ as a representation of $(1,\GL_l)$), such that
$\mathcal{T}$ is well defined and nonzero on $\mathcal{V}\otimes\pi_1^{\vee}\otimes \mathcal{V}'\otimes \mathcal{V}''$.
Again we can apply \eqref{eq:relation for T with s} and obtain a relation $\mu(a)|a|^{bs}=\theta_h((1,a))$ (with $b\ne0$) for all $a\in F^*$. Here $\mu$ is uniquely determined by $\mathcal{V},\mathcal{V}',\mathcal{V}''$ and $h$. In the archimedean case this condition excludes one $s$.

As we vary $\mathcal{V}$ and $\mathcal{V}'$ over the finite filtrations of $\rho/\overline{\mathfrak{v}^j\rho}$ and $\pi_2^{\vee}/\overline{\mathfrak{u}_R^i\pi_2^{\vee}}$, and also vary $j$ and $i$, the actions of ${}^h(1,C_{\GL_l})$ and $C_{\GL_l}$ are given by a discrete set of characters, by Lemma~\ref{lem:adm 2 discrete}. The action of $(1,C_{\GL_l})$ on $\mathcal{V}''$ is also given by a discrete set of characters, because the central characters of the set of irreducible constituents of $\{\Lambda_{\nu}\}_{\nu}$ as representations of $(1,\GL_l)$ form a lattice.
Thus the total subset of $s$ we exclude is still discrete (for each $h$). Again, repeating this for finitely many $h$, we will obtain a discrete set $\mathcal{B}$.

\subsubsection{The space \eqref{eq:Hom L 0} is at most one-dimensional outside $\mathcal{B}$: archimedean case}

Re-denote the Bruhat cells appearing in the decomposition $P\backslash H/ Q$ by
$Y_0,\ldots, Y_l$, numbered such that if $Y_i\subset \overline{Y_j}$, $i>j$. In particular $Y_0$ is the open Bruhat cell
(i.e., $Y_0=Pw_0Q$). We have
\begin{align*}
\Hom_{D}(V(s,\rho\otimes\eta),\psi_U^{-1}\otimes\pi_1\otimes\pi_2)&\cong ((V(s,\rho\otimes\eta)\otimes \psi_U\otimes\pi^{\vee}_1\otimes\pi^{\vee}_2)^*)^{\Delta D}\\&\cong \cD'(H,(|\det|^{s-1/2}\rho\otimes\eta)\otimes \psi_U\otimes\pi^{\vee}_1\otimes\pi^{\vee}_2)^{D\times P}.
\end{align*}
First we show that outside $\mathcal{B}$,
\begin{align}\label{1=Goal}
\cD'_{\overline{Y_1}}(H,(|\det|^{s-1/2}\rho\otimes\eta)\otimes \psi_U\otimes\pi^{\vee}_1\otimes\pi^{\vee}_2)^{D\times P}=0.
\end{align}
For any $i>0$, let $X_i=\bigcup_{j=0}^i Y_j$, it is an open subset of $H$ and $Y_i$ is a closed submanifold of $X_i$.
It is enough to show that for any $i>0$, outside $\mathcal{B}$ we have
 \begin{align}\label{=Step}
 \cD'_{Y_i}(X_i,(|\det|^{s-1/2}\rho\otimes\eta)\otimes \psi_U\otimes\pi^{\vee}_1\otimes\pi^{\vee}_2)^{D\times P}=0.
 \end{align}
Indeed we show by induction on $i$ that for any distribution $\mathcal{T}$ belonging to the left hand side of
\eqref{=Step}, the restriction $\mathcal{T}|_{X_i}$ vanishes. The base case $i=0$ holds by definition, and the induction step is \eqref{=Step}.  Since $X_k=H$ we get $\mathcal{T}=0$.

To prove \eqref{=Step}, we divide $Y_i$ into two cases depending on the first two vanishing arguments from \S~\ref{the vanishing 3 types} (which apply to all $s$).
Assume \eqref{psi U nontrivial} holds and recall this condition only depends on the double coset (this is proved in
Proposition~\ref{proposition:1st reduction of w} below). In this case we show, for all $s$,
\begin{align*}
\cD'_{Y_i}(X_i,(|\det|^{s-1/2}\rho\otimes\eta)\otimes \psi_U\otimes\pi^{\vee}_1\otimes\pi^{\vee}_2)^{U\times U_P}=0.
\end{align*}
Indeed by \cite[\S~2, p.~70]{KV}, the left hand side can be identified with the subspace of $(U\times U_P)$-invariant maps from
$C_c^{\infty}(X_i,\psi_U)$ supported on $Y_i$ to $\left((|\det|^{s-1/2}\rho\otimes\eta)\otimes\pi^{\vee}_1\otimes\pi^{\vee}_2\right)^*$ (recall that over archimedean fields
${}^*$ denotes the continuous dual). Since $U\times U_P$ acts trivially on $(|\det|^{s-1/2}\rho\otimes\eta)\otimes\pi^{\vee}_1\otimes\pi^{\vee}_2$, such a nonzero map $\mathcal{L}$ would define
a nonzero distribution in $\cD'_{Y_i}(X_i,\psi_U)^{U\times U_P}$ (e.g., fix some functional which is nonzero on the image of $\mathcal{L}$). But $\cD'_{Y_i}(X_i,\psi_U)^{U\times U_P}=0$
by \cite[Theorem~3.15, case (iii)]{KV} (in their notation $M_y^{(r)}=\Lambda_v$ which can be taken to be trivial as explained above, and $\mathcal{O}=Y_i$).

Now assume \eqref{psi U nontrivial} does not hold.
We prove a more general result: for all $s$,
\begin{align}\label{=Goal}
\cD'_{Y_i}(X_i,(|\det|^{s-1/2}\rho\otimes\eta)\otimes \psi_U\otimes\pi^{\vee}_1\otimes\pi^{\vee}_2)^{U\times P}=0.
\end{align}
We deduce it from Theorem \ref{thm:AG} as follows. Let $X=X_i$, $Y=Y_i$, $C=U\times P$; $E=(|\det|^{s-1/2}\rho\otimes\eta)\otimes \pi^{\vee}_1\otimes\pi^{\vee}_2$ with $U$ acting trivially and $P$ acting only on $|\det|^{s-1/2}\rho\otimes\eta$; and $\chi=\psi_U^{-1}\times 1$.  Let $A=Q\times P$ and extend the action of $C$ on $E$ to an action of $A$ by letting $Q$ act trivially. Condition \eqref{=KV} follows from our proof of $\mathcal{H}(h)=0$ in this case (which uses the fact that $\rho$ is $(k,c)$). Note that $\mathcal{H}(h)$ can indeed be identified with the space of distributions on the orbit $PhD$ by, e.g., \cite[Theorem~5.2.4.5]{Warner1972I}.
The set $\{\chi^a|_{C_z}:a\in A\}$ is finite: first, $\psi_U|_{U\cap {}^{h^{-1}}U_P}=1$ and because this condition is independent of the representative $h$ in the double coset $PhQ$, $\chi^a$ is trivial on $U\cap{}^{h^{-1}}U_P$; and second, $Q\cap {}^{h^{-1}}M_P$ is a parabolic subgroup of $M_P$ and $U\cap{}^{h^{-1}}M_P$ is its unipotent radical (see \eqref{eq:beta} below).
From this and Lemma \ref{lem:Uhat} we deduce that the set $\{\chi^a|_{C_z}:a\in A\}$ is a finite union of orbits. All orbits are locally closed since they are orbits of an algebraic action of an algebraic group (note that the characters are unitary).
Thus Theorem \ref{thm:AG} implies \eqref{=Goal}.

Altogether we have shown \eqref{1=Goal}. Therefore restriction of $D\times P$-equivariant distributions from $H$ to $Y_0$ is injective.
Now $D\times P$ acts on $Y_0$ with finitely many orbits $Z_0,\ldots,Z_r$, enumerated such that $Z_i\subset \overline{Z_j}$ implies $i>j$, in particular, $Z_0=P\delta D$ is the open orbit. As above is suffices to prove that for any $i>0$, but for $s\notin \mathcal{B}$,
 \begin{align}\label{2=Step}
 \cD'_{Z_i}(\bigcup_{j=0}^i Z_j,(|\det|^{s-1/2}\rho\otimes\eta)\otimes \psi_U\otimes\pi^{\vee}_1\otimes\pi^{\vee}_2)^{D\times P}=0.
 \end{align}
Let $A=C=D\times P$, $E=(|\det|^{s-1/2}\rho\otimes\eta)\otimes \pi^{\vee}_1\otimes\pi^{\vee}_2$ with $D$ acting only on $\pi^{\vee}_1\otimes\pi^{\vee}_2$ ($U$ acting trivially), $P$ acting only on $|\det|^{s-1/2}\rho\otimes\eta$; and $\chi=\psi_D\times 1$, where $\psi_D$ is the character of $D$ defined by $\psi_U^{-1}$ extended trivially to $(G,G)$. Our proof of $\mathcal{H}(h)=0$ in this case
(using \eqref{eq:relation for T with s}) implies \eqref{=KV} for $s\notin\mathcal{B}$, and Theorem \ref{thm:AG} implies \eqref{2=Step}. It then follows that restriction of $D\times P$-equivariant distributions from $H$ to $Z_0$ is injective. Combining this with the fact that
$\dim\mathcal{H}(\delta)\leq1$, we are done.

\subsubsection{The main result}
We turn to formulate our main result. Define
\begin{align*}
d(s,\rho,\eta,\pi_1,\pi_2)=\dim\Hom_{(G,G)}(J_{U,\psi_U^{-1}}(V(s,\rho\otimes\eta)),\pi_1\otimes\pi_2).
\end{align*}

\begin{theorem}\label{theorem:uniqueness}
Let $\pi_1$, $\pi_2$ and $\rho$ be as above.
\begin{enumerate}[leftmargin=*]
\item\label{part1}Outside a discrete subset of $s$, $d(s,\rho,\eta,\pi_1,\pi_2)\leq\dim\Hom_{G}(\chi_0\pi_1^{\vee},\pi_2^{\iota})$.
\item\label{part2}If $\pi_1$ and $\pi_2$ are irreducible, outside a discrete subset of $s$,
$d(s,\rho,\eta,\pi_1,\pi_2)=0$ unless $\pi_1=\chi_0(\pi_2^{\iota})^{\vee}$ in which case $d(s,\rho,\eta,\pi_1,\pi_2)\leq1$.
\end{enumerate}
Furthermore, assume $\pi_2$ is supercuspidal and $\rho$ is not necessarily of finite length. Then
the assertions of \eqref{part1} and \eqref{part2} hold for all $s$, granted one of the following:
\begin{enumerate}[leftmargin=*,label=(\alph*)]
\item\label{part3}$H\ne\GL_{2kc}$ and $c>2$; or $H=\Sp_{4k}$ ($G=\Sp_{2}$); or $H\ne\GL_{2kc}$, $c=2$ and $\rho=\rho_c(\tau)$ for an irreducible supercuspidal representation $\tau$ of $\GL_k$ and $k>1$.
\item\label{part4}$H=\GL_{2kc}$, $ck>1$, $\pi_1$ is also supercuspidal, and $\rho_i=\rho_c(\tau_i)$ for irreducible supercuspidal representations $\tau_1$ and $\tau_2$ of $\GL_k$.
\end{enumerate}
\end{theorem}
\begin{remark}
If $\eta\ne\chi_{\pi_1}^{-1}$, $d(s,\rho,\eta,\pi_1,\pi_2)=0$ outside a discrete subset of $s$.
\end{remark}
The proof of the theorem occupies \S~\ref{section not GL}--\S~\ref{section GL}. Note that the case of $\GL_1\times\GL_k$ over non-archimedean fields was proved in \cite[Lemma~35]{CFGK2} for $\pi_1=\pi_2^{\vee}$ and $\chi_0=1$ ($P\backslash H/D$ is finite in this case).

Recall the representations $\pi$ and $\rho$ defined in \S~\ref{Doubling setup}:
$\pi$ is an irreducible admissible representation of $G$, and
$\rho$ is either an admissible finite length $(k,c)$ representation of $\GL_{kc}$ which admits a central character, or
the tensor product $\rho_1\otimes\chi^{-1}\rho_2$ of two such representations $\rho_i$ of $\GL_{kc}$, with a quasi-character $\chi$ of $\GL_k$,
in which case $\chi_0=\chi^{-k}$ (the central character of $\rho_1$ is the inverse of the central character of $\rho_2$). This is a minor generalization of \cite{CFK}, where $\rho$ was taken to be $\rho_c(\tau)$ (or $\rho_i=\rho_c(\tau_i)$, $i=1,2$).

Combining Theorem~\ref{theorem:uniqueness} with the doubling integral, we obtain the following.
\begin{corollary}\label{coro:meromorphic continuation for doubling 1}
Let $F$ be non-archimedean and consider \eqref{local integral} for the representations $\pi$ and $\rho$ defined in \S~\ref{Doubling setup}.
\begin{enumerate}[leftmargin=*]
\item\label{mero coro:part1}If $f$ is a rational section in $q^{-s}$, $Z(s,\omega,f)$ admits meromorphic continuation to a rational function in $q^{-s}$.
\item\label{mero coro:part2}$d(s,\rho,\chi_{\pi},\chi_0\pi^{\vee},\pi^{\iota})\geq1$ for all $s$.
\end{enumerate}
\end{corollary}
\begin{proof}
For part~\eqref{mero coro:part1}, by Theorem~\ref{theorem:uniqueness} with $\pi_1=\chi_0\pi^{\vee}$ and $\pi_2=\pi^{\iota}$ (then $\eta=\chi_{\pi_1}^{-1}=\chi_{\pi}$), the dimension of \eqref{eq:doubling Hom} is at most $1$ outside a discrete subset of $s$. Now the meromorphic continuation follows from
Theorem~\ref{theorem:basic props of doubling integrals} and Bernstein's continuation principle (\cite{Banks}).

For part~\eqref{mero coro:part2}, fix some $s_0$. Consider the family $\mathcal{I}$ of integrals $Z(s,\omega,f)$, where $\omega$ varies over the matrix coefficients of $\pi^{\vee}$, and $f$ varies over the sections of $V(s,\rho\otimes\chi_{\pi})$ that are polynomial in $q^{\mp s}$. The set of poles of $Z(s,\omega,f)\in\mathcal{I}$ belongs to a finite set of values of $q^{-s}$ which depends only on the representations $\pi$ and $\rho$, by \cite{Banks} (we do not claim the multiplicity of a pole is bounded independently of $\omega$ and $f$). Therefore, there is $r>0$ such that all integrals of $\mathcal{I}$ are holomorphic in the punctured disk of radius $r$ around $s_0$. Let $\gamma$ be the boundary of this disk. Moreover, by Theorem~\ref{theorem:basic props of doubling integrals} \eqref{props 3}, there is $Z(s,\omega,f)\in\mathcal{I}$ which is a nonzero constant at $s_0$. Thus Cauchy's integral formula gives a nonzero morphism
$(\omega,f)\mapsto\tfrac1{2\pi i}\oint_{\gamma}\frac{Z(s,\omega,f)}{s-s_0}\,ds$ in \eqref{eq:doubling Hom}.
\end{proof}
\begin{corollary}\label{coro:meromorphic continuation for doubling 2}
Consider \eqref{local integral} for the representations $\pi$ and $\rho$ defined in \S~\ref{Doubling setup}.
Assume $\pi$ is irreducible supercuspidal and the additional assumptions \ref{part3} or \ref{part4} of Theorem~\ref{theorem:uniqueness} hold.
\begin{enumerate}[leftmargin=*]
\item\label{supercusp coro:part1}$d(s,\rho,\chi_{\pi},\chi_0\pi^{\vee},\pi^{\iota})=1$ for all $s$.
\item\label{supercusp coro:part2}If $f$ is a polynomial section in $q^{\mp s}$, $Z(s,\omega,f)$ admits analytic continuation to a
polynomial function in $q^{\mp s}$.
\end{enumerate}
\end{corollary}
\begin{proof}
The first assertion follows from Theorem~\ref{theorem:uniqueness} combined with
Corollary~\ref{coro:meromorphic continuation for doubling 1} \eqref{mero coro:part2}. The second holds because when
$d(s,\rho,\chi_{\pi},\chi_0\pi^{\vee},\pi^{\iota})=1$ for all $s$, by the corollary in \cite{Banks} the continuation is to a polynomial.
\end{proof}
\subsection{The case $H\ne\GL_{2kc}$}\label{section not GL}
As explained in \S~\ref{outline}, we will consider each $\mathcal{H}(h)$ separately. We prove that all but finitely many spaces
$\mathcal{H}(h)$ vanish using the first two methods, show the vanishing of the remaining $\mathcal{H}(h)$ with $h\not\sim\delta$ outside a discrete subset $\mathcal{B}$, then prove $\dim\mathcal{H}(\delta)\leq1$.

Recall $n=\lfloor c/2\rfloor$, and since we prove the result for both odd and even $c$ simultaneously, we also use
$\lceil c/2\rceil$, which is $n$ when $c$ is even and $n+1$ otherwise.

We start with describing a choice of representatives.
Since $P\backslash H/ N_H$ can be identified with $W(M_P)\backslash W(H)$ and $Q=M_Q\ltimes U$, we can write
$P\backslash H/ D=\coprod_hPhD$ with $h=wu$, where $w$ is a representative from $W(M_P)\backslash W(H)$ and $u\in M_Q\cap N_H$.
Since $W(M_P)\backslash W(H)$ is embedded in $\mathbb{Z}_2^{kc}$, we can identify $w$ with a $kc$-tuple of $0$ and $1$, where
the $i$-th coordinate corresponds to the permutation matrix
\begin{align*}
\left(\begin{smallmatrix}
I_{kc-i}
\\&0&&1
\\&&I_{2(i-1)}
\\&\epsilon_0&&0
\\&&&& I_{kc-i}
\end{smallmatrix}\right).
\end{align*}
If $H\ne\Sp_{2kc}$, only even products of such matrices can appear in $w$. In this case denote for an integer $a\geq0$,
$\jmath_a=(1,0^{kc-1})$ if $a$ is odd otherwise $\jmath_a=(0^{kc})$. Note that $\jmath_a$ normalizes $N_H$. If $H=\Sp_{2kc}$, we set $\jmath_a=(0^{kc})$ for uniformity. We use $*$ in an expression for $w$ to signify an undetermined coordinate (either $0$ or $1$).

For the case $k=1$, we can parameterize $P\backslash H/D=P\backslash H/(G,G)$ using the elements
\begin{align*}
\jmath_l(0^{c-l},1^l)({}^{\jmath_l}u_l), \qquad 0\leq l\leq n,
\end{align*}
where
\begin{align}\label{eq:weyl rep k=1}
(0^{c-l},1^l)=\left(\begin{smallmatrix}
&&&&&J_{{l}}\\&I_{2(c-l)}\\\epsilon_0J_{l}
\end{smallmatrix}\right)
\end{align}
and
\begin{align}\label{eq:unipotent rep k=1}
u_l=
\left(\begin{smallmatrix}
I_l&&I_l\\&I_{n-l}\\&&I_l\\&&&I_{2c-2(n+l)}\\&&&&I_l&&-I_l\\&&&&&I_{n-l}\\&&&&&&I_l
\end{smallmatrix}\right).
\end{align}
Here if $l<\lceil c/2\rceil$, ${}^{\jmath_l}u_l=u_l$ (always the case for an odd $c$). The double cosets for $k=1$ were described in \cite[\S~2]{PSR}; see also \cite[\S~4]{GRS7} for $\Sp_{2c}$; $\SO_{2c}$ and $\GSpin_{2c}$ with even $c$ are similar, and for odd $c$ also refer to the description of the embedding of $\SO_c\times \SO_c$ in $\SO_{2c}$ given in \cite[Example~15]{CFK} (see Example~\ref{example:odd c}).

We start by generalizing this description, to some extent, to all $k\geq1$.
For $x\in M_Q$, denote its projection into the direct product of $k-1$ copies of $\GL_c$ by $\ell(x)$, then $x=\ell(x)\ell_0(x)$, where $\ell_0(x)\in H_0$. For $k=1$, $x=\ell_0(x)$ and $\ell(x)$ is trivial. If $y\in M_Q$,
\begin{align*}
\ell({}^xy)={}^{\ell(x)}\ell(y),\qquad \ell_0({}^xy)={}^{\ell_0(x)}\ell_0(y).
\end{align*}
For any $(g_1,g_2)\in (G,G)$,
\begin{align}\label{eq:conj x by g1 g2}
{}^{(g_1,g_2)}x=\ell({}^{(g_1,g_2)}x)\ell_0({}^{(g_1,g_2)}x),
\end{align}
but because $(1,g_2)\in H_0$,
\begin{align}\label{eq:conj x by g2}
{}^{(1,g_2)}x=\ell(x)({}^{(1,g_2)}\ell_0(x)).
\end{align}

\begin{proposition}\label{proposition:structure of w u}
Let $h=wu$, where $w$ is a representative from $W(M_P)\backslash W(H)$ and $u\in M_Q\cap N_H$.
Then $h\sim\hat{w}\hat{u}$ with the following properties. There is $0\leq l\leq n$ such that
\begin{align}\label{eq:start condition of w}
&\hat{w}=\jmath_a(0^{c-l},1^{l},w_2,\ldots,w_k),\qquad\forall i, w_i\in\{0,1\}^{c},
\end{align}
where $a$ is the sum of coordinates $1$ in $(0^{c-l},1^{l},w_2,\ldots,w_k)$. Additionally $\hat{u}\in M_Q$, there
is $\sigma=(g,1)\in(G,1)$ such that $g$ is a representative of an element in $W(G)$ and
${}^{\sigma}\hat{u}\in M_Q\cap N_H$, ${}^{\jmath_a}\ell_0(\hat{u})$ takes the form
\begin{align}\label{eq:second form u final}
\left(\begin{smallmatrix}
I_l&&A_l&&&&\\&I_{n-l}\\&&I_l\\&&&I_{2c-2(n+l)}\\&&&&I_l&&A_l'\\&&&&&I_{n-l}\\&&&&&&I_l
\end{smallmatrix}\right),
\end{align}
and there are no zero rows in $A_l$.
\end{proposition}
\begin{proof}
Let $E=M_{E}\ltimes U_{E}$ denote the standard parabolic subgroup of $H_0$ such that $M_{E}=\GL_n\times\mathcal{G}_{2(c-n)}$, and identify $N_{\GL_n}$ with its natural image in $M_{E}$. According to the description of $(G,G)$ in \S~\ref{Doubling setup}, $N_{\GL_n}\ltimes C_{U_E}<\ell_0((G,1))$. Let $g\in N_G$ be with
$\ell_0((g,1))\in N_{\GL_n}\ltimes C_{U_E}$, such that the projection of $\ell_0(u(g,1))$ into $N_{H_0}$ is trivial on $N_{\GL_n}$. Put $u_1=u(g,1)\in M_Q\cap N_H$, $wu\sim wu_1$. We also have some control over the projection of the unipotent part of the representative into $C_{U_E}$ (see below).

If $c$ is even, then also $N_{\mathcal{G}_{2(c-n)}}<(1,G)$, and for $g\in N_{\mathcal{G}_{2(c-n)}}$ such that the projection of
$\ell_0(u_1(1,g))$ into $N_{\mathcal{G}_{2(c-n)}}$ is trivial, $wu_1\sim wu_2$ with $u_2=u_1(1,g)$. The projection of $\ell_0(u_2)$ into
$N_{\GL_n}\ltimes C_{U_E}$ coincides with that of $\ell_0(u_1)$.

If $c$ is odd, $N_{\mathcal{G}_{2(c-n)}}/(1,N_G)\cong\Mat_{n\times1}$. Choosing $g\in N_G$ and taking $u_2=u_1(1,g)$, we can assume the projection of $\ell_0(u_2)$ into $N_{\mathcal{G}_{2(c-n)}}$ takes the form
\begin{align}\label{eq:remaining inner column odd c}
\left(\begin{smallmatrix}I_{n}&y_1&y_2&y''\\
&1&&y_1'\\&&1&y_2'\\&&&I_n
\end{smallmatrix}\right)\in N_{\mathcal{G}_{2(n+1)}},
\end{align}
where $y_i'$, $y''$ uniquely depend on $y_1,y_2$ and $H$, and we can choose either $y_1=0$ or $y_2=0$ (see Example~\ref{example:odd c}).
Observe that $w$ conjugates precisely one of the columns $y_1$ or $y_2$ in \eqref{eq:remaining inner column odd c} into $P$, so that if we choose the other column to be zero (i.e., define $g\in N_G$ accordingly), then now we already have both zero. In other words we can write $u_3=z^{-1}u_2$ for $z$ defined by $y_1$ or $y_2$ such that ${}^{w}z\in P$, then $wu_1\sim wu_2=wzu_3\sim wu_3$. If $c$ is even, put $u_3=u_2$, so that the projection of $\ell_0(u_3)$ into $N_{\mathcal{G}_{2(c-n)}}$ is trivial now for both odd and even $c$.

One can take
a representative $g$ of an element in $W(G)$ such that
\begin{align}\label{eq:w_1}
w_1=w(1,g)=j_a(0^{\lceil c/2\rceil},*^{kc-\lceil c/2\rceil}).
\end{align}
Then $wu_3\sim wu_3(1,g)=w_1({}^{(1,g)^{-1}}u_3)$. Put $u_4={}^{(1,g)^{-1}}u_3$, it is of the same form as $u_3$: this conjugation merely permutes the columns in the projection of $\ell_0(u_3)$ into $C_{U_{E}}\backslash U_{E}$.
Now we can write
\begin{align*}
\ell_0(u_4)=\left(\begin{smallmatrix}I_n&v&v''\\&I_{2(c-n)}&v'\\&&I_n\end{smallmatrix}\right)\in N_{H_0},
\end{align*}
where $v\in\Mat_{n\times 2(c-n)}$ is arbitrary and $v',v''$ are uniquely defined given $v$ and $H$.

Put $v=(v_1,v_2)$, $v_1=(z_1,z_2)$ and $v_2=(z_3,z_4)$, where $z_1,z_4\in\Mat_{n\times \lceil c/2\rceil-1}$ and $z_2,z_3\in\Mat_{n\times1}$.
The element $w_1$ does not permute any column of $z_1$ or $z_4$, and conjugates the block $z_4$ into $M_P$. Hence one can write $u_5=z^{-1}u_4$, where $z\in N_{H_0}$ is defined by $z_4$, and the corresponding block $z_4$ of $\ell_0(u_5)$ is $0$, then $w_1u_4=w_1zu_5\sim w_1u_5$.

Next we see that $w_1$ also conjugates precisely one column $z_j$ out of $\{z_2,z_3\}$ into $P$. If $a$ is even, $j=3$
and we can assume $z_3=0$. Otherwise $j=2$, we can assume $z_2=0$. In both cases we multiply $u_5$ on the left by a suitable matrix $z^{-1}$, and $w_1u_5\sim w_1u_6$ with
$u_6=z^{-1}u_5$. We deduce
\begin{align*}
{}^{\jmath_a}\ell_0(u_6)=\left(\begin{smallmatrix}I_n&v&0&v''\\&I_{c-n}&&0\\&&I_{c-n}&v'\\&&&I_n\end{smallmatrix}\right)\in N_{H_0}.
\end{align*}

If $c$ is odd, $v$ contains $n+1$ columns. We show that the rightmost column of ${}^{\jmath_a}\ell_0(u_6)$ can be made $0$. Indeed
we can take $g\in N_{G}$ such that
\begin{align*}
\ell_0((g,1))=\left(\begin{smallmatrix}I_n&&\epsilon_2x&-\epsilon_1x&&y\\&I_{n}\\&&1&&&\epsilon_1x'\\&&&1&&-\epsilon_2x'\\&&&&I_n\\&&&&&I_n\end{smallmatrix}\right)
\in N_{H_0}
\end{align*}
(see Example~\ref{example:odd c}). The element $w_1$ permutes precisely one of the middle $2$ columns into $P$, either the column with
$\epsilon_2x$ or with $-\epsilon_1x$. Then if $\ell_0((g,1))$ is chosen such that the other column is $0$ in $u_6(g,1)$, and $z_j\in N_{H_0}$ is defined
by the column of $\ell_0((g,1))$ which is permuted by $w_1$ into $P$ (thus ${}^{w_1}z_j\in P$),
\begin{align*}
w_1u_6\sim w_1u_6(g,1)=w_1z_jz_j^{-1}u_6(g,1)\sim w_1z_j^{-1}u_6(g,1).
\end{align*}
Put $u_7=z_j^{-1}u_6(g,1)$, then
\begin{align*}
{}^{\jmath_a}\ell_0(u_{7})=\left(\begin{smallmatrix}I_n&v&&&v''\\&I_{n}&&&\\&&I_{2(c-2n)}\\&&&I_{n}&v'\\&&&&I_n\end{smallmatrix}\right)\in N_{H_0}.
\end{align*}
For uniformity, denote $u_7=u_6$ when $c$ is even.

By the definition of $H$, the block $v''$ in ${}^{\jmath_a}\ell_0(u_{7})$ above can be taken independently of $v$. Hence we can multiply
$u_7$ on the right by $(g,1)$ with $g\in N_G$, where $\ell_0((g,1))\in C_{U_E}$ is defined using $v''$, and obtain $u_8=u_7(g,1)$
such that ${}^{\jmath_a}\ell_0(u_8)$ is of the form
\begin{align}\label{eq:starting form u''}
\left(\begin{smallmatrix}I_n&v&&&\\&I_{n}&&&\\&&I_{2(c-2n)}\\&&&I_{n}&v'\\&&&&I_n\end{smallmatrix}\right)\in H_0.
\end{align}
Then $w_1u_7\sim w_1u_8$. If $c$ is odd, $\ell_0(u_8)$ commutes with $\jmath_a$ (since then $2(c-2n)=2$).

At this point we still have $u_8\in M_Q\cap N_H$, since the only changes from $u$ to $u_8$ involve multiplying by elements of $M_Q\cap N_H$ (on the right or left).

For any matrix $u_0$ of the form
\eqref{eq:starting form u''}, denote the block of $v$ by $v(u_0)$.
For any representative $w'=(*^{kc})$, let $\mathcal{R}(w')$ denote the set of $1\leq i\leq n$ such that $w'$ permutes the $i$-th row of the block $v$ of a general matrix \eqref{eq:starting form u''}. Note that $\mathcal{R}(w')$ only depends on the coordinates $\lceil c/2\rceil +1,\ldots,c$ of $w'$ (enumerating the coordinates of $w'$ from left to right).

For each row $i$ of $v({}^{\jmath_a}\ell_0(u_8))$, one can always write $u_9=z_i^{-1}u_8$, where the $i$-th row of $v({}^{\jmath_a}\ell_0(u_9))$ is zero, ${}^{\jmath_a}z_i$ is of the form \eqref{eq:starting form u''} and any row $j\ne i$ in $v({}^{\jmath_a}z_i)$ is zero. Moreover, $z_i\in P$, and
if $i\notin\mathcal{R}(w_1)$, $w_1$ commutes with $z_i$. Hence
$w_1u_8=z_iw_1u_9\sim w_1u_9$.
Since we can apply this separately to each row, we can assume that for each $1\leq i\leq n$,
either the $i$-th row of $v({}^{\jmath_a}\ell_0(u_9))$ is zero or $i\in\mathcal{R}(w_1)$. The difference between
$u_8$ and $u_9$, is that the nonzero rows of $v({}^{\jmath_a}\ell_0(u_9))$ occur only at rows $i$ which $w_1$ permutes.

Consider $i$ such that both the $i$-th row of $v({}^{\jmath_a}\ell_0(u_9))$ is zero and $i\in\mathcal{R}(w_1)$.
In this case take $\sigma_1=(g,1)$ where $g\in G$ is a representative of an element of $W(G)$ of minimal length, such that
$\mathcal{R}(\sigma_1)=\{i\}$. More specifically take $g$ with $\ell_0((g,1))=\jmath_1(0^{c-i},1,0^{i-1})$ (if $c$ is odd, the right hand side is multiplied by $\diag(I_{c-1},2\epsilon_1^2,2\epsilon_2^2,I_{c-1})$, see Example~\ref{example:odd c}). Since $\sigma_1=\ell(\sigma_1)\ell_0(\sigma_1)$ and $\ell(\sigma_1)\in P$,
\begin{align*}
w_1u_9\sim w_1u_9\sigma_1=
w_1\sigma_1({}^{\sigma_1^{-1}}u_9)=
\ell(\sigma_1)({}^{\ell(\sigma_1)^{-1}}w_1)\ell_0(\sigma_1)({}^{\sigma_1^{-1}}u_9)\sim
({}^{\ell(\sigma_1)^{-1}}w_1)\ell_0(\sigma_1)({}^{\sigma_1^{-1}}u_9).
\end{align*}
Put $w_2=({}^{\ell(\sigma_1)^{-1}}w_1)\ell_0(\sigma_1)$, it is again a representative from $W(M_P)\backslash W(H)$
and
\begin{align*}
\mathcal{R}(w_2)=
\mathcal{R}(w_1\ell_0(\sigma_1))=\mathcal{R}(w_1)-\{i\}.
\end{align*}
Let $u_{10}={}^{\sigma_1^{-1}}u_9$. We have $\ell_0(u_{10})=\ell_0(u_9)$ if $H=\Sp_{2kc}$ or $c$ is odd, otherwise $\ell_0(u_{10})$ differs from $\ell_0(u_9)$ only in the middle $2$ columns: these columns are exchanged because of $\jmath_1$. The element
$\ell(u_{10})$ need not be in $N_H$ anymore, only in $M_Q$, but ${}^{\sigma_1}u_{10}\in M_Q\cap N_H$ and by \eqref{eq:conj x by g1 g2},
also ${}^{\ell(\sigma_1)}\ell(u_{10})\in (H_0\backslash M_Q)\cap N_H$. Since we can apply this procedure separately to each row $i$, we can assume the $i$-th row of $v({}^{\jmath_a}\ell_0(u_{10}))$ is nonzero if and only if $i\in\mathcal{R}(w_2)$.
However, we can no longer assume $\ell(u_{10})\in N_H$.

Regard $\GL_n$ as the direct factor of the standard Levi subgroup
$\GL_n\times\mathcal{G}_{c-2n}$ of $G$. For any representative $g$ of an element of $W(\GL_n)$, set $\sigma_2=(g,1)$. Given
arbitrary sets $\mathcal{R}(w')$ and $\mathcal{R}'\subset\{1,\ldots,n\}$ of the same size, one can find $\sigma_2$ for which $\mathcal{R}({}^{\sigma_2^{-1}}w')=\mathcal{R}'$. Because $i\in \mathcal{R}(w_2)$ if and only if
the $i$-th row of $v({}^{\jmath_a}\ell_0(u_{10}))$ is nonzero, we can choose $\sigma_2$ such that $\mathcal{R}({}^{\sigma_2^{-1}}w_2)=\{1,\ldots,l\}$, where $0\leq l\leq n$ is the size of $\mathcal{R}(w_2)$, and simultaneously
\begin{align}\label{eq:starting form u'''}
\ell_0({}^{\jmath_a}({}^{\sigma_2^{-1}}u_{10}))=
\left(\begin{smallmatrix}
I_l&&v&&&&\\&I_{n-l}\\&&I_n\\&&&I_{2(c-2n)}\\&&&&I_n&&v'\\&&&&&I_{n-l}\\&&&&&&I_l
\end{smallmatrix}\right)
\end{align}
where none of the rows of $v$ are zero. Set $w_3={}^{\sigma_2^{-1}}w_2$ and $u_{11}={}^{\sigma_2^{-1}}u_{10}$. Since $\sigma_2\in (G,1)\cap P$,
\begin{align*}
w_2u_{10}\sim w_2u_{10}\sigma_2=
\sigma_2({}^{\sigma_2^{-1}}w_2)({}^{\sigma_2^{-1}}u_{10})\sim w_3u_{11}.
\end{align*}
Now $\mathcal{R}(w_3)=\{1,\ldots,l\}$ and note that $w_3$ is still of the form
\eqref{eq:w_1} (with possibly a different $a$, but of the same parity), because when we pass to $w_2$ and then to $w_3$, we do not change the coordinates $2,\ldots,\lceil c/2\rceil$ of $w_1$ ($c-i\geq \lceil c/2\rceil$ for $i\leq n$).
Moreover, ${}^{w_3}({}^{\jmath_a}(1,g))\in M_P$ for any $g\in\GL_n$ (where $\GL_n<\GL_n\times\mathcal{G}_{c-2n}<G$);
if $\jmath_a$ is trivial, $w_3$ simply commutes with $(1,\GL_n)$. Also
\begin{align}\label{eq:important prop of u9}
{}^{\sigma_1\sigma_2}u_{11}={}^{\sigma_1}u_{10}\in M_Q\cap N_H.
\end{align}

The rank of $v$ in \eqref{eq:starting form u'''} is at most $l$, whence we can further use ${}^{\jmath_a}(1,g_0)$ with $g_0\in\GL_n$ to reduce $v$ to an $l\times l$ block (e.g., in a column reduced echelon form). Denote $\hat{w}=w_3$ and $\hat{u}={}^{{}^{\jmath_a}(1,g_0)^{-1}}u_{11}$. Now
$\mathcal{R}(\hat{w})=\{1,\ldots,l\}$ and $\hat{w}$ takes the form \eqref{eq:start condition of w}, namely $\jmath_a(0^{c-l},1^{l},*^{(k-1)c})$.
Since ${}^{w_3}{}^{\jmath_a}(1,g)\in M_P$ for any $g\in\GL_n$,
$w_3u_{11}\sim\hat{w}\hat{u}$.

Regarding $\hat{u}$, ${}^{\jmath_a}\ell_0(\hat{u})$ takes the form \eqref{eq:second form u final} with
$A_l=v$. Denote $\sigma=\sigma_1\sigma_2$ with the notation above.
We claim ${}^{\sigma}\hat{u}\in N_H$ (clearly ${}^{\sigma}\hat{u}\in M_Q$). Since the conjugation by ${}^{\jmath_a}(1,g_0)^{-1}$ only affects the columns of $v$ and rows of $v'$ in
\eqref{eq:starting form u'''}, the result follows from
\eqref{eq:important prop of u9}.
\end{proof}
While it is relatively straightforward to obtain condition \eqref{psi U nontrivial} when $h=w$,
the representatives $wu$ are more difficult to describe, because of the form of $\ell(u)$.
The following lemma implies that (with our current structure of $u$) it is sufficient to obtain \eqref{psi U nontrivial} for $w$.
\begin{lemma}\label{lemma:easier condition on psiU}
Let $h=wu$, where $w$ and $u$ are given by Proposition~\ref{proposition:structure of w u}. Assume
\begin{align}\label{psi U nontrivial using only w}
\psi_U|_{U\cap {}^{w^{-1}}U_P}\ne1.
\end{align}
Then \eqref{psi U nontrivial} holds as well, namely
$\psi_U|_{U\cap {}^{h^{-1}}U_P}\ne1$.
\end{lemma}
\begin{proof}
By \eqref{psi U nontrivial using only w} there exists a root in $U$, such that for the subgroup $Y<U$ generated by this root, ${}^wY<U_P$ and $\psi_U|_Y\ne1$. Since $u\in M_Q$, it normalizes $U$ whence ${}^{u^{-1}}Y<U$, and also ${}^h({}^{u^{-1}}Y)={}^wY<U_P$. It remains to show $\psi_U|_{{}^{u^{-1}}Y}\ne1$, since then \eqref{psi U nontrivial} holds.

We can identify the quotient of $U$ by its commutator subgroup, with the direct product of $k-2$ copies of
$\Mat_c$, and one copy of $\Mat_{c\times 2c}$. The root defining $Y$ corresponds to a coordinate $(i,j)$ in one of these copies.
Looking at the definition of $\psi_U$, we can be more specific. Either $(i,j)$ belongs to one of the $k-2$ blocks of size $c\times c$, on which $\psi_U$ is given by $\psi\circ\tr$, or $(i,j)$ belongs to one of two $n\times n$ blocks inside $\Mat_{c\times 2c}$
($u^{1,1}$ or $u^{2,2}$, see \eqref{matrix:middle 4c block of u}), and again $\psi_U$ restricts to $\psi\circ\tr$ on these blocks. Denote the block by $B$. In both cases, since $\psi_U|_Y\ne1$, the coordinate $(i,j)$ appears as a diagonal coordinate of $B$. When $c$ is odd there is a third possibility, that $(i,j)$ appears in the block $B\in\Mat_{1\times2}$ on which $\psi_U$ is given by
$\psi(\epsilon_1B_{1,1}-\epsilon_2B_{1,2})$, and $(i,j)$ is either the coordinate of $B_{1,1}$ or $B_{1,2}$. In this case also note that $w$ of the prescribed structure can not permute both $B_{1,1}$ and $B_{1,2}$ into $U_P$.

Write ${}^{\sigma}u=u_0^{-1}\in M_Q\cap N_H$, where $\sigma\in (G,1)$ is given by Proposition~\ref{proposition:structure of w u}. Then
${}^{\sigma}Y$ is again a root subgroup, and since conjugation by $\sigma$ permutes the coordinates of $U$ and stabilizes $\psi_U$,
${}^{\sigma}Y$ is still defined be a coordinate $(i,j)$ which belongs to one of the blocks $B'$ described above. In fact if $B\in\Mat_{c}$, we must have $B'=B$; if $B\in\Mat_{n}$, there are two options for $B'$, one of which is $B$; and for $B\in\Mat_{1\times 2}$, $B'=B$. Of course $\psi_U$ is nontrivial on ${}^{\sigma}Y$.

The conjugation of ${}^{\sigma}Y$ by $u_0$ must be performed with more care, because $u_0$ normalizes $U$ but may not stabilize $\psi_U$.
First consider the case where $B=B'\in\Mat_{c}$. Then ${}^{\sigma}Y$ is the root subgroup defined by the $(d,d)$-th diagonal coordinate in $B$, for some $1\leq d\leq c$. For a fixed element $y\in Y$, assume the $(d,d)$-th coordinate of ${}^{\sigma}y$ is $x\ne0$. It is the only nonzero coordinate in the projection of $y$ to $B$. Since $u_0\in M_Q$, the nontrivial coordinates of ${}^{u_0}({}^{\sigma}y)$ are still contained in $B$.
This means that the only nonzero coordinates of ${}^{u_0}({}^{\sigma}y)$ on which $\psi_U$ can possibly be nontrivial, are coordinates in the block $B$. Because $u_0\in M_Q\cap N_H$, the $(d,d)$-th coordinate of
${}^{u_0}({}^{\sigma}y)$ is still $x$, and all other nontrivial coordinates belong to the set of coordinates in $B$ of the form $\{(i',j')\ne(d,d):i'\leq d,j'\geq d\}$, i.e., are above or to the right of the $(d,d)$-th coordinate. Therefore
$\psi_U({}^{u_0}({}^{\sigma}y))=\psi(x)$, hence $\psi_U$ is nontrivial on ${}^{u_0\sigma}Y$.

Next assume $B'\in\Mat_{n}$ and proceed with similar notation. Now ${}^{u_0}({}^{\sigma}y)$ can contain nontrivial coordinates outside $B'$. Assume $B'$ is the top left $n\times n$ block in $\Mat_{c\times2c}$ (i.e., $u^{1,1}$). Then ${}^{u_0}({}^{\sigma}y)$ contains $x$ in the $(d,d)$-th coordinate, $1\leq d\leq n$, and arbitrary elements in the coordinates
$(i',j')\ne(d,d)$, where $i'\leq d$ only varies over the rows of $B'$, but $j'$ varies over all columns $j'\geq d$ of $B'$ and also the columns to the right of $B'$, up to the rightmost column of $\Mat_{c\times2c}$ (this is the $(k+1)c$-th column as a matrix in $U$). Otherwise $B'$ is the bottom right block (which is $u^{2,2}$). Then ${}^{u_0}({}^{\sigma}y)$ contains $x$ in the $(d,d)$-th coordinate and may contain nontrivial coordinates for $(i',j')\ne(d,d)$, where $i'$ varies over the rows $i'\leq d$ of $B'$ and the rows above $B'$, up to the first row of
$\Mat_{c\times2c}$ (row $(k-2)c+1$ for matrices in $U$),
and $j'\geq d$ only varies over columns of $B'$. In both cases $\psi_U$ is trivial on all of the possibly nonzero coordinates $(i',j')$,
and the $(d,d)$-th coordinate is $x$, thus $\psi_U|_{{}^{u_0\sigma}Y}\ne1$.

If $c$ is odd we also consider $B=B'\in\Mat_{1\times2}$. Observe that now $\psi_U$ is trivial on all coordinates above or to the right of
$B_{1,2}$, and also on all coordinates above or to the right of $B_{1,1}$ except $B_{1,2}$. Hence if the nonzero coordinate $x$ of ${}^{\sigma}y$ is in $B_{1,2}$, $\psi_U|_{{}^{u_0\sigma}Y}\ne1$, but also if $x$ is in $B_{1,1}$ we have $\psi_U|_{{}^{u_0\sigma}Y}\ne1$, because multiplying $u_0({}^{\sigma}y)$ on the right by $u_0^{-1}$ leaves $B_{1,2}$ zero (when $c$ is odd, the $(kc,kc+1)$-th coordinate of any element of $N_H$ is zero).

Now because $\sigma\in (G,1)$, it immediately follows that $\psi_U$ is nontrivial on
${}^{\sigma^{-1}u_0\sigma}Y={}^{u^{-1}}Y$, completing the proof of the lemma.
\end{proof}

Re-denote $h=wu$ where $w$ and $u$ satisfy the properties of Proposition~\ref{proposition:structure of w u}. In particular $w$ defines the integer $0\leq l\leq n$.
\begin{proposition}\label{proposition:1st reduction of w}
We have $\mathcal{H}(h)=0$ unless
\begin{align}\label{eq:w_i first reduction}
&w_i=(1^{\lceil c/2\rceil},*^{n-l},1^{{l}}),\qquad\forall 1<i\leq k.
\end{align}
\end{proposition}
\begin{proof}
For $k=1$ there is nothing to prove, assume $k>1$. Write $w_2=(w_2',w_2'')$ with
$w_2'\in\{0,1\}^{\lceil c/2\rceil}$, $w_2''\in\{0,1\}^{n}$. If $w_2'$ is not of the form $(1^{{l}},*^{\lceil c/2\rceil-l})$, then $l>0$ (for $l=0$, $w_2'$ is automatically of the form $(1^{{l}},*^{\lceil c/2\rceil-l})=(*^{\lceil c/2\rceil})$). Let $Y<U$ be the subgroup
of elements with the middle $2(c+l)\times2(c+l)$ block of the form
\begin{align*}
\left(\begin{smallmatrix}
I_l&&&&&y&&yA_l'\\
&I_l&&&&&&&-A_ly'\\
&&I_{n-l}&&\\
&&&I_l&&&&&y'\\
&&&&I_{2c-2(n+l)}\\
&&&&&I_l\\
&&&&&&I_{n-l}\\
&&&&&&&I_l\\
&&&&&&&&I_l\end{smallmatrix}\right).
\end{align*}
Recall $\psi_U$ restricts to $\psi\circ\tr$ on the block $u^{2,2}$ (see \eqref{matrix:middle 4c block of u}). The block $yA_l'$ above occupies the bottom right $l\times l$ block of $u^{2,2}$.
Since there are no zero rows in $A_l$, there are no zero columns in $A_l'$. Hence for each $1\leq i\leq l$, the form $y\mapsto (yA_l')_{i,i}$ on $\Mat_l$ is not identically $0$. Then if one of the first $l$ coordinates of $w_2'$ is $0$, we can take a subgroup $Y_i<Y$ on which
$\psi_U|_{Y_i}\ne1$ and ${}^hY_i<U_P$, hence $\mathcal{H}(h)=0$ by \eqref{psi U nontrivial}. Thus
we can write $w_2'=(1^{l},*^{\lceil c/2\rceil-l})$ (whether $l>0$ or $l=0$).

If $w_2'\ne (1^{n},*^{\lceil c/2\rceil-n})$, one of the rows from the top left $n-l\times n-l$ block of $u^{2,2}$ is conjugated by $w$ into $U_P$.
Hence we can take a subgroup $Y<U$ such that $\psi_U|_Y\ne1$ and ${}^wY<U_P$, then
$\mathcal{H}(h)=0$ by \eqref{psi U nontrivial using only w}.

If $c$ is odd, $\psi_U$ restricts to a nontrivial character on the middle two coordinates $(u^3,u^4)$ of row $n+1$ in \eqref{matrix:middle 4c block of u}, and the columns of $u^3$ and $u^4$ are either swapped or remain unchanged by $w$.
Then if $w_2'\ne (1^{\lceil c/2\rceil})$, one of these coordinates is conjugated by $w$ into $U_P$, and if $Y<U$ is defined by this coordinate, we have $\psi_U|_Y\ne1$ and ${}^wY<U_P$. Thus $\mathcal{H}(h)=0$ by \eqref{psi U nontrivial using only w}. (Because $2c-2(n+l)\geq2$, ${}^uY=Y$ and we can also apply \eqref{psi U nontrivial} directly.) Thus
$w_2'=(1^{\lceil c/2\rceil})$ whether $c$ is even or odd. We proceed for all $c$.

Recall $\psi_U$ restricts to $\psi\circ\tr$ on the
top left $l\times l$ block of $u^{1,1}$. Since the first $c$ coordinates of $w$ are $\jmath_a(0^{c-l},1^l)$, $w$ permutes the columns of this block into columns in $U_P$, hence if
$w_2''\ne(*^{n-{l}},1^{{l}})$, we can again find $Y<U_P$ such that
$\psi_U|_Y\ne1$ and ${}^wY<U_P$, so that $\mathcal{H}(h)=0$ by \eqref{psi U nontrivial using only w}. Altogether
$w_2=(1^{\lceil c/2\rceil},*^{n-l},1^{l})$.

If $k=2$ we are done, assume $k>2$. We show $w_3=(1^{\lceil c/2\rceil},*^{n-l},1^{l})$. Recall $V_{(c^{k-1})}<U$.
Because $w_2=(1^{\lceil c/2\rceil},*^{n-l},1^{l})$, $w$ conjugates the last $\lceil c/2\rceil$ and first $l$ columns
of $v_{k-2,k-1}$ (see \S~\ref{representations} for this notation) into $U_P$. Hence if $w_3\ne(1^{\lceil c/2\rceil},*^{n-l},1^{l})$, a diagonal coordinate of one of the blocks inside $v_{k-2,k-1}$, namely the bottom right $\lceil c/2\rceil\times\lceil c/2\rceil$ block if $w_3\ne(1^{\lceil c/2\rceil},*^{n})$, or the top left
$l\times l$ block if $w_3\ne(*^{c-{l}},1^{{l}})$, is conjugated by $w$ into $U_P$, so that if $Y<U$ is generated by this coordinate, $\mathcal{H}(h)=0$ by \eqref{psi U nontrivial using only w}.
Proceeding in this manner for $3<i\leq k$, each time using $v_{k-i+1,k-i+2}$ and \eqref{psi U nontrivial using only w}, we deduce $w_i=(1^{\lceil c/2\rceil},*^{n-l},1^{l})$.
\end{proof}

For each $1<i\leq k$, since $w_i$ takes the form \eqref{eq:w_i first reduction}, we can uniquely identify a maximal integer
$0\leq d_{i-1}\leq n-l$ such that $w_i=(1^{\lceil c/2\rceil},*^{n-l-d_{i-1}},1^{l+d_{i-1}})$. By maximality $(*^{n-l-d_{i-1}})=(*^{n-l-d_{i-1}-1},0)$ (if $d_{i-1}<n-l$),
but the remaining coordinates are still undetermined. As we show next if $\mathcal{H}(h)\ne0$, we can replace $h$ by a representative for which $(*^{n-l-d_{i-1}})=(0^{n-l-d_{i-1}})$ (this may mean the integers $d_{i-1}$ are larger), and even fix an ascending order on $d_{i-1}$. Note that for $k=1$, the integers $d_{i-1}$ are undefined.
\begin{proposition}\label{proposition:2nd reduction of w}
We have $\mathcal{H}(h)=0$, unless $h\sim \hat{w}\hat{u}$ where
\begin{align}\label{eq:w_i second reduction}
&\hat{w}=\jmath_a(0^{c-l},1^{l},w_2,\ldots,w_k),\quad\forall 1<i\leq k, \, w_i=(1^{\lceil c/2\rceil},0^{n-l-d_{i-1}},1^{l+d_{i-1}}),\quad d_1\leq\ldots\leq d_{k-1},
\end{align}
and $\hat{u}$ satisfies the conditions of Proposition~\ref{proposition:structure of w u}, in particular
$\ell_0(\hat{u})$ takes the form \eqref{eq:second form u final} (and $A_l$ does not have any zero row).
\end{proposition}
\begin{proof}
Write $w_i=(1^{\lceil c/2\rceil},w_i',1^l)$ with $w'_i\in\{0,1\}^{n-l}$. The rightmost $d_{i-1}$ coordinates of $w'_i$ are $1$. We start with the following observation. Let $1\leq j\leq n-l$ and assume
$1<i_0\leq k$ is minimal such that $w_{i_0}'[j]$ (the $j$-th coordinate of $w_{i_0}'$) equals $1$. We claim
$\mathcal{H}(h)=0$ unless $w_i'[j]=1$ for all $i\geq i_0$.
Otherwise, assume $i>i_0$ is minimal with $w_i'[j]=0$. Write the top left $n\times n$ block of the block $v_{k-i+1,k-i+2}$ of $V_{(c^{k-1})}$ in the form
$\left(\begin{smallmatrix}   v^1 & v^2 \\   v^3 & v^4 \end{smallmatrix}\right)$ with $v^1\in\Mat_{l}$ and $v^4\in\Mat_{n-l}$.
Then a unipotent subgroup $Y$ containing coordinates from $v^4$ will satisfy $\psi_U|_Y\ne1$ and ${}^wY<U_P$, whence
$\mathcal{H}(h)=0$ by \eqref{psi U nontrivial using only w}.

We turn to show that we can sort the coordinates of $w_2,\ldots,w_k$ to obtain \eqref{eq:w_i second reduction}.
Identify $\GL_{n-l}$ with its natural image in the middle factor of the standard Levi subgroup
$\GL_l\times \GL_{n-l}\times \mathcal{G}_{c-2n}$ of $G$. Then $P\cap(\GL_{n-l},1)$ contains a full set of representatives for $W(\GL_{n-l})$.
Given such representative $g$, $h\sim h(g,1)^{-1}\sim ({}^{(g,1)}w)({}^{(g,1)}u)$ where
$\hat{u}={}^{(g,1)}u$ still satisfies the conditions of Proposition~\ref{proposition:structure of w u} and $\ell_0(\hat{u})=\ell_0(u)$ ($(g,1)$ commutes with \eqref{eq:second form u final}). Hence
one can use such conjugations to permute the entries in each $w_i$, while maintaining the prescribed structure of $u$. Using transpositions from $W(\GL_{n-l})$ we can permute each consecutive pair $(w_i'[j],w_i'[j+1])$. If
$w_i'[j]=w_i'[j+1]$, the conjugation has no affect on this pair.
Choose some $j$ such that there is a minimal $i_0$ with
$(w_{i_0}'[j],w_{i_0}'[j+1])=(1,0)$. If $j$ does not exist, then $w_i'=(0^{n-l-d_{i-1}},1^{d_{i-1}})$ for all $1<i\leq k$, and by what we have proved, if $\mathcal{H}(h)\ne0$, $d_1\leq\ldots\leq d_{k-1}$ so that \eqref{eq:w_i second reduction} holds. If $j$ exists, then
again by the above observation (assuming $\mathcal{H}(h)\ne0$),
for $i>i_0$, either $(w_i'[j],w_i'[j+1])=(1,0)$, in which case the order is swapped, or
$w_i'[j]=w_i'[j+1]=1$. Proceeding in this manner we obtain
\eqref{eq:w_i second reduction}.
\end{proof}

Re-denote $h=wu$, with $w$ and $u$ given by Proposition~\ref{proposition:2nd reduction of w}, in particular $w$ satisfies \eqref{eq:w_i second reduction}. Recall that in general if $Y<{}^hU\cap M_P$, ${}^{h^{-1}}Y<P_h$ and by definition any morphism in $\mathcal{H}(h)$ factors through $J_{Y,{}^{h}\psi_U^{-1}}(\rho)$ (see \S~\ref{Basic properties}).
We turn to compute ${}^hU\cap M_P$. To simplify the presentation we slightly alter $w$, using multiplication on the left by representatives of $W(M_P)$, which we identify with permutation matrices in $\GL_{kc}$. First,
we multiply $w$ on the left by $\diag(I_{(k-1)c},J_l,I_{c-l})$, this changes the innermost block $J_l$ into $I_l$ (see
\eqref{eq:weyl rep k=1}). Then for $w_i$, $1<i\leq k$, we multiply $w$ on the left by
\begin{align*}
\diag(I_{(k-i)c},J_{l+d_{i-1}},I_{n-l-d_{i-1}},J_{\lceil c/2\rceil},I_{c(i-1)}).
\end{align*}
For example if $k=2$,
\begin{align}\label{eq:example k=2}
w=\jmath_a
\left(\begin{smallmatrix}
&&&&&&&&I_{l+d_1}\\
&I_{n-l-d_1}&&&&&&&\\
&&&&&&I_{\lceil c/2\rceil}&&\\
&&&&&I_l&&&\\
&&&&I_{2(c-l)}&&&&\\
&&&\epsilon_0I_l&&&&&\\
&&\epsilon_0I_{\lceil c/2\rceil}&&&&&&\\
&&&&&&&I_{n-l-d_1}&\\
\epsilon_0I_{l+d_1}&&&&&&&&
\end{smallmatrix}\right).
\end{align}
For $1\leq j\leq k-1$, define $\gamma_j\in\GL_{kc}$ by
\begin{align*}
\gamma_j=&\diag(I_{n-l-d_{k-j}+\sum_{i=1}^{j-1}n-l-d_{k-i}},\left(\begin{smallmatrix}&I_{kc-\lceil c/2\rceil-(j-1)c-n}\\I_{\lceil c/2\rceil}\end{smallmatrix}\right),I_{l+d_{k-j}+\sum_{i=1}^{j-1}\lceil c/2\rceil+l+d_{k-i}})\\
&\times
\diag(I_{\sum_{i=1}^{j-1}n-l-d_{k-i}},\left(\begin{smallmatrix}&I_{kc-(j-1)c-l-d_{k-j}}\\I_{l+d_{k-j}}\end{smallmatrix}\right),I_{\sum_{i=1}^{j-1}\lceil c/2\rceil+l+d_{k-i}}).
\end{align*}
E.g.,
\begin{align*}
\gamma_1=\diag(I_{n-l-d_{k-1}},\left(\begin{smallmatrix}&I_{kc-\lceil c/2\rceil-n}\\I_{\lceil c/2\rceil}\end{smallmatrix}\right),I_{l+d_{k-1}})
\left(\begin{smallmatrix}&I_{kc-l-d_{k-1}}\\I_{l+d_{k-1}}\end{smallmatrix}\right).
\end{align*}
Further multiply $w$ on the left by $\gamma_{k-1}\cdot\ldots\cdot\gamma_1$ (henceforth we only use this form for $w$). For the computation of ${}^hU\cap M_P$ also note that ${}^hU={}^wU$. Now we see that ${}^hU\cap M_P=V_{\beta}$, where $\beta$ is the composition of $kc$ given by
\begin{align}\label{eq:beta}
\beta=(n-l-d_{k-1},\ldots,n-l-d_{1},c,\lceil c/2\rceil+l+d_{1},\ldots,\lceil c/2\rceil+l+d_{k-1}).
\end{align}
(The purpose of the elements $\gamma_i$ was to obtain an upper triangular ${}^hU\cap M_P$.)
The character ${}^{h}\psi_U$ is a character of $V_{\beta}$ by restriction,
denote $\psi_{V_{\beta}}={}^{h}\psi_U|_{V_{\beta}}$. We can not fully describe $\psi_{V_{\beta}}$ without determining $\ell(u)$, but the lemma below will provide the information we need.
First we describe ${}^{w\ell_0(u)}\psi_U|_{V_{\beta}}$. For $v\in V_{\beta}$ write
\begin{align}\label{eq:v in V beta}
v=\left(\begin{smallmatrix}I_{n-l-d_{k-1}}&b_1&\cdots\\&\ddots&\ddots&\\
&&I_{n-l-d_{1}}&b_{k-1}&\cdots\\&&&I_c&b_k&\cdots\\&&&&\ddots&\ddots\\&&&&&I_{\lceil c/2\rceil+l+d_{k-2}}&b_{2k-2}
\\&&&&&&I_{\lceil c/2\rceil+l+d_{k-1}}
\end{smallmatrix}\right).
\end{align}
Here
\begin{align*}
(b_i)_{1\leq i\leq 2k-2}=(b_1,\ldots,b_{k-2},b_{k-1},b_k,b_{k+1},\ldots,b_{2k-2})
\end{align*}
is a general element of the product
\begin{align*}
\prod_{j=k-1}^2\Mat_{n-l-d_j\times n-l-d_{j-1}}\times \Mat_{n-l-d_1\times c}
\times \Mat_{c\times \lceil c/2\rceil+l+d_1}
\times \prod_{j=1}^{k-2}\Mat_{\lceil c/2\rceil+l+d_j\times \lceil c/2\rceil+l+d_{j+1}}.
\end{align*}
Note that for $k=2$, $(b_i)_{1\leq i\leq 2k-2}\in \Mat_{n-l-d_1\times c}
\times \Mat_{c\times \lceil c/2\rceil+l+d_1}$. Then
\begin{align}\label{psi_U on V beta 0}
&{}^{w\ell_0(u)}\psi_U(v)=\psi(\sum_{j=k-1}^{2}\tr(b_{k-j}\left(\begin{smallmatrix}0_{d_j-d_{j-1} \times  n-l-d_j} \\ I_{n-l-d_{j}}\end{smallmatrix}\right))
+\tr(b_{k-1}\left(\begin{smallmatrix}0_{l+d_1\times n-l-d_1}\\ I_{n-l-d_1} \\ 0_{\lceil c/2\rceil\times n-l-d_1}\end{smallmatrix}\right))
\\&\quad\quad
-\tr(b_{k}\left(
\begin{smallmatrix}0&0&-\epsilon_0A_l&0&0\\0&\epsilon_0I_{n-l}&0&0&0
\\0&0&0&0&I_{c-2n}
\\0&0&0&0_{d_1\times n-l}&0
\\I_l&0&0&0&0
\end{smallmatrix}\right))
-\sum_{j=1}^{k-2}\tr(b_{k+j}\left(\begin{smallmatrix}I_{\lceil c/2\rceil}& 0_{\lceil c/2\rceil\times d_j+l} \\0_{d_{j+1}-d_j\times \lceil c/2\rceil} & 0_{d_{j+1}-d_j\times d_j+l} \\0_{d_j+l\times \lceil c/2\rceil} & I_{d_j+l}\end{smallmatrix}\right))).\nonumber
\end{align}
Here the sum $\sum_{j=k-1}^{2}$ is omitted if $k=2$; and if $c-2n=1$ ($0\leq c-2n\leq1$), the coordinate $I_{c-2n}=1$ initially depends on the constants $\epsilon_1,\epsilon_2$ (see \S~\ref{Doubling setup}, $2\epsilon_1\epsilon_2=1$), but we can use another conjugation of $w$ by an element of $M_P$ to fix this coordinate to be $1$ (without otherwise changing \eqref{psi_U on V beta 0}). Additionally, for $l=n$ and $A_l$ of rank $l$, the character \eqref{psi_U on V beta 0} belongs to the orbit of $\psi_k^{-1}$.

\begin{example}\label{eq:example k=2,3}
For $k=2$ and an even $a$, after multiplying \eqref{eq:example k=2} on the left by $\gamma_1$ we have
\begin{align}\label{eq:example k=2,3 final w for k=2}
&w=
\left(\begin{smallmatrix}
0&I_{n-l-d_1}&0&0&0&0&0&0&0&0\\
0&0&0&0&0&0&I_l&0&0&0\\
0&0&0&0&I_{c-l}&0&0&0&0&0\\
0&0&0&0&0&0&0&I_{\lceil c/2\rceil}&0&0\\
0&0&0&0&0&0&0&0&0&I_{l+d_1}\\
\epsilon_0I_{l+d_1}&0&0&0&0&0&0&0&0&0\\
0&0&\epsilon_0I_{\lceil c/2\rceil}&0&0&0&0&0&0&0\\
0&0&0&0&0&I_{c-l}&0&0&0&0\\
0&0&0&\epsilon_0I_l&0&0&0&0&0&0\\
0&0&0&0&0&0&0&0&I_{n-l-d_1}&0
\end{smallmatrix}\right),
\\&\nonumber
V_{\beta}=V_{(n-l-d_1,c,\lceil c/2\rceil+l+d_1)}=\{\left(\begin{smallmatrix}
                               I_{n-l-d_1} & b_1 & *  \\
                                & I_c &   b_2 \\
                                &  & I_{\lceil c/2\rceil+l+d_1}\end{smallmatrix}\right)\}
\end{align}
and $\psi_{V_{\beta}}$ depends only on $b_1$ and $b_2$. E.g., if $u=\ell_0(u)$, its restriction to $b_1$ is given by $\psi$ composed with the trace of the $n-l-d_1\times n-l-d_1$ block of $b_1$ starting at column $l+d_1+1$ of $b_1$. For $k=3$ and again an even $a$,
after multiplying $w$ on the left by $\gamma_2\gamma_1$ we obtain
\begin{align*}
&\left(\begin{smallmatrix}
0&I_{n-l-d_2}&0&0&0&0&0&0&0&0&0&0&0&0&0&0\\
0&0&0&0&I_{n-l-d_1}&0&0&0&0&0&0&0&0&0&0&0\\
0&0&0&0&0&0&0&0&0&I_l&0&0&0&0&0&0\\
0&0&0&0&0&0&0&I_{c-l}&0&0&0&0&0&0&0&0\\
0&0&0&0&0&0&0&0&0&0&I_{\lceil c/2\rceil}&0&0&0&0&0\\
0&0&0&0&0&0&0&0&0&0&0&0&I_{l+d_1}&0&0&0\\
0&0&0&0&0&0&0&0&0&0&0&0&0&I_{\lceil c/2\rceil}&0&0\\
0&0&0&0&0&0&0&0&0&0&0&0&0&0&0&I_{l+d_2}\\
\epsilon_0I_{l+d_2}&0&0&0&0&0&0&0&0&0&0&0&0&0&0&0\\
0&0&\epsilon_0I_{\lceil c/2\rceil}&0&0&0&0&0&0&0&0&0&0&0&0&0\\
0&0&0&\epsilon_0I_{l+d_1}&0&0&0&0&0&0&0&0&0&0&0&0\\
0&0&0&0&0&\epsilon_0I_{\lceil c/2\rceil}&0&0&0&0&0&0&0&0&0&0\\
0&0&0&0&0&0&0&0&I_{c-l}&0&0&0&0&0&0&0\\
0&0&0&0&0&0&\epsilon_0I_l&0&0&0&0&0&0&0&0&0\\
0&0&0&0&0&0&0&0&0&0&0&I_{n-l-d_1}&0&0&0&0\\
0&0&0&0&0&0&0&0&0&0&0&0&0&0&I_{n-l-d_2}&0\\
\end{smallmatrix}\right),
\\&
\beta=(n-l-d_2,n-l-d_1,c,\lceil c/2\rceil+l+d_1,\lceil c/2\rceil+l+d_2).
\end{align*}
\end{example}
\begin{remark}\label{remark:convenient computations of V beta}
It is convenient to compute $V_{\beta}$ in $2$ steps: first compute ${}^wU\cap M_P$ using
$w$ without the elements $\gamma_i$, e.g., \eqref{eq:example k=2},
then conjugate by these elements in order to obtain $V_{\beta}$.
\end{remark}
\begin{proposition}\label{proposition:wu_0 nontrivial implies h nontrivial orbit}
Assume $k>1$ and $l<n$. If $\mathcal{H}(h)\ne0$, $\psi_{V_{\beta}}$ belongs to the orbit of
\begin{align}\label{psi_U on V beta}
&v\mapsto
\psi(\sum_{j=k-1}^{2}\tr(b_{k-j}({*}_{n-l-d_{j-1} \times  n-l-d_j}))
+\tr(b_{k-1}\left(\begin{smallmatrix}{*}_{l+d_1 \times n-l-d_1}\\ I_{n-l-d_1} \\ {*}_{\lceil c/2\rceil\times n-l-d_1}\end{smallmatrix}\right))
\\&\quad\quad
-\tr(b_{k}\left(
\begin{smallmatrix}0&0&-\epsilon_0A_l&0&0\\0&\epsilon_0I_{n-l}&0&0&0
\\0&0&0&0&I_{c-2n}
\\0&0&0&0_{d_1\times n-l}&0
\\I_l&0&0&0&0
\end{smallmatrix}\right))
-\sum_{j=1}^{k-2}\tr(b_{k+j}\left(\begin{smallmatrix}I_{\lceil c/2\rceil}& {0}_{\lceil c/2\rceil\times d_j+l} \\{*}_{d_{j+1}-d_j\times \lceil c/2\rceil} & {*}_{d_{j+1}-d_j\times d_j+l}\\{*}_{d_j+l\times \lceil c/2\rceil} & {*}_{d_j+l}\end{smallmatrix}\right))).\nonumber
\end{align}
Here $*$ means undetermined block entries. When $\ell(u)$ is the identity element, all coordinates were computed above and \eqref{psi_U on V beta} coincides with \eqref{psi_U on V beta 0}.
\end{proposition}
\begin{proof}
We introduce notation for blocks of unipotent matrices in $M_Q$ and $U$.
Recall from Lemma~\ref{lemma:easier condition on psiU} that $\psi_U$ is defined by $k-2$ blocks $B_i\in\Mat_c$, $1\leq i\leq k-2$,
$2$ blocks $B_1',B_2'\in\Mat_n$ and when $c$ is odd also by $B''\in\Mat_{1\times2}$.
Set $d_0=0$ and $d_k=d_{k-1}$. For each $1\leq i\leq k-2$, $B_i$ is further divided into subblocks by writing it as the upper right block of
\begin{align*}
\left(\begin{smallmatrix}
  I_{l+d_{k-i-1}}&&&&B_i^{1,1}&B_i^{1,2}&B_i^{1,3}&B_i^{1,4}\\
  &I_{d_{k-i}-d_{k-i-1}}&&&B_i^{2,1}&B_i^{2,2}&B_i^{2,3}&B_i^{2,4}\\
  &&I_{n-l-d_{k-i}}&&B_i^{3,1}&B_i^{3,2}&B_i^{3,3}&B_i^{3,4}\\
  &&&I_{\lceil c/2\rceil}&B_i^{4,1}&B_i^{4,2}&B_i^{4,3}&B_i^{4,4}\\
  &&&&I_{l+d_{k-i-1}}\\
  &&&&&I_{d_{k-i}-d_{k-i-1}}\\
  &&&&&&I_{n-l-d_{k-i}}\\
  &&&&&&&I_{\lceil c/2\rceil}
\end{smallmatrix}\right).
\end{align*}
The blocks $B_1',B_2'$ are contained in the following blocks:
\begin{align*}
\left(\begin{smallmatrix}
  I_l&&&&{B'}_1^{1,1}&{B'}_1^{1,2}&{B'}_1^{1,3}\\
  &I_{d_1}&&&{B'}_1^{2,1}&{B'}_1^{2,2}&{B'}_1^{2,3}\\
  &&I_{n-l-d_1}&&{B'}_1^{3,1}&{B'}_1^{3,2}&{B'}_1^{3,3}\\
  &&&I_{\lceil c/2\rceil}&{B'}_1^{4,1}&{B'}_1^{4,2}&{B'}_1^{4,3}\\
  &&&&I_l\\
  &&&&&I_{d_1}\\
  &&&&&&I_{n-l-d_1}
\end{smallmatrix}\right),
\left(\begin{smallmatrix}
  I_{l+d_1}&&&&&{B'}_2^{1,1}&{B'}_2^{1,2}&{B'}_2^{1,3}\\
  &I_{n-l-d_1}&&&&{B'}_2^{2,1}&{B'}_2^{2,2}&{B'}_2^{2,3}\\
  &&I_{\lceil c/2\rceil-l}&&&{B'}_2^{3,1}&{B'}_2^{3,2}&{B'}_2^{3,3}\\
  &&&I_{l}&&{B'}_2^{4,1}&{B'}_2^{4,2}&{B'}_2^{4,3}\\
  &&&&I_{2c-n-l}\\
  &&&&&I_l\\
  &&&&&&I_{n-l}\\
  &&&&&&&I_{l}
\end{smallmatrix}\right).
\end{align*}
If $c$ is odd, we also have the $c\times 2$ block containing $B''$ which we write in the form
\begin{align*}
&\left(\begin{smallmatrix}
{B''}^{1}\\{B''}^{2}\\{B''}^{3}\\{B''}^{4}
\end{smallmatrix}\right),
{B''}^{1}\in\Mat_{l+d_1\times2},{B''}^{2}\in\Mat_{n-l-d_1\times2},{B''}^{3}\in\Mat_{1\times2},{B''}^{4}\in\Mat_{n\times2}.
\end{align*}
For $1\leq i\leq 4$, $B''^{i}=(B''^{i,1},B''^{i,2})$.
In terms of the blocks $B_i$, $B'_i$ and $B''$, $\psi_U$ is given by
\begin{align}\label{eq:blocks of psi_U}
\psi(\sum_{i=1}^{k-2}\sum_{j=1}^4\tr(B_i^{j,j})+\sum_{j=1}^3\tr({B'}_1^{j,j})+\tr(\left(\begin{smallmatrix}0_{n-l\times c-2n} & I_{n-l}\end{smallmatrix}\right){B'}_2^{3,2})+\tr({B'}_2^{4,3})+
{B''}^{3}\left(\begin{smallmatrix}\epsilon_1\\-\epsilon_2\end{smallmatrix}\right)).
\end{align}

Let $\mathscr{M}_P$, $\mathscr{U}_P$ and $\mathscr{U}_P^-$ denote the lists of blocks $B_i^{t,t'},{B'}_i^{t,t'},B''^{t,t'}$
conjugated by $w$ into $M_P$, $U_P$ and $U_P^-$, respectively (these can still be computed using \eqref{eq:w_i second reduction},
$w$ differs from \eqref{eq:w_i second reduction} by an element of $M_P$). If $c$ is odd, let $a_0\in\{1,2\}$ be the column of $B''$ which $w$ conjugates into column $kc+1$, it consists of the blocks $(B''^{1,a_0},B''^{2,a_0},B''^{3,a_0},B''^{4,a_0})$. We see that
\begin{align*}
\mathscr{M}_P=&\{B_i^{1,1},B_i^{1,4},B_i^{2,1},B_i^{2,4},B_i^{3,2},B_i^{3,3},B_i^{4,1},B_i^{4,4}:1\leq i\leq k-2\}\\&\coprod
\{{B'}_1^{1,1},{B'}_1^{2,1},{B'}_1^{3,2},{B'}_1^{3,3},{B'}_1^{4,1},{B'}_2^{1,1},{B'}_2^{1,2},{B'}_2^{2,3},{B'}_2^{3,1},{B'}_2^{3,2},{B'}_2^{4,1},{B'}_2^{4,2}\},\\
&\coprod\{{B''}^{1,a_0},{B''}^{2,3-a_0},{B''}^{3,a_0},{B''}^{4,a_0}\},\\
\mathscr{U}_P=&\{B_i^{3,1},B_i^{3,4}:1\leq i\leq k-2\}\coprod\{{B'}_1^{3,1},{B'}_2^{2,1},{B'}_2^{2,2},{B''}^{2,a_0}\},\\
\mathscr{U}_P^-=&\{B_i^{1,2},B_i^{1,3},B_i^{2,2},B_i^{2,3},B_i^{4,2},B_i^{4,3}:1\leq i\leq k-2\}\\&\coprod\{{B'}_1^{1,2},{B'}_1^{1,3},{B'}_1^{2,2},{B'}_1^{2,3},{B'}_1^{4,2},{B'}_1^{4,3},{B'}_2^{1,3},{B'}_2^{3,3},{B'}_2^{4,3},
{B''}^{1,3-a_0},{B''}^{3,3-a_0},{B''}^{4,3-a_0}\}.
\end{align*}
Since ${}^{\ell(\sigma)}\ell(u)\in M_Q\cap N_H$ with $\sigma\in (G,1)$, we can write
$\ell(u)=\diag(z_{1},\ldots,z_{k-1})$, for $z_i={}^{w_{\sigma}}v_i$ where $w_{\sigma}\in W(\GL_c)$ corresponds to the projection of
$(\sigma,1)^{-1}$ into the $i$-th copy of $\GL_c$ and $v_i\in N_{\GL_c}$. Note that if we write a general element of ${}^{w_{\sigma}}N_{\GL_c}$ in the form
\begin{align*}
\left(\begin{smallmatrix}
X_1&X_2&X_3\\
X_4&X_5&X_6\\
X_7&X_8&X_9
\end{smallmatrix}\right)
\end{align*}
where $X_1,X_5$ and $X_9$ are square matrices (of arbitrary sizes), then $X_1,X_5,X_9$ are already invertible, and so are
$\left(\begin{smallmatrix}I&X_2\\-X_4&I\end{smallmatrix}\right)$ and $\left(\begin{smallmatrix}
I&X_6\\-X_8&I\end{smallmatrix}\right)$, whence $I+X_2X_4$, $I+X_4X_2$, $I+X_6X_8$ and $I+X_8X_6$ are also invertible ($X_i,X_j$ need not be square matrices).

Since the left coset of $w$ in $W(M_P)\backslash W(H)$ is still represented by \eqref{eq:w_i second reduction}, we can write
$z_i=z_i'm_i$ where ${}^w\diag(z_1',\ldots,z_{k-1}',I_{2c},z_1'^*,\ldots,z_{k-1}'^*)\in M_P$ and
\begin{align*}
&m_i=
\left(\begin{smallmatrix}
    I_{l+d_{k-i}}+M_i^1M_i^2&M_i^1&0\\
    M_i^2&I_{n-l-d_{k-i}}+M_i^3M_i^4&M_i^3\\
  0&M_i^4&I_{\lceil c/2\rceil}
\end{smallmatrix}\right)\in\GL_c,\\
&I_{l+d_{k-i}}+M_i^1M_i^2\in\GL_{l+d_{k-i}},\qquad
I_{n-l-d_{k-i}}+M_i^3M_i^4\in \GL_{n-l-d_{k-i}}.
\end{align*}
These matrices are invertible because $m_i={}^{w_{\sigma}}v_i'$ where $v_i'\in N_{\GL_c}$.
We have
\begin{align*}
&m_i^{-1}=
\left(\begin{smallmatrix}
    I_{l+d_{k-i}}&-M_i^1&M_i^1M_i^3\\
    -M_i^2&I_{n-l-d_{k-i}}+M_i^2M_i^1&-(I_{n-l-d_{k-i}}+M_i^2M_i^1)M_i^3\\
  M_i^4M_i^2&-M_i^4(I_{n-l-d_{k-i}}+M_i^2M_i^1)&I_{\lceil c/2\rceil}+M_i^4(I_{n-l-d_{k-i}}+M_i^2M_i^1)M_i^3
\end{smallmatrix}\right).
\end{align*}
Since $h\sim ph$ for any $p\in P$, we can already assume $z_i=m_i$.

We show that $\psi_{V_{\beta}}$ belongs to the orbit of a character whose restriction to the blocks $b_{k-1},b_k,b_{k+1},\ldots,b_{2k-2}$ agrees with \eqref{psi_U on V beta}, otherwise $\mathcal{H}(h)=0$. This will complete the proof. To this end it suffices to compute ${}^{u}\psi_U$ on the blocks of $U$ conjugated by $w$ into $b_{k-1},b_k,b_{k+1},\ldots,b_{2k-2}$. The contribution of $\ell_0(u)$ is simple to compute and was essentially given in \eqref{psi_U on V beta 0}. To determine ${}^{\ell(u)}\psi_U$ (thereby ${}^u\psi_U$) we compute
\begin{align*}
m_{k-1}^{-1}B'_1,\qquad m_{k-1}^{-1}B'_2,\qquad m_{k-1}^{-1}B'',\qquad
m_i^{-1}B_im_{i+1},\qquad\forall 1\leq i\leq k-2.
\end{align*}

Columns $l+d_1+1,\ldots, n$ of $b_{k-1}$ (the only columns of $b_{k-1}$ where \eqref{psi_U on V beta} is determined) consist of the block $B_1'^{3,3}$, conjugated to $b_{k-1}$ by $w$ (other columns are conjugated from $B_1'^{3,2},B_2'^{2,3}$ and columns between the columns of $u^{1,1}$ and $u^{2,2}$). The coordinates of $b_k$ are uniquely defined by
\begin{align*}
B_1'^{1,1},B_1'^{2,1},B_1'^{4,1},B_2'^{1,1},B_2'^{1,2},B_2'^{3,1},B_2'^{3,2},B_2'^{4,1},B_2'^{4,2},
{B''}^{1,a_0},{B''}^{2,a_0},{B''}^{3,3-a_0},{B''}^{4,a_0}
\end{align*}
and by additional $l+d_1+\lceil c/2\rceil \times n-l$ coordinates appearing to the left of $B_2'^{1,1},B_2'^{3,1},B_2'^{4,1}$ ($\psi_U$ and ${}^{\ell_0(u)}\psi_U$ are trivial on the corresponding columns, thereby also ${}^u\psi_U$ because multiplying on the left by $m_{k-1}^{-1}$ can not introduce a character on a column where ${}^{\ell_0(u)}\psi_U$ was trivial, so we do not provide notation for these), as well as the form defining $H$. Note that $B''$ is omitted if $c$ is even.

When we multiply $m_{k-1}^{-1}B'_1$ we see that if the top $l$ rows of $M_{k-1}^1$ are nonzero, ${}^u\psi_U$ is nontrivial on $B_1'^{3,1}\in\mathscr{U}_P$ and then $\mathcal{H}(h)=0$ by \eqref{psi U nontrivial}. Hence we can assume
the top $l$ rows of $M_{k-1}^1$ are $0$, which implies ${}^u\psi_U$ is trivial on the coordinates of $b_k$ obtained from
$B_1'$, namely $B_1'^{1,1},B_1'^{2,1},B_1'^{4,1}$ ($\psi_U$ and ${}^{\ell_0(u)}\psi_U$ are also trivial there).
Additionally ${}^u\psi_{U}$ restricts to
$\psi(\tr((I_{n-l-d_1}+M_{k-1}^2M_{k-1}^1)B_1'^{3,3}))$ on $B_1'^{3,3}$, and since
\begin{align*}
{}^w\diag(I_{(k-2)c+l+d_1},I_{n-l-d_1}+M_{k-1}^2M_{k-1}^1,I_{2(\lceil c/2\rceil+c)},(I_{n-l-d_1}+M_{k-1}^2M_{k-1}^1)^*,I_{l+d_1+(k-2)c})\in M_P,
\end{align*}
$\psi_{V_{\beta}}$ belongs to the orbit
of a character which agrees with \eqref{psi_U on V beta} on $b_{k-1}$ and the coordinates of $b_k$ conjugated from $B_1'$.

The character ${}^{\ell_0(u)}\psi_U$ is given on the blocks of $B'_2$, which $w$ conjugates into $b_k$, by
\begin{align*}
\psi(\tr(\varphi_k {B'_2}^{\circ})),\qquad
\varphi_k=\left(
\begin{smallmatrix}
0_{l\times l+d_1}&0_{l\times n-l-d_1}&0_{l\times \lceil c/2\rceil-l}& A(X)\\
0_{n-l\times l+d_1}&0_{n-l\times n-l-d_1}&\left(\begin{smallmatrix}0_{n-l\times c-2n} & I_{n-l}\end{smallmatrix}\right)
& 0_{n-l\times l}
\end{smallmatrix}\right).
\end{align*}
Here ${B'_2}^{\circ}$ is the $c\times n$ block consisting of ${B'_2}^{t,t'}$ with $1\leq t\leq 4$ and $1\leq t'\leq 2$ (all of these blocks except for $t=2$ are conjugated into $b_k$). Multiplying $\varphi_km_{k-1}^{-1}$ we deduce $\mathcal{H}(h)=0$, unless
the product of $\varphi_k$ and columns $l+d_1+1,\ldots,n$ of $m_{k-1}^{-1}$
defines a trivial character on
$({B'_2}^{2,1},{B'_2}^{2,2})\in\mathscr{U}_P$, which amounts to
\begin{align*}
\left(
\begin{smallmatrix}
0_{l\times \lceil c/2\rceil-l}& A(X)\\
\left(\begin{smallmatrix}0_{n-l\times c-2n} & I_{n-l}\end{smallmatrix}\right)& 0_{n-l\times l}
\end{smallmatrix}\right)
(-M_{k-1}^4(I_{n-l-d_1}+M_{k-1}^2M_{k-1}^1))=0_{\lceil c/2\rceil}.
\end{align*}
Hence the product of $\varphi_k$ and last $\lceil c/2\rceil$ columns of $m_{k-1}^{-1}$ equals
\begin{align*}
\left(
\begin{smallmatrix}
0_{l\times \lceil c/2\rceil-l}& A(X)\\
\left(\begin{smallmatrix}0_{n-l\times c-2n} & I_{n-l}\end{smallmatrix}\right)& 0_{n-l\times l}
\end{smallmatrix}\right)
(I_{\lceil c/2\rceil}+M_{k-1}^4(I_{n-l-d_1}+M_{k-1}^2M_{k-1}^1)M_{k-1}^3)=\left(
\begin{smallmatrix}
0_{l\times \lceil c/2\rceil-l}& A(X)\\
\left(\begin{smallmatrix}0_{n-l\times c-2n} & I_{n-l}\end{smallmatrix}\right)& 0_{n-l\times l}
\end{smallmatrix}\right),
\end{align*}
thus ${}^u\psi_U$ agrees with $\psi_{U}$ on the blocks contained in $B'_2$.

If $c$ is odd, the restriction of ${}^{u}\psi_U$ to $B''$ is given by
\begin{align*}
\psi(\tr(\left(
\begin{smallmatrix}0_{2\times n}&\left(\begin{smallmatrix}\epsilon_1\\-\epsilon_2\end{smallmatrix}\right)&0_{2\times n}\end{smallmatrix}\right)m_{k-1}^{-1}B'')).
\end{align*}
Since $B''^{2,a_0}\in\mathscr{U}_P$ and $\epsilon_1\epsilon_2\ne0$, we deduce the first row of
$-M_{k-1}^4(I_{n-l-d_1}+M_{k-1}^2M_{k-1}^1)$ is $0$, and because $I_{n-l-d_1}+M_{k-1}^2M_{k-1}^1$ is invertible we obtain that
the first row of $M_{k-1}^4$ is $0$. Then the first row of $m_{k-1}^{-1}$ is $\left(\begin{smallmatrix}0_{n}&1&0_n\end{smallmatrix}\right)$, whence ${}^u\psi$ and $\psi_U$ agree on $B''$.

Altogether we have shown that $\psi_{V_{\beta}}$ belongs to the orbit
of a character which agrees with \eqref{psi_U on V beta} on $b_{k-1}$ and $b_k$.

Consider $b_{k+i}$, $1\leq i\leq k-2$. The coordinates of $b_{k+i}$ are uniquely defined by the blocks
\begin{align*}
B_{k-i-1}^{1,1},B_{k-i-1}^{1,4},B_{k-i-1}^{2,1},B_{k-i-1}^{2,4},B_{k-i-1}^{4,1},B_{k-i-1}^{4,4}.
\end{align*}
More precisely if we denote for $X\in\Mat_{a\times b}$, $X'=-J_b{}^tXJ_a$,
\begin{align}\label{eq:b k+1}
b_{k+i}=\left(\begin{smallmatrix}
          (B_{k-i-1}^{4,4})' & (B_{k-i-1}^{2,4})' & (B_{k-i-1}^{1,4})' \\
          (B_{k-i-1}^{4,1})' & (B_{k-i-1}^{2,1})' & (B_{k-i-1}^{1,1})'
        \end{smallmatrix}\right).
\end{align}
We multiply $m_{k-i-1}^{-1}B_{k-i-1}m_{k-i}$. Since
$\psi_U$ restricts to $\psi\circ\tr$ on $B_{k-i-1}$, the restriction of ${}^u\psi_U$
to $B_{k-i-1}$ is given by $\psi(\tr(m_{k-i}m_{k-i-1}^{-1}B_{k-i-1}))$. If this restriction is nontrivial on
$B_{k-i-1}^{3,1},B_{k-i-1}^{3,4}\in\mathscr{U}_P$, we obtain $\mathcal{H}(h)=0$.

On $B_{k-i-1}^{3,4}$, ${}^u\psi_U$ is given by the product of the last $\lceil c/2\rceil$ rows of
$m_{k-i}$ and columns $l+d_{i+1}+1,\ldots,n$ of $m_{k-i-1}^{-1}$; $\mathcal{H}(h)=0$ unless this product vanishes:
\begin{align*}
\left(\begin{smallmatrix}
0_{\lceil c/2\rceil\times l+d_i}&M_{k-i}^4&I_{\lceil c/2\rceil}
\end{smallmatrix}\right)
\left(\begin{smallmatrix}-M^1_{k-i-1} \\ I_{n-l-d_{i+1}}+ M^2_{k-i-1}M^1_{k-i-1}\\-M^4_{k-i-1}(I_{n-l-d_{i+1}}+ M^2_{k-i-1}M^1_{k-i-1})\end{smallmatrix}\right)=0.
\end{align*}
Thus the product of the last $\lceil c/2\rceil$ rows of $m_{k-i}$ and
last $\lceil c/2\rceil$ columns of $m_{k-i-1}^{-1}$ is
\begin{align*}
\left(\begin{smallmatrix}
0_{\lceil c/2\rceil\times l+d_i}&M_{k-i}^4&I_{\lceil c/2\rceil}
\end{smallmatrix}\right)
\left(\begin{smallmatrix}M^1_{k-i-1}M^3_{k-i-1} \\ -(I_{n-l-d_{i+1}}+ M^2_{k-i-1}M^1_{k-i-1})M^3_{k-i-1}\\I_{\lceil c/2\rceil}+M^4_{k-i-1}(I_{n-l-d_{i+1}}+ M^2_{k-i-1}M^1_{k-i-1})M^3_{k-i-1}\end{smallmatrix}\right)=I_{\lceil c/2\rceil}.
\end{align*}
This means that the restriction of ${}^u\psi_U$ to $B_{k-i-1}^{4,4}$, which corresponds to the bottom right $\lceil c/2\rceil\times\lceil c/2\rceil$ block of $m_{k-i}m_{k-i-1}^{-1}$, is $\psi\circ\tr$, so that it agrees with $\psi_U$ on this block.

On $B_{k-i-1}^{3,1}$, ${}^u\psi_U$ is defined by the product of the first $l+d_i$ rows of
$m_{k-i}$ and columns $l+d_{i+1}+1,\ldots,n$ of $m_{k-i-1}^{-1}$. Then $\mathcal{H}(h)=0$ unless
\begin{align*}
\left(\begin{smallmatrix}I_{l+d_i}+M_{k-i}^1M_{k-i}^2&M_{k-i}^1&0_{l+d_i\times\lceil c/2\rceil}\end{smallmatrix}\right)
\left(\begin{smallmatrix}-M^1_{k-i-1} \\ I_{n-l-d_{i+1}}+ M^2_{k-i-1}M^1_{k-i-1}\\ -M^4_{k-i-1}(I_{n-l-d_{i+1}}+ M^2_{k-i-1}M^1_{k-i-1})\end{smallmatrix}\right)=0.
\end{align*}
Hence
\begin{align*}
\left(\begin{smallmatrix}I_{l+d_i}+M_{k-i}^1M_{k-i}^2&M_{k-i}^1&0_{l+d_i\times\lceil c/2\rceil}\end{smallmatrix}\right)
\left(\begin{smallmatrix}M^1_{k-i-1}M^3_{k-i-1} \\ -(I_{n-l-d_{i+1}}+ M^2_{k-i-1}M^1_{k-i-1})M^3_{k-i-1}\\I_{\lceil c/2\rceil}+M^4_{k-i-1}(I_{n-l-d_{i+1}}+ M^2_{k-i-1}M^1_{k-i-1})M^3_{k-i-1}\end{smallmatrix}\right)=0.
\end{align*}
Therefore the restrictions of ${}^u\psi_U$ and $\psi_U$ to $B_{k-i-1}^{4,1}$, which correspond to the
top right $l+d_i\times\lceil c/2\rceil$ block of $m_{k-i}m_{k-i-1}^{-1}$, are both trivial.

Finally using \eqref{eq:b k+1} and noting that the leftmost $l$ columns of $X$ are the bottom $l$ rows of $X'$ (and the entries are permuted),
we find that ${}^u\psi_U$ is given on the blocks which $w$ conjugates into $b_{k+i}$ by
\begin{align*}
&\psi(\tr(\left(\begin{smallmatrix}
          (B_{k-i-1}^{4,4})' & (B_{k-i-1}^{2,4})' & (B_{k-i-1}^{1,4})' \\
          (B_{k-i-1}^{4,1})' & (B_{k-i-1}^{2,1})' & (B_{k-i-1}^{1,1})'
        \end{smallmatrix}\right)
        \left(\begin{smallmatrix}
          I_{\lceil c/2\rceil} & 0_{\lceil c/2\rceil\times l+d_i}\\
          {*}_{l+d_{i+1}\times\lceil c/2\rceil} & {*}_{l+d_{i+1}\times l+d_i} & \\
        \end{smallmatrix}\right))).
\end{align*}
We conclude $\psi_{V_{\beta}}$ belongs to the orbit of \eqref{psi_U on V beta}.
\end{proof}
\begin{proposition}\label{proposition:d_1 < n-l}
Assume $d_1<n-l$ (in particular $k>1$ and $l<n$, because $d_1\geq0$). Then $J_{V_{\beta},\psi_{V_{\beta}}^{-1}}(\rho)=0$, in particular $\mathcal{H}(h)=0$.
\end{proposition}
\begin{proof}
Any morphism in $\mathcal{H}(h)$ factors through $J_{V_{\beta},\psi_{V_{\beta}}^{-1}}(\rho)$. We show $J_{V_{\beta},\psi_{V_{\beta}}^{-1}}(\rho)=0$. Suppose otherwise. The subgroup $V_{\beta}$ and character $\psi_{V_{\beta}}^{-1}$ define a degenerate Whittaker model in the sense of
\cite{MW3,GGS}. The character $\psi_{V_{\beta}}^{-1}$ uniquely defines a nilpotent element ${}^t\varphi\in\Mat_{kc}$ such that $\psi_{V_{\beta}}^{-1}(v)=\psi(\tr(v({}^t\varphi)))$ for all $v\in V_{\beta}$. Then $\varphi\in \Mat_{kc}$ is an upper triangular nilpotent matrix.
We prove $\varphi$ is nilpotent of order at least $k+1$. By \cite[Theorem~E]{GGS}, the orbit of $\varphi$ belongs to the closure of the wave-front set $\mathrm{WF}(\rho)$ of $\rho$, but this orbit is greater than or non-comparable with $(k^c)$, contradicting the fact that $\rho$ is $(k,c)$. When $\pi_2$ is supercuspidal (in particular, the field is non-archimedean) and $\rho$ is not necessarily of finite length, we derive the same contradiction from \cite[Theorem~A]{GGS}.

By Proposition~\ref{proposition:wu_0 nontrivial implies h nontrivial orbit}, we can assume $\psi_{V_{\beta}}$ is given by \eqref{psi_U on V beta}. Since $\varphi+I_{kc}\in V_{\beta}$, we let $b_1,\ldots,b_{2k-2}$ denote the blocks of $\varphi$ above the principal diagonal
(see \eqref{eq:v in V beta}). These can be read off \eqref{psi_U on V beta}, namely $b_i$ is the transpose of the block appearing to the right of $b_i$ in \eqref{psi_U on V beta} (one should include the signs appearing in \eqref{psi_U on V beta} before $\tr{}$). E.g., $b_{k-1}=-\left(\begin{smallmatrix}{*}_{n-l-d_1 \times l+d_1}& I_{n-l-d_1} & {*}_{n-l-d_1\times \lceil c/2\rceil}\end{smallmatrix}\right)$.

We apply a sequence of conjugations to $\varphi$, conjugating $k$ nonzero coordinates from $\varphi$, one coordinate from each block
$b_{k-1},b_k,\ldots,b_{2k-2}$: since $d_1<n-l$, $\psi_{V_{\beta}}$ is nontrivial on $b_{k-1}$, so that the block $b_{k-1}$ of $\varphi$ is nonzero; and because $k>1$, these are $k$ blocks.

The only nonzero blocks of $\varphi$ are the blocks $b_{1},\ldots,b_{2k-2}$, and these blocks contain nonzero entries at the coordinates defined by \eqref{psi_U on V beta}. Define the following partial sums $d^j$ of the integers appearing in the composition $\beta$, from right to left (see \eqref{eq:beta}):
\begin{align*}
&d^j=\begin{cases}
\sum_{i=1}^j\lceil c/2\rceil+l+d_{k-i}& 1\leq j\leq k-1,\\
c+d^{k-1}& j=k,\\
n-l-d_1+d^k& j=k+1.
\end{cases}
\end{align*}
First we conjugate $\varphi$ by
\begin{align*}
\varepsilon_1=\diag(I_{kc-d^{k-1}-\lceil c/2\rceil-1},\left(\begin{smallmatrix}&&1\\&I_{\lceil c/2\rceil-1}&\\1\end{smallmatrix}\right),I_{d^{k-1}}).
\end{align*}
Note that $\varepsilon_1$ normalizes $V_{\beta}$.
The $(n,n)$-th coordinate of $b_k$, which is $\epsilon_0=\pm1$ because $l<n$, becomes the $(c,n)$-th coordinate, and the $(n-l-d_1,n)$-th coordinate of $b_{k-1}$, which is $-1$, becomes the $(n-l-d_1,c)$-th coordinate --- the bottom right coordinate.
Both of these coordinates are independent of the blocks where $\psi_{V_{\beta}}$ is undetermined (denoted $*$ in \eqref{psi_U on V beta}), and are the only nonzero entries on their columns.
We can further conjugate ${}^{\varepsilon_1}\varphi$ by an element
of the group
\begin{align*}
\{\diag(I_{(\sum_{i=1}^{k-1}n-l-d_i)+c},
\left(\begin{smallmatrix}b\\&I_{l+d_1}\end{smallmatrix}\right),\ldots,\left(\begin{smallmatrix}b\\&I_{l+d_{k-1}}\end{smallmatrix}\right)):b\in\GL_{\lceil c/2\rceil}\}<M_{\beta},
\end{align*}
to take the $(c,n)$-th coordinate of $b_k$ into the $(c,1)$-th coordinate, without affecting any of the blocks $b_{k-1},\ldots,b_{2k-2}$ of ${}^{\varepsilon_1}\varphi$ except the block $b_k$ (the diagonal embedding of $b\in\GL_{n}$ instead of $b\in\GL_{\lceil c/2\rceil}$ in each of the last $k-1$ blocks is sufficient). Now let
\begin{align*}
\varepsilon_2=\diag(I_{(\sum_{i=1}^{k-1}n-l-d_i)+c},
\left(\begin{smallmatrix}&&1\\&I_{\lceil c/2\rceil+l+d_1-2}&\\1\end{smallmatrix}\right),\ldots,
\left(\begin{smallmatrix}&&1\\&I_{\lceil c/2\rceil+l+d_{k-1}-2}&\\1\end{smallmatrix}\right))\in M_{\beta},
\end{align*}
where if $\lceil c/2\rceil+l+d_{k-i}=1$, the corresponding block of size
$1+(\lceil c/2\rceil+l+d_{k-i}-2)+1$ is $I_1$.
Again $\varepsilon_2$ normalizes $V_{\beta}$.
Conjugating ${}^{\varepsilon_1}\varphi$ by $\varepsilon_2$, the $(c,1)$-th coordinate of $b_k$ is taken into the $(c,\lceil c/2\rceil+l+d_1)$-th coordinate (the bottom right coordinate), and the top left coordinate of $b_{k+i}$, which is independent of the undetermined blocks and is the only nonzero coordinate on its column, is taken into the bottom right coordinate for each $1\leq i\leq k-2$. We conclude that the bottom right coordinate of each of the blocks $b_{k-1},\ldots,b_{2k-2}$ of ${}^{\varepsilon_2\varepsilon_1}\varphi$ is nonzero ($-1$ for $b_{k-1}$, $\epsilon_0$ for $b_k$, $1$ for all other blocks), and in each of these blocks, the bottom right coordinate is the only nonzero entry in its column. Therefore $\varphi$ is nilpotent of order at least $k+1$.
\end{proof}
\begin{example}
Consider $c=k=2$, $G=\Sp_2$, $H=\Sp_8$ and $l=0$. Then $0\leq d_1\leq n-l=1$. For $d_1=0$, $w\in M_P(0^2,w_2)$ with $w_2=(1,0)$.
One can take
\begin{align*}
w=\left(\begin{smallmatrix}
                  1 &  &  &  &    \\
                   &  &  &   1 &  \\
                   &  & I_4 &  &    \\
                   & -1 &  &  &    \\
                   &  &  &  &   1
                \end{smallmatrix}\right).
\end{align*}
Then
\begin{align*}
u=\left(\begin{smallmatrix}I_2&x&y\\&I_4&x'\\&&I_2\end{smallmatrix}\right),\qquad \psi_U(u)=\psi(x_{1,1}+x_{2,4}),\qquad
{}^wU\cap M_P=\left\{\left(\begin{smallmatrix}
                               1 & y_{1,1} & x_{1,1} & x_{1,2} \\
                                & 1 &  &  \\
                                & x_{2,4} & 1 &  \\
                                & x_{2,3} &  & 1
                             \end{smallmatrix}\right)\right\}.
\end{align*}
Multiplying $w$ on the left by $\gamma_1=\diag(1,\left(\begin{smallmatrix}&I_2\\1\end{smallmatrix}\right))$,
\begin{align*}
V_{\beta}=V_{(1,2,1)}=\left\{\left(\begin{smallmatrix}
                               1 & x_{1,1} & x_{1,2} & y_{1,1} \\
                                & 1 &  & x_{2,4} \\
                                &  & 1 & x_{2,3} \\
                                &  &  & 1
                             \end{smallmatrix}\right)\right\},\qquad\varphi=\left(\begin{smallmatrix}
                                                                  0 & -1 & 0 &0  \\
                                                                   & 0 & 0 & -1 \\
                                                                   &  & 0 & 0 \\
                                                                   &  &  & 0
                                                                \end{smallmatrix}\right).
\end{align*}
The nilpotency order of $\varphi$ is $3$, and since $\rho$ is a $(k,c)=(2,2)$ representation, $J_{V_{\beta},\psi_{V_{\beta}}^{-1}}(\rho)=0$.
\end{example}
Now consider the cases $d_1=n-l$ or $k=1$. There are only finitely many representatives (i.e., representatives $h_i,h_j$ with $h_i\not\sim h_j$) satisfying this condition. This is trivial when $k=1$. For $k>1$ recall $h=wu$ with $\ell_0(u)$ of the form \eqref{eq:second form u final}. Since $d_1=n-l$, and thereby $d_1=\ldots=d_{k-1}=n-l$, we finally have for any $m\in M_Q$, ${}^w\ell(m)\in P$, in particular ${}^w\ell(u)\in P$
hence $h\sim w\ell_0(u)$. We simplify $\ell_0(u)$ and deduce that there are only finitely many representatives remaining.
Regard $\GL_l$ as the direct factor of the standard Levi subgroup $\GL_l\times\mathcal{G}_{c-2l}$ of $G$. For $g_1,g_2\in\GL_l$, because (finally) ${}^w(g_1,g_2)\in P$,
\begin{align*}
w\ell_0(u)\sim w\ell_0(u)(g_1,g_2)\sim w({}^{(g_1,g_2)^{-1}}\ell_0(u)).
\end{align*}
Looking at \eqref{eq:second form u final}, we can now assume $A_l=\diag(I_{l'},0_{l-l'})$, where $l'\leq l$ is the rank of $A_l$. There are only finitely many such representatives. Furthermore if $l'<l$,
take a representative $g$ of an element of $W(G)$ such that
$w\ell_0((g,1))$ does not permute the rows $l'+1,\ldots,l$ of \eqref{eq:second form u final}. Since (as opposed to
the proof of Proposition~\ref{proposition:structure of w u}) ${}^w\ell((g,1))\in P$, $w(g,1)\sim w\ell_0((g,1))$. Then
\begin{align*}
w({}^{(g_1,g_2)^{-1}}\ell_0(u))\sim w(g,1)({}^{(g,1)^{-1}(g_1,g_2)^{-1}}\ell_0(u))\sim w\ell_0((g,1))({}^{(g,1)^{-1}(g_1,g_2)^{-1}}\ell_0(u)).
\end{align*}
Since $l'\leq l$, $w\ell_0((g,1))$ now trivially satisfies \eqref{eq:w_i second reduction} for $l'$ with $d_1=\ldots=d_{k-1}=n-l'$ (the multiplications on the left by elements of $W(M_P)$ do not matter for this), and if we re-denote $w=w\ell_0((g,1))$ and $l=l'$,
we have $h=w({}^{\jmath_{(k-1)c+l}}u_l)$. If $c$ is odd,
$u_l$ commutes with any $\jmath_a$, otherwise $\jmath_{(k-1)c+l}=\jmath_l$. Hence
$h=w({}^{\jmath_{l}}u_l)$ (as in the $k=1$ case).

Thus there are only $n+1$ representatives $h$ to consider, and note that the representative $w({}^{\jmath_n}u_n)$ satisfies $h\sim\delta$.

\begin{proposition}\label{proposition:d1 = n-l l < n}
Assume $d_1=n-l$ or $k=1$, and $l<n$. Then $\mathcal{H}(h)=0$ outside a discrete subset of $s$.
Moreover, under each one of the conditions of \ref{part3}, e.g.,
$\pi_2$ is supercuspidal (and $c>2$ or $G=\Sp_2$), $\mathcal{H}(h)=0$ for all $s$.
\end{proposition}
\begin{proof}
Now $V_{\beta}=V_{(c^k)}$ which is trivial when $k=1$, in which case we set $\psi_{V_{\beta}}=1$. If $k>1$,
since now $\ell(u)$ is trivial, one can read $\psi_{V_{\beta}}$ directly off
\eqref{psi_U on V beta 0}, then
\begin{align}\label{psi_U on V beta d1 n-l}
\psi_{V_{\beta}}(v)=&\psi(-\tr(b_{k}
\left(\begin{smallmatrix}0&0&-\epsilon_0I_l&0&0\\0&\epsilon_0I_{n-l}&0&0&0
\\0&0&0&0&I_{c-2n}
\\0&0&0&0_{n-l}&0
\\I_l&0&0&0&0
\end{smallmatrix}\right))-\sum_{j=1}^{k-2}\tr (b_{k+j})).
\end{align}
(Note that $A_l$ was replaced by $I_l$ in the first matrix.)

Consider the parabolic subgroup $R=M_R\ltimes U_R<G$ where $M_R\cong\GL_{n-l}\times\mathcal{G}_{c-2(n-l)}$ and
\begin{align*}
{}^{\jmath_{(c+1)l}}U_R=\left\{\left(\begin{smallmatrix}I_l&&&u_1&0\\u_2&I_{n-l}&u_4&u_3&u_1'\\&&I_{c-2n}&u_4'\\&&&I_{n-l}\\&&&u_2'&I_l\end{smallmatrix}\right)\right\}.
\end{align*}
Here if $c$ is even, $\jmath_{(c+1)l}=\jmath_l$, otherwise $\jmath_{(c+1)l}$ is trivial.
Since $l<n$, this is a nontrivial parabolic subgroup unless $c=2$ and $G\ne\Sp_2$. If $c$ is odd, the image of $U_R$ in $H_0$ is given by (see Example~\ref{example:odd c})
\begin{align*}
\left\{\left(\begin{smallmatrix}I_l&&\epsilon_2u_4&\epsilon_1u_4&u_1&0\\u_2&I_{n-l}&&&u_3&u_1'\\&&1&&&\epsilon_1u_4'\\&&&1&&\epsilon_2u_4'\\&&&&I_{n-l}\\&&&&u_2'&I_l\end{smallmatrix}\right)\right\}.
\end{align*}
Denote the Lie algebra of $U_R$ by $\mathfrak{u}_R$.
\begin{lemma}\label{lemma:Jacquet module is a trivial rep of U_R}
The following holds for all $s$.
\begin{enumerate}[leftmargin=*]
\item If $F$ is non-archimedean, $J_{V_{\beta},\psi_{V_{\beta}}^{-1}}(\rho)$ is a trivial representation of ${}^h(1,{}^{\jmath_{(c+1)l}}U_R)$.
\item For an archimedean field, ${}^h(1,{}^{\jmath_{(c+1)l}}\mathfrak{u}_R)$ acts locally nilpotently on
$J_{V_{\beta},\psi_{V_{\beta}}^{-1}}(\rho)^*$.
\end{enumerate}
\end{lemma}
The proofs of this lemma and the following one appear after the proof of the proposition. If $\pi_2$ is supercuspidal (nontrivially),
the proposition follows immediately from Lemma~\ref{lemma:Jacquet module is a trivial rep of U_R}, in particular we do not need to exclude any $s$ (see also the discussion preceding \eqref{eq:relation for T with s}).

Identify the group $\GL_{n-l}$ with its image in $M_R$,
\begin{align*}
{}^h(1,\GL_{n-l})=\left\{\diag(I_{n+l},a,I_{c-2n+(k-1)c}):a\in\GL_{n-l}\right\},
\end{align*}
where the right hand side is implicitly regarded as a subgroup of $M_P$. By \eqref{psi_U on V beta d1 n-l},
${}^h(1,\GL_{n-l})$ stabilizes $\psi_{V_{\beta}}$, and because it also normalizes $V_{\beta}$, ${}^h(1,\GL_{n-l})$ acts on
$J_{V_{\beta},\psi_{V_{\beta}}^{-1}}(\rho)$.

\begin{lemma}\label{lemma:Jacquet module is a finite length}
The following holds for all $s$.
\begin{enumerate}[leftmargin=*]
\item\label{p1} If $F$ is non-archimedean, $J_{V_{\beta},\psi_{V_{\beta}}^{-1}}(\rho)$ admits a finite length filtration as a representation of
${}^h(1,\GL_{n-l})$, where ${}^h(1,C_{\GL_{n-l}})$ acts by a character on each constituent. (The constituents need not be irreducible.)
\item Over archimedean fields, there is a maximal parabolic subgroup of $\GL_{kc}$ whose Levi part contains
${}^h(1,\GL_{n-l})$ as a direct factor, such that the Lie algebra $\mathfrak{v}$ of its unipotent radical acts locally nilpotently on $J_{V_{\beta},\psi_{V_{\beta}}^{-1}}(\rho)^*$.
\item\label{p3} If $c=2$, $\rho=\rho_c(\tau)$ for an irreducible supercuspidal representation $\tau$ of $\GL_k$ and $k>1$,
$J_{V_{\beta},\psi_{V_{\beta}}^{-1}}(\rho)=0$.
\end{enumerate}
\end{lemma}
This implies $\mathcal{H}(h)=0$ outside a discrete subset of $s$, because ${}^{h^{-1}}(|\det|^{s-1/2})(1,aI_{n-l})=|a|^{(n-l)(s-1/2)}$ for all $a\in F^*$, and $l<n$, then we can apply \eqref{eq:relation for T with s}. Also
Lemma~\ref{lemma:Jacquet module is a finite length}~\eqref{p3} implies $\mathcal{H}(h)=0$ for all $s$, in the remaining case
of \ref{part3}.
\end{proof}
\begin{proof}[Proof of Lemma~\ref{lemma:Jacquet module is a trivial rep of U_R}]
First consider a non-archimedean field.
We show $J_{V_{\beta},\psi_{V_{\beta}}^{-1}}(\rho)$ is a trivial representation of ${}^h(1,{}^{\jmath_{(c+1)l}}U_R)$.
For $z\in U_R$, ${}^h(1,{}^{\jmath_{(c+1)l}}z)=m_zz'$ where $z'\in U_P$ and
\begin{align*}
m_z=\diag(\left(\begin{smallmatrix}I_l\\&I_{n-l}\\&&I_l\\u_1'&&u_2&I_{n-l}&\epsilon u_4\\&&&&I_{c-2n}\end{smallmatrix}\right),I_{(k-1)c}).
\end{align*}
Here $\epsilon=\epsilon_1$ or $\epsilon_2$ depending on $\jmath_{l}$; also for the computation note that since $l<n$, $\jmath_{l}$ commutes with $u_l$. When $l=0$ and $c$ is even, $m_z$ is trivial whence
${}^h(1,{}^{\jmath_{(c+1)l}}U_R)<U_P$, so that
$J_{V_{\beta},\psi_{V_{\beta}}^{-1}}(\rho)$ is immediately a trivial representation of ${}^h(1,{}^{\jmath_{(c+1)l}}U_R)$.

Let $Z$ be the subgroup of $M_P$ generated by the matrices $m_z$ as $z$ varies in $U_R$. It is an abelian group. The rank of
$Z$ is $(2l+c-2n)(n-l)$. Since $\psi_{V_{\beta}}^{-1}$ restricts to a trivial character on the rows
$2n,\ldots,2n+1-(n-l)$ of $b_k$ (which are the last $n-l$ rows if $c$ is even), $Z$ stabilizes
$\psi_{V_{\beta}}^{-1}$ (it clearly normalizes $V_{\beta}$). Thus $J_{V_{\beta},\psi_{V_{\beta}}^{-1}}(\rho)$ is a representation of $Z$
and for each character $\chi$ of $Z$,
\begin{align*}
J_{Z,\chi}(J_{V_{\beta},\psi_{V_{\beta}}^{-1}}(\rho))=J_{V_{\beta}\ltimes Z,\psi_{V_{\beta}}^{-1}\otimes\chi}(\rho).
\end{align*}
A similar identity holds for any subgroup of $Z$.

For $b\in\GL_c$, denote $b^{\triangle}=\diag(b,b^{\triangle'})\in M_{(c^k)}$ where $b^{\triangle'}$ is the diagonal embedding of $\GL_c$ in
$M_{(c^{k-1})}$. The group $\diag(I_c,\GL_c^{\triangle'})$ stabilizes the restriction of $\psi_{V_{\beta}}^{-1}$ to the blocks $b_{k+1},\ldots,b_{2k-2}$ (but not to $b_k$). The group $\GL_l\times\GL_l\times\GL_{n-l}\times\GL_{c-2n}$
embedded in $M_{(c^k)}$ by
\begin{align}\label{eq:b1 b2 b3}
[x_1,x_2,x_3,x_4]=\diag(\left(\begin{smallmatrix}x_1\\&I_{n-l}\\&&x_2\\&&&x_3\\&&&&x_4\end{smallmatrix}\right),
\left(\begin{smallmatrix}x_2\\&I_{n-l}\\&&x_4\\&&&I_{n-l}\\&&&&x_1\end{smallmatrix}\right)^{\triangle'}),
\end{align}
where $x_1,x_2\in\GL_l$, $x_3\in\GL_{n-l}$ and $x_4\in\GL_{c-2n}$,
acts on the set of characters of $Z$ with infinitely many orbits, but precisely $2$ orbits separately on each block
$Z'=u_1', u_2$ and $\epsilon u_4$, and stabilizes $\psi_{V_{\beta}}^{-1}$. It is enough to prove that each block $Z'$ acts trivially on $J_{V_{\beta},\psi_{V_{\beta}}^{-1}}(\rho)$. By \cite[5.9--5.12]{BZ1}, it suffices to show that for each nontrivial character $\chi'$ of $Z'$,
\begin{align}\label{eq:Z chi vanishing}
J_{V_{\beta}\ltimes Z',\psi_{V_{\beta}}^{-1}\otimes\chi'}(\rho)=0,
\end{align}
since then for each $Z'$, $J_{V_{\beta},\psi_{V_{\beta}}^{-1}}(\rho)=J_{V_{\beta}\ltimes Z',\psi_{V_{\beta}}^{-1}}(\rho)$, and thus
$J_{V_{\beta},\psi_{V_{\beta}}^{-1}}(\rho)=J_{V_{\beta}\ltimes Z,\psi_{V_{\beta}}^{-1}}(\rho)$.
This implies that ${}^h(1,{}^{\jmath_{(c+1)l}}U_R)$ acts trivially on $J_{V_{\beta},\psi_{V_{\beta}}^{-1}}(\rho)$.

Let $\chi'$ be a nontrivial character of $Z'$. As in the proof of Proposition~\ref{proposition:d_1 < n-l}, we let $\varphi$ denote the transpose of the nilpotent element defined by the character $\psi_{V_{\beta}}^{-1}\otimes\chi'$ of $V_{\beta}\ltimes Z'$, and show that $\varphi$ is nilpotent of order at least $k+1$, then \eqref{eq:Z chi vanishing} essentially follows from \cite[Theorems~A, E]{GGS} because $\rho$ is $(k,c)$, but an additional argument is used because $V_{\beta}\ltimes Z'$ does not correspond to a unipotent orbit (see below).

Conjugating $\varphi$ by a suitable element \eqref{eq:b1 b2 b3}, we can assume the bottom left coordinate of $Z'$ in $\varphi$ is $1$, and all other coordinates on the same column of $Z'$ are $0$. Assume (momentarily) $c$ is even or $l>0$.
One can permute $\epsilon u_4$ (for odd $c$) with the first column of $u_2$ using conjugation by
\begin{align*}
\diag(\left(\begin{smallmatrix}I_n\\&&&-1\\&&I_{n-1}\\&-1\end{smallmatrix}\right),
\left(\begin{smallmatrix}&&1\\&I_{n-1}\\1\\&&&I_n\end{smallmatrix}\right)^{\triangle'}).
\end{align*}
This element normalizes $V_{\beta}$ and stabilizes the blocks
$b_{k},\ldots,b_{2k-2}$ of $\varphi$.
Moreover,
we can exchange the blocks $u_1'$ and $u_2$ using conjugation by
\begin{align}\label{eq:matrix conjugating u1 and u2}
\diag(\left(\begin{smallmatrix}&&I_l\\&I_{n-l}\\I_l\\&&&I_{\lceil c/2\rceil-l}\end{smallmatrix}\right),I_{(k-1)c}),
\end{align}
hence we can always assume, for any $Z'$, that after a conjugation the bottom left coordinate of the block $u_1'$ of $\varphi$ is $1$. The matrix \eqref{eq:matrix conjugating u1 and u2} normalizes $V_{\beta}$, fixes the blocks
$b_{k+1},\ldots,b_{2k-2}$ of $\varphi$, but permutes the coordinates of the block $b_k$ of $\varphi$. In particular the $(n+1,1)$-th coordinate in the block $b_k$ of $\varphi$, which is $-\epsilon_0$, is conjugated into the $(1,1)$-th coordinate of this block. If $Z'=u_1'$ we can skip this conjugation. When $c$ is odd and $l=0$, in which case $Z'=\epsilon u_4$ (evidently $u_1'$ and $u_2$ are trivial when $l=0$), we use conjugation by
\begin{align*}
\diag(\left(\begin{smallmatrix}&&1\\&I_{c-2}\\1\end{smallmatrix}\right),
\left(\begin{smallmatrix}&&1\\&I_{c-2}\\1\end{smallmatrix}\right)^{\triangle'}),
\end{align*}
so that the $(c,n+1)$-th coordinate of the block $b_k$ of $\varphi$, which is $1$, is permuted to the $(1,n+1)$-th coordinate, and
the bottom left coordinate of $Z'$ in $\varphi$ is permuted to the $(2n,1)$-th coordinate of $\varphi$. At any rate, $\varphi$ has a nonzero coordinate in the first row of $b_k$.

Conjugating $\varphi$ by an element of $\diag(I_c,\GL_c^{\triangle'})$,
we can always assume $\varphi$ contains $1$ on the top right coordinate of $b_k$, and additionally (still) contains $1$ in the
$(2n,1)$-th coordinate (the bottom left coordinate of $u_1'$ if $l>0$).
If $c$ is odd, we can permute this coordinate to the $(c,1)$-th coordinate using conjugation by $\diag(I_{c-2},\left(\begin{smallmatrix}&1\\1\end{smallmatrix}\right),I_{(k-1)c})$.

Now conjugating $\varphi$ by
\begin{align*}
\diag(\left(\begin{smallmatrix}&&1\\&I_{c-2}\\1\end{smallmatrix}\right)
,I_{k(c-1)}),
\end{align*}
the top right coordinate of $b_k$ is permuted into its bottom right, so that now $\varphi$ has $1$ in the bottom right coordinate of each of the blocks $b_k,\ldots,b_{2k-2}$, and the $(c,1)$-th coordinate of $\varphi$ is permuted into the $(1,c)$-th coordinate.
Moreover, the only nonzero entry in column $c$ is the $(1,c)$-th coordinate, and in each $b_k,\ldots,b_{2k-2}$ the only nonzero entry on the  rightmost column is the bottom right one.
This brings us to the situation in the proof of Proposition~\ref{proposition:d_1 < n-l}, with two exceptions. First, instead of $-1$ in the bottom right coordinate of the block $b_{k-1}$, we have $1$ in the $(1,c)$-th coordinate (the first coordinate to the left of the top left coordinate of $b_k$). It again follows that $\varphi$ is nilpotent of order at least $k+1$.

Second and more importantly, the unipotent subgroup $V_{\beta}\ltimes Z'$ does not correspond to a unipotent orbit (i.e., it is not of the form $V(\sigma)$, see \S~\ref{kc representations}). However, we reduced \eqref{eq:Z chi vanishing} to the vanishing of
$J_{V_{\beta}\ltimes Z'',\psi_{V_{\beta}}^{-1}\otimes\chi^{''}}(\rho)$,
where $Z''<N_{\GL_{kc}}$ corresponds to the $(1,c)$-th coordinate and $\chi''$ is a nontrivial character of $Z''$. We see that $V_{\beta}\ltimes Z''<V_{(1,c-1,c^{k-1})}$ and any character $\psi'$ of $V_{(1,c-1,c^{k-1})}$ which extends $\psi_{V_{\beta}}^{-1}\otimes\chi^{''}$ is still nilpotent of order $k+1$. Also, the torus $T_{\GL_{c-1}}$ acts on the added $c-2$ new coordinates of $V_{(1,c-1,c^{k-1})}$ with finitely many orbits (one can identify each diagonal coordinate from $T_{\GL_{c-1}}$ with an element of $T_{\GL_{kc}}$ which fixes $\psi_{V_{\beta}}^{-1}\otimes\chi^{''}$). Thus (by \cite[5.9--5.12]{BZ1}) $J_{V_{\beta}\ltimes Z'',\psi_{V_{\beta}}^{-1}\otimes\chi^{''}}(\rho)=0$ if for any $\psi'$ as above,
$J_{V_{(1,c-1,c^{k-1})},\psi'}(\rho)=0$. The latter holds by \cite[Theorems~A, E]{GGS} because $\rho$ is $(k,c)$.
We conclude \eqref{eq:Z chi vanishing} for any nontrivial $\chi'$ and any block $Z'$.

For the archimedean case, again the result is trivial if $l=0$ and $c$ is even.
The (abelian) Lie algebra $\mathfrak{z}$ of $Z$ decomposes into the direct sum of one-dimensional Lie algebras $\mathfrak{z}_{i,j}$,
corresponding to the coordinates $Z_{i,j}$ of $Z$ (which can be identified with roots of $\GL_{kc}$).

For each $(i,j)$, there is a subgroup of \eqref{eq:b1 b2 b3} which acts on the characters of $Z_{i,j}$ with $2$ orbits; one can identify
this subgroup with $T_{\GL_2}$. Then we can proceed as in the non-archimedean case and prove
$J_{V_{\beta}\ltimes Z_{i,j},\psi_{V_{\beta}}^{-1}\otimes\chi'}(\rho)=0$ for any nontrivial character $\chi'$, where to deduce this
from the vanishing of $J_{V_{(1,c-1,c^{k-1})},\psi'}(\rho)=0$ for all $\psi'$ we apply \cite[Corollary~3.0.2]{GGS2} (instead of \cite{BZ1}).
By the transitivity of the Jacquet functor, this implies that there are no continuous distributions on
$J_{V_{\beta},\psi_{V_{\beta}}^{-1}}(\rho)$ that transform on the left under $Z_{i,j}$ by $\chi'$, i.e., $(J_{V_{\beta},\psi_{V_{\beta}}^{-1}}(\rho)^*)^{(Z_{i,j},\chi')}=0$. Hence by \cite[Proposition 3.0.1]{GGS2},
$\mathfrak{z}_{i,j}$ acts locally nilpotently on $J_{V_{\beta},\psi_{V_{\beta}}^{-1}}(\rho)^*$.
Note that
for the proof of \eqref{eq:Z chi vanishing} above we really used only one coordinate, the bottom left one, for each block, and using conjugation we can assume this coordinate is $Z_{i,j}$.

We deduce that each $\mathfrak{z}_{i,j}$ acts locally nilpotently, hence so does
$\mathfrak{z}$.
\end{proof}
\begin{proof}[Proof of Lemma~\ref{lemma:Jacquet module is a finite length}]
Consider a non-archimedean field.
We prove that $J_{V_{\beta},\psi_{V_{\beta}}^{-1}}(\rho)$ admits a finite length filtration as a representation of ${}^h(1,\GL_{n-l})$, and on each constituent ${}^h(1,C_{\GL_{n-l}})$ acts by a character, by showing $J_{V_{\beta},\psi_{V_{\beta}}^{-1}}(\rho)$ factors through a Jacquet module along a unipotent radical of a certain parabolic subgroup, with respect to a trivial character.

After conjugating $J_{V_{\beta},\psi_{V_{\beta}}^{-1}}(\rho)$ by
\begin{align*}
&\kappa=\diag(\left(\begin{smallmatrix}&&I_{n-l}\\&I_{n}\\I_l\\&&&I_{c-2n}\end{smallmatrix}\right),I_{k(c-1)}),
\end{align*}
we regard $J_{V_{\beta},\psi_{V_{\beta}}^{-1}}(\rho)$ as a representation of
${}^{\kappa}({}^h(1,\GL_{n-l}))=\diag(\GL_{n-l},I_{kc-(n-l)})$. In addition,
this conjugation only changes the restriction of $\psi_{V_{\beta}}^{-1}$ to $b_k$, now given by
\begin{align*}
\psi(\tr(b_{k}\left(\begin{smallmatrix}0&0&-\epsilon_0I_l&0&0\\0&\epsilon_0I_{n-l}&0&0&0\\0&0&0&0&I_{c-2n}\\0_{n-l}&0&0&0&0\\0&0&0&I_l&0\end{smallmatrix}\right)))
\end{align*}
(use \eqref{psi_U on V beta d1 n-l}). Now $\psi_{V_{\beta}}^{-1}$ is trivial on the top $n-l$ rows of $b_k$, hence
$J_{V_{\beta},\psi_{V_{\beta}}^{-1}}(\rho)$ is a representation of
$V_{(n-l,c-(n-l))}$, which we identify with its image in the top left $c\times c$ block of $\GL_{kc}$.

We claim
\begin{align}\label{eq:Jacquet admissible Levi non archimedean}
J_{V_{\beta},\psi_{V_{\beta}}^{-1}}(\rho)=J_{V_{(n-l,c-(n-l),c^{k-1})},\psi_{V_{\beta}}^{-1}}(\rho).
\end{align}
Here $\psi_{V_{\beta}}^{-1}$ is extended trivially to $V_{(n-l,c-(n-l))}$. Before proving \eqref{eq:Jacquet admissible Levi non archimedean}, we explain how it leads to the result.
Because $V_{(n-l,kc-(n-l))}<V_{(n-l,c-(n-l),c^{k-1})}$, the right hand side of \eqref{eq:Jacquet admissible Levi non archimedean} becomes
\begin{align*}
J_{\diag(I_{n-l},V_{(c-(n-l),c^{k-1})}),\psi_{V_{\beta}}^{-1}}(J_{V_{(n-l,kc-(n-l))}}(\rho)).
\end{align*}
Since $\rho$ is an admissible finite length representation of $\GL_{kc}$, $J_{V_{(n-l,kc-(n-l))}}(\rho)$ is an admissible finite length representation of $M_{(n-l,kc-(n-l))}$. As such, it admits a finite filtration with irreducible admissible constituents. On each constituent $\mathcal{V}$, $C_{M_{(n-l,kc-(n-l))}}$ acts by a character, and because
${}^{\kappa h}(1,C_{\GL_{n-l}})<C_{M_{(n-l,kc-(n-l))}}$, ${}^{\kappa h}(1,C_{\GL_{n-l}})$ also acts by a character. Note that $\mathcal{V}$ may certainly be reducible (or not admissible) as a representation of ${}^{\kappa h}(1,\GL_{n-l})$.
By the exactness of the Jacquet functor, $J_{V_{\beta},\psi_{V_{\beta}}^{-1}}(\rho)$ admits a finite filtration where on each constituent
${}^{\kappa h}(1,C_{\GL_{n-l}})$ still acts by the same character. This completes the proof of the main assertion --- part~\eqref{p1} --- for the non-archimedean case.
Regarding \eqref{p3}, when $c=2$, $\rho=\rho_c(\tau)$ for an irreducible supercuspidal representation of $\GL_k$ and $k>1$, the Jacquet module $J_{V_{(n-l,kc-(n-l))}}(\rho)$ vanishes since $n-l=1$ (\cite[2.13 (a)]{BZ2}).

Write
$v\in V_{(n-l,c-(n-l))}$ in the form $v=(v_1,v_2,v_3,v_4)$ with $v_1\in\Mat_{n-l}$, $v_2,v_3\in\Mat_{n-l\times l}$ and
$v_4\in\Mat_{n-l\times c-2n}$.
The group ${}^{\kappa}[x_1,x_2,x_3,x_4]$ (see \eqref{eq:b1 b2 b3}), together with the group $\GL_{n-l}$ embedded in $M_{(c^k)}$ by $x_5\mapsto\diag(I_{n-l},x_5,I_{2l+c-2n})^{\triangle}$, stabilizes $\psi_{V_{\beta}}^{-1}$ and acts on the set of characters of $V_{(n-l,c-(n-l))}$ with infinitely many orbits, but only $2$ on each block $v_i$ separately. Using the transitivity of the Jacquet functor and \cite[5.9--5.12]{BZ1}, \eqref{eq:Jacquet admissible Levi non archimedean} follows at once if we prove separately,
that for each block $Z'=v_i$ and nontrivial character $\chi'$ of $Z'$,
\begin{align}\label{eq:V chi vanishing}
J_{V_{\beta}\ltimes Z',\psi_{V_{\beta}}^{-1}\otimes\chi'}(\rho)=0.
\end{align}

Let $\varphi$ denote the transpose of the nilpotent element defined by the character $\psi_{V_{\beta}}^{-1}\otimes\chi'$ of $V_{\beta}\ltimes Z'$. We prove $\varphi$ is nilpotent of order at least $k+1$, then since $\rho$ is $(k,c)$, \cite[Theorems~A, E]{GGS} imply \eqref{eq:V chi vanishing}.

First we show that, after possibly a suitable conjugation, $\varphi$ is nontrivial on the
$(1,c)$-th coordinate, and all other blocks remain unchanged except the block $b_k$, where there is a nonzero entry on the
bottom right coordinate.

We can assume the top right coordinate of $Z'$ in $\varphi$ is nonzero, and it is the only nonzero entry on that column.
If $Z'=v_4$, the $(1,c)$-th entry of $\varphi$ is nonzero. Using conjugation by
an element of $\diag(I_c,\GL_c^{\triangle'})$, the $(c,n+1)$-th coordinate of $b_k$ becomes its $(c,c)$-th coordinate, and the other blocks $b_{k+1},\ldots,b_{2k-2}$ are unchanged.

For $Z'=v_2$, we conjugate $\varphi$ by
\begin{align*}
\diag(\left(\begin{smallmatrix}I_{n-l}\\&I_{n-l}\\&&&I_l\\&&I_l\\&&&&I_{c-2n}\end{smallmatrix}\right),
\left(\begin{smallmatrix}&&I_{l}\\&I_{c-2l}\\I_l\end{smallmatrix}\right)^{\triangle'}),
\end{align*}
and for $Z'=v_1$ we conjugate by
\begin{align*}
\diag(\left(\begin{smallmatrix}I_{n-l}\\&&&I_l\\&&I_l\\&I_{n-l}\\&&&&I_{c-2n}\end{smallmatrix}\right),
\left(\begin{smallmatrix}I_{l}\\&&&I_{l}\\&&I_{c-n-l}&\\&I_{n-l}\end{smallmatrix}\right)^{\triangle'}).
\end{align*}
Both conjugations preserve $b_{k+1},\ldots,b_{2k-2}$, the $(1,2n)$-th coordinate of $\varphi$ becomes nontrivial, and the rightmost column of $b_k$ has (precisely) one nontrivial coordinate, which is the $(c-1,c)$-th coordinate if $c$ is odd, otherwise it is the bottom coordinate (the $(c,c)$-th coordinate).
For an even $c$, $\varphi$ is of the prescribed form; if $c$ is odd, using another conjugation by
$\diag(I_{c-2},\left(\begin{smallmatrix}&1\\1\end{smallmatrix}\right),I_{(k-1)c})$, the $(c-1,c)$ entry of block $b_k$ is permuted into its bottom right coordinate, and the $(1,2n)$-th coordinate of $\varphi$ becomes its $(1,c)$-th coordinate.

We conclude that in all cases of $Z'$, when $\chi'$ is nontrivial, the $(1,c)$-th coordinate of $\varphi$ and the bottom right coordinate of
each block $b_k,\ldots,b_{2k-2}$ of $\varphi$ are nonzero (these coordinates are all $1$ except for $b_k$, where the coordinate is $\pm1$), and the corresponding nonzero entry is the unique one in its column. Thus as in the proof of Lemma~\ref{lemma:Jacquet module is a trivial rep of U_R} (and again considering all extensions to characters of $V_{(1,c-1,c^{k-1})}$ in order to ``adjust" $V_{\beta}\ltimes Z'$ to a unipotent radical), $\varphi$ is nilpotent of order at least $k+1$ and \eqref{eq:V chi vanishing} follows.

Over archimedean fields, as in the proof of
Lemma~\ref{lemma:Jacquet module is a trivial rep of U_R}, we deduce that the Lie algebra
$\mathfrak{v}_{(n-l,c-(n-l))}$ of $V_{(n-l,c-(n-l))}$ acts locally nilpotently on $J_{V_{\beta},\psi_{V_{\beta}}^{-1}}(\rho)^*$ by carrying out the proof of \eqref{eq:V chi vanishing} and applying \cite[Proposition~3.0.1]{GGS2} separately for each coordinate of each $v_i$.
Let $\mathfrak{v}'$ denote the Lie algebra of the unipotent subgroup $V_{(n-l,kc-(n-l))}\cap V_{\beta}$. Since $\mathfrak{v}'$ acts trivially on $J_{V_{\beta},\psi_{V_{\beta}}^{-1}}(\rho)$, $\mathfrak{v}_{(n-l,kc-(n-l))}=\mathfrak{v}_{(n-l,c-(n-l))}\oplus \mathfrak{v}'$ acts locally nilpotently on $J_{V_{\beta},\psi_{V_{\beta}}^{-1}}(\rho)^*$.
\end{proof}
\begin{example}
Consider $c=4$, $k=2$, $G=\Sp_4$, $H=\Sp_{16}$ and $l=1$. Assume $d_1=n-l=1$. Then $w\in M_P(0^3,1,w_2)$, $w_2=(1^4)$ and $u=\ell_0(u)\in H_0$ is given by
\begin{align*}
u=\diag(I_4,
\left(\begin{smallmatrix}
1&&1\\&1\\&&1\end{smallmatrix}\right),I_2,\left(\begin{smallmatrix}
1&&-1\\&1\\&&1\end{smallmatrix}\right),I_4).
\end{align*}
The element $w$ is given by \eqref{eq:example k=2,3 final w for k=2}, note that
$\gamma_1=\diag(\left(\begin{smallmatrix} &I_4   \\I_2\end{smallmatrix}\right),I_2)\left(\begin{smallmatrix}& I_6 \\I_2 & \end{smallmatrix}\right)$ (see Remark~\ref{remark:convenient computations of V beta}). We have $\beta=(4^2)$, and if we write an element of $V_{\beta}$ in the form $v=\left(\begin{smallmatrix} I_4 & x  \\&I_4\end{smallmatrix}\right)$, $\psi_{V_{\beta}}(v)=\psi(-x_{1,4}+x_{2,2}-x_{3,1})$. The Jacquet module
$J_{V_{\beta},\psi_{V_{\beta}}^{-1}}(\rho)$ is a representation of ${}^h(1,U_R)$ with
\begin{align*}
&U_R=\left\{\left(\begin{smallmatrix}
             1 &  & u_1 & 0 \\
             u_2 & 1 & u_3 & u_1 \\
              &  &1  &  \\
              &  & -u_2 & 1
           \end{smallmatrix}\right)\right\}.
\end{align*}
We see that for $z\in U_R$,
\begin{align*}
\qquad{}^{u}(1,z)=\diag(I_4,\left(\begin{smallmatrix}
                            1 &  &  &  & u_1 &  &  &  \\
                             & 1 &  &  &  &  &  &  \\
                             &  & 1 &  & u_1 &  &  &  \\
                             &  & u_2 & 1 & u_3 & u_1 &  & u_1 \\
                             &  &  &  & 1 &  &  & \\
                             &  &  &  & -u_2 & 1 &  &  \\
                             &  &  &  &  &  & 1 &  \\
                             &  &  &  &  &  &  & 1
                          \end{smallmatrix}\right),I_4).
\end{align*}
Then ${}^h(1,z)=m_zz'$ with $m_z\in M_P$ and $z'\in U_P$, and such that as an element of $\GL_8$,
\begin{align*}
m_z=\diag(\left(\begin{smallmatrix}
             1 &  &  &  \\
              & 1 &  & \\
              &  &1  &  \\
             u_1 &  & u_2 & 1
           \end{smallmatrix}\right),I_4).
\end{align*}
Denote the subgroup of elements of this form by $Z$, then
$J_{V_{\beta},\psi_{V_{\beta}}^{-1}}(\rho)$ is a representation of $Z$. We proceed over non-archimedean fields. To show that
$J_{V_{\beta},\psi_{V_{\beta}}^{-1}}(\rho)$ is a trivial representation of ${}^h(1,U_R)$ amounts to proving
$J_{V_{\beta}\ltimes Z,\psi_{V_{\beta}}^{-1}\otimes\chi}(\rho)=0$ for any nontrivial character $\chi$ of $Z$. Combining $Z$ and $V_{\beta}$ together we are considering the following unipotent subgroup and character:
\begin{align*}
\{v=\left(\begin{smallmatrix}
                            1 &  &  &  & x_{1,1} & x_{1,2} &x_{1,3}  & x_{1,4} \\
                             & 1 &  &  & x_{2,1} & x_{2,2} &x_{2,3}  & x_{2,4} \\
                             &  & 1 &  & x_{3,1} & x_{3,2} &x_{3,3}  & x_{3,4} \\
                            u_1 &  & u_2 & 1 & x_{4,1} & x_{4,2} &x_{4,3}  & x_{4,4} \\
                             &  &  &  & 1 &  &  & \\
                             &  &  &  &  & 1 &  &  \\
                             &  &  &  &  &  & 1 &  \\
                             &  &  &  &  &  &  & 1
                          \end{smallmatrix}\right)\},\qquad
                          (\psi_{V_{\beta}}^{-1}\otimes\chi)(v)=\psi(x_{1,4}-x_{2,2}+x_{3,1}+\alpha_1 u_1+\alpha_2 u_2).
\end{align*}
We have an action of $T_{\GL_2}$ on $u_1$ and $u_2$ separately, given by
$\diag(x_1,I_2,x_3,I_3,x_1)$ (for $u_1$) and $\diag(I_2,x_2,x_3,x_2,I_3)$. When considering each coordinate separately, there are $2$ orbits.
The corresponding $\varphi$ takes the form
\begin{align*}
\varphi=\left(\begin{smallmatrix}
                0 & 0 & 0 & 0 & 0 & 0 & 0 & 1 \\
                0 & 0 & 0 & 0 & 0 & -1 & 0 & 0 \\
                0 & 0 & 0 & 0 & 1 & 0 & 0 & 0 \\
                \alpha_1 & 0 & \alpha_2 & 0 & 0 & 0 & 0 & 0 \\
                 &  &  &  & 0 &  &  &  \\
                 &  & &  &  & 0 &  &  \\
                 &  &  &  &  &  & 0 &  \\
                 &  &  &  &  &  &  & 0
              \end{smallmatrix}\right),
\end{align*}
and a nontrivial $\chi$ means $(\alpha_1,\alpha_2)\ne(0,0)$. Using conjugations by $\diag(J_3,I_5)$ and by $\diag(I_4,g)$ for a permutation matrix $g\in\GL_4$ if necessary, we can assume the $(4,1)$-th and $(1,8)$-th coordinates of $\varphi$ are nonzero, then conjugating by $\diag(\left(\begin{smallmatrix}&  & 1 \\& I_2 &  \\1 &  & \end{smallmatrix}\right),I_4)$, we see that $\varphi$ is nilpotent of order at least $3$. Thus
$J_{V_{\beta},\psi_{V_{\beta}}^{-1}}(\rho)$ is a trivial representation of ${}^h(1,U_R)$ ($\rho$ is $(k,c)=(2,4)$).

The Jacquet module $J_{V_{\beta},\psi_{V_{\beta}}^{-1}}(\rho)$ is also a representation of
\begin{align*}
{}^h(1,\GL_{n-l})=\{\diag(I_3,a,I_4):a\in F^*\}.
\end{align*}
Conjugating by $\kappa=\diag(\left(\begin{smallmatrix}&&1\\&I_2\\1\end{smallmatrix}\right),I_4)$, we can regard $J_{V_{\beta},\psi_{V_{\beta}}^{-1}}(\rho)$ as a representation of $\diag(V_{(1,3)},I_4)$. Combining the coordinates of $V_{(1,3)}$ and $V_{\beta}$, we then have
\begin{align*}
\left\{v=\left(\begin{smallmatrix}
                            1 & v_1 & v_2 & v_3 & x_{4,1} & x_{4,2} &x_{4,3}  & x_{4,4} \\
                             & 1 &  &  & x_{2,1} & x_{2,2} &x_{2,3}  & x_{2,4} \\
                             &  & 1 &  & x_{3,1} & x_{3,2} &x_{3,3}  & x_{3,4} \\
                             &  &  & 1 & x_{1,1} & x_{1,2} &x_{1,3}  & x_{1,4} \\
                             &  &  &  & 1 &  &  & \\
                             &  &  &  &  & 1 &  &  \\
                             &  &  &  &  &  & 1 &  \\
                             &  &  &  &  &  &  & 1
                          \end{smallmatrix}\right)\right\},
\end{align*}
and note that we permuted the coordinates of $x_{i,j}$ (so that $\psi_{V_{\beta}}^{-1}$ remains as above).
The tensor product of an arbitrary character $\chi$ of $V_{(1,3)}$ with $\psi_{V_{\beta}}^{-1}$ takes the form
\begin{align*}
(\psi_{V_{\beta}}^{-1}\otimes\chi)(v)=\psi(x_{1,4}-x_{2,2}+x_{3,1}+\vartheta_1 v_1+\vartheta_2 v_2+\vartheta_3 v_3).
\end{align*}
As above, we have the action of $T_{\GL_2}$ on each of the coordinates $v_i$:
$\diag(x_3,x_5,I_3,x_5,I_2)$ for $v_1$,
$\diag(x_3,1,x_2,1,x_2,I_3)$ for $v_2$, and $\diag(x_3,I_2,x_1,I_3,x_1)$ for $v_3$. The corresponding $\varphi$ is then nilpotent of order at least $3$, when $\vartheta_1\vartheta_2\vartheta_3\ne0$. This is immediate when $\vartheta_3\ne0$, using $x_{1,4}$. If $\vartheta_3=0$ and $\vartheta_2\ne0$, we conjugate by $\diag(I_2,J_2,\left(\begin{smallmatrix}&&1\\&I_2\\1\end{smallmatrix}\right))$ and use $x_{3,1}$, otherwise $\vartheta_1\ne0$ and we conjugate by $\diag(1,J_3,1,J_3)$ and use $x_{2,2}$.

Therefore $J_{V_{\beta},\psi_{V_{\beta}}^{-1}}(\rho)$ factors through $V_{(1,7)}$, so that the original
module $J_{V_{\beta},\psi_{V_{\beta}}^{-1}}(\rho)$ factors through ${}^{\kappa^{-1}}V_{(1,7)}$. Since $\det(\diag(I_3,a,I_4))=a$, we can use
\eqref{eq:relation for T with s} to deduce $\mathcal{H}(h)=0$.
\end{example}
Propositions~\ref{proposition:structure of w u}--\ref{proposition:d1 = n-l l < n} imply $\mathcal{H}(h)=0$ for all $h$ such that $h\not\sim\delta$. Finally consider $h=w({}^{\jmath_n}u_n)\sim\delta$. We prove that for all $s$, $\dim\mathcal{H}(\delta)=\dim\Hom_{G}(\pi_1^{\vee},\pi_2^{\iota})$. In this case
$P_{\delta}=(G^{\iota}\times C_H^{\circ})\ltimes V_{(c^k)}$, where $G^{\iota}=\{(g,{}^{\iota}g):g\in G\}$,
$\psi_{V_{\beta}}^{-1}$ belongs to the orbit of $\psi_k$ (the choice of $\delta$ of \cite{CFK} gives precisely $\psi_k$),
and any morphism in $\mathcal{H}(\delta)$ factors through $J_{V_{(c^k)},\psi_k}(\rho)$. Note that $C_H^{\circ}$ is trivial unless $H=\GSpin_{2kc}$, in which case $C_H^{\circ}<P_{\delta}$ because $C_G^{\circ}$ is mapped by the embedding $g\mapsto (g,1)$ bijectively
into $C_H^{\circ}$ (see also \eqref{embedding G G in H GSpin}).
Therefore
\begin{align*}
\mathcal{H}(\delta)&=\Hom_{G^{\iota}\times C_H^{\circ}}({}^{\delta^{-1}}J_{V_{(c^k)},\psi_k}(\rho)\otimes\eta\otimes\pi_1^{\vee}\otimes\pi_2^{\vee},1).
\end{align*}
Here $|\det|^{s-1/2}$ and $\theta_h$ are absent because they are trivial on $G^{\iota}\times C_H^{\circ}$.

For $H=\GSpin_{2kc}$ we assumed $\chi_{\pi_1},\chi_{\pi_2}$ exist, then $\mathcal{H}(\delta)=0$ unless $\eta=\chi_{\pi_1}^{-1}=\chi_{\pi_2}$, because for $z_1,z_2\in C_G^{\circ}$, $(z_1,z_2)$ is the element $z_1^{-1}z_2$ of $C_H^{\circ}$. When this condition holds, we can finally
ignore $C_H^{\circ}$ and $\eta$ altogether.

Recall that $\GL_c^{\triangle}$ denotes the diagonal embedding of $G$ in $M_{(c^k)}$, and $J_{V_{(c^k)},\psi_k}(\rho)$ is a trivial representation of
$\SL_c^{\triangle}$. Since ${}^{\delta}G^{\iota}<\SL_c^{\triangle}$
(for $H\ne\GSpin_{2kc}$, ${}^{\delta}G^{\iota}=G^{\triangle}$), the action of $G^{\iota}$ on ${}^{\delta^{-1}}J_{V_{(c^k)},\psi_k}(\rho)$ is trivial, and because $\dim J_{V_{(c^k)},\psi_k}(\rho)=1$ (see \S~\ref{kc representations}),
\begin{align*}
\Hom_{G^{\iota}}({}^{\delta^{-1}}J_{V_{(c^k)},\psi_k}(\rho)\otimes\pi_1^{\vee}\otimes\pi_2^{\vee},1)
=\Hom_{G}(\pi_1^{\vee}\otimes(\pi_2^{\vee}){}^{\iota},1)=\Hom_{G}(\pi_1^{\vee},\pi_2^{\iota}).
\end{align*}

This completes the proof of the first part of the theorem. For the second part, clearly when $\pi_1$ and $\pi_2$ are irreducible,
$\dim\Hom_{G}(\pi_1^{\vee},\pi_2^{\iota})\leq1$ and is zero unless $\pi_1=(\pi_2^{\iota})^{\vee}$.
Under the assumptions of \ref{part3} (e.g., $\pi_2$ is supercuspidal and $c>2$), we do not need to exclude any $s$, by
Proposition~\ref{proposition:d1 = n-l l < n} (which is the only case where the vanishing depends on $s$).

\subsection{The case $H=\GL_{2kc}$}\label{section GL}
Write $P\backslash H/ D=\coprod_hPhD$ with $h=wu$, where $w$ is a representative from $W(M_P)\backslash W(H)$ and $u\in M_Q\cap N_H$.
Throughout this section we fix the standard identification of $W(H)$ with permutation matrices in $H$. One can still describe
$w$ as a $2kc$-tuple of $\{0,1\}$: if the $i$-th coordinate of $w$ is $1$, $w$ permutes the $i$-th row into one of the first $kc$ rows,
and if it is $0$, then $w$ permutes this row into one of the last $kc$ rows. Of course only vectors whose total sum of coordinates is $kc$ are permissible ($U_P$ contains $kc$ nontrivial rows). Let $p_0(w)$ denote the middle $2c$ coordinates of $w$, and $p_1(w)$ (resp., $p_2(w)$) denote the first (resp., last) $(k-1)c$ coordinates. Also note that
in general, $w$ permutes row $i$ into $U_P$ if and only if $w^{-1}$ permutes column $i$ out of $U_P$.

For the case $k=1$, we can parameterize $P\backslash H/(G\times G)$ using the elements
\begin{align*}
(0^l,1^{c-l},0^{c-l},1^l)u_{l,j}, \qquad 0\leq l\leq c, \qquad 0\leq j\leq l,
\end{align*}
where
\begin{align*}
(0^l,1^{c-l},0^{c-l},1^l)=\left(\begin{smallmatrix}&&I_{l}\\&I_{2(c-l)}\\I_{l}\end{smallmatrix}\right),\qquad
u_{l,j}=\left(\begin{smallmatrix}I_{l}&&A_{l,j}\\&I_{2(c-l)}\\&&I_{l}\end{smallmatrix}\right)
, \quad A_{l,j}=\left(\begin{smallmatrix}&I_j\\0\end{smallmatrix}\right).
\end{align*}
The choice of matrix for $(0^l,1^{c-l},0^{c-l},1^l)$ is not canonical, but can be used for convenience.

Assume $k\geq1$. Recall
\begin{align*}
M_Q=\GL_c\times\ldots\times\GL_c\times H_0\times\GL_c\times\ldots\times\GL_c,\qquad H_0=\GL_{2c}.
\end{align*}
Given $x\in M_Q$, denote its projection into the left (resp., right) direct product of $k-1$ copies of $\GL_c$ by $\ell_1(x)$ (resp., $\ell_2(x)$), put $\ell(x)=\ell_1(x)\ell_2(x)$, and let $\ell_0(x)$ be the projection into $H_0$.
We have the analogs of \eqref{eq:conj x by g1 g2} and \eqref{eq:conj x by g2}, in particular since $(1,G)<H_0$, conjugation
by elements of $(1,G)$ does not affect $\ell(x)$.
\begin{proposition}\label{proposition:GL structure of w u}
Let $h=wu$, where $w$ is a representative from $W(M_P)\backslash W(H)$ and $u\in M_Q\cap N_H$.
Then $h\sim\hat{w}\hat{u}$, where $p_0(\hat{w})=(0^{l},1^{c-l},0^{c-l'},1^{l'})$ for some $0\leq l,l'\leq c$,
$\hat{u}\in M_Q$, there is $\sigma\in(W(G),W(G))$ with ${}^{\sigma}\hat{u}\in M_Q\cap N_H$,
and $\ell_0(\hat{u})$ takes the form
\begin{align}\label{eq:GL second form u final}
\left(\begin{smallmatrix}I_{l}&&X\\&I_{2c-l-l'}\\&&I_{l'}\end{smallmatrix}\right),
\end{align}
for some $X$.
\end{proposition}
\begin{proof}
Identify $N_{\GL_c}\times N_{\GL_c}$ with its image in $M_{(c,c)}<H_0$. Since $N_{\GL_c}\times N_{\GL_c}<\ell_0((G,G))$, one can assume
$\ell_0(u)\in V_{(c,c)}$. Through this $\ell(u)$ is multiplied by an element in $M_Q\cap N_H$.

One can find $\sigma\in(W(G),W(G))$ such that
$w\sigma=w_1$ satisfies $p_0(w_1)=(0^{l},1^{c-l},0^{c-l'},1^{l'})$ for some $0\leq l,l'\leq c$.
Denote $u_1={}^{\sigma^{-1}}u$. Since
$\ell_0((G,G))<M_{(c,c)}$, $\ell_0(u_1)\in V_{(c,c)}$. Also
$\ell(u_1)\in M_Q$, but might not be in $N_H$. Then $wu\sim wu\sigma=w_1u_1$.

Write the top right $c\times c$ block of $\ell_0(u_1)$ in the form
$\left(\begin{smallmatrix}X^1 & X^2 \\X^3 & X^4\end{smallmatrix}\right)$, $X^2\in\Mat_{l\times l'}$. Now $w_1$ conjugates the blocks $X^i$, $i\ne2$, into $P$. Hence if
$u_2=z^{-1}u_1$ with $z\in V_{(c,c)}$ defined by these blocks, $w_1u_1=w_1zu_2\sim w_1u_2$. Now $\ell_0(u_2)$ takes the form
\eqref{eq:GL second form u final}.
Note that $\ell(u_2)=\ell(u_1)$, whence $\ell({}^{\sigma}u_2)=\ell({}^{\sigma}u_1)=\ell(u)\in M_Q\cap N_H$, and
$\ell_0({}^{\sigma}u_2)\in N_{H_0}$ because $\ell_0((G,G))<M_{(c,c)}$. Thus
${}^{\sigma}u_2\in M_Q\cap N_H$. Then $\hat{w}=w_1$ and $\hat{u}=u_2$ satisfy the required properties.
\end{proof}

\begin{lemma}\label{lemma:GL easier condition on psiU}
Let $h=wu$, where $w$ and $u$ are given by Proposition~\ref{proposition:GL structure of w u}. Assume
\begin{align}\label{GL psi U nontrivial using only w}
\psi_U|_{U\cap {}^{w^{-1}}U_P}\ne1.
\end{align}
Then \eqref{psi U nontrivial} also holds, i.e., $\psi_U|_{U\cap {}^{h^{-1}}U_P}\ne1$.
\end{lemma}
\begin{proof}
The proof is a repetition of the proof of Lemma~\ref{lemma:easier condition on psiU}, with only one case to consider. In the notation of that proof, we only have to consider the case where the coordinate $(i,j)$ defining $Y$ belongs to a block $B\in\Mat_c$. Then ${}^{\sigma}Y$ is
also defined by a coordinate in the same block $B$. One change here is that
$\sigma\in (W(G),W(G))$ instead of $(W(G),1)$, but this does not make any difference. In fact, $(1,G)$ commutes with all of the blocks of $U$ where $\psi_U$ is nontrivial.
\end{proof}

Re-denote $h=wu$ where $w$ and $u$ satisfy the properties of Proposition~\ref{proposition:GL structure of w u}. In particular $w$ defines the integers $0\leq l',l\leq c$. Write
\begin{align*}
w=(w^1_{k},\ldots,w^1_{1},w^2_{1},\ldots,w^2_{k}),\qquad \forall i,j,w^j_i\in\{0,1\}^c.
\end{align*}
With this notation
\begin{align*}
p_0(w)=(w^1_1,w^2_1),\qquad w^1_1=(0^{l},1^{c-l}),\qquad w^2_1=(0^{c-l'},1^{l'}).
\end{align*}

\begin{proposition}\label{proposition:GL 1st reduction of w}
We have $\mathcal{H}(h)=0$ unless
\begin{align}\label{eq:GL w_i first reduction}
w^1_i=(0^l,*^{c-l}),\qquad w^2_i=(*^{l},1^{c-{l}}),\qquad \forall 1<i\leq k.
\end{align}
\end{proposition}
\begin{proof}
For $k=1$ there is nothing to prove, assume $k>1$.
Since $w^1_1=(0^{l},1^{c-{l}})$, conjugation by $w$ leaves the last $c-{l}$ rows of $u^{2,2}$ (see \eqref{matrix:middle 4c block of u}) in $U_P$, these are rows
\begin{align*}
(k-1)c+{l}+1,\ldots,kc.
\end{align*}
The character $\psi_U$ is nontrivial on the bottom right $c-{l}\times c-{l}$ block of $u^{2,2}$, whence
$\mathcal{H}(h)=0$ by \eqref{GL psi U nontrivial using only w}, unless $w^{-1}$ permutes the last $c-{l}$ columns of $u^{2,2}$, columns
\begin{align*}
(k+1)c+{l}+1,\ldots,(k+2)c,
\end{align*}
outside of $U_P$. This means $w$ permutes rows
\begin{align*}
(k+1)c+{l}+1,\ldots,(k+2)c
\end{align*}
into $U_P$, i.e., $w^2_2=(*^{l},1^{c-{l}})$. But these are also the last $c-{l}$ rows of
the block $v_{1,2}$ of the bottom right copy of $V_{(c^{k-1})}$ in $U$. Since $\psi_U$ restricts to $\psi\circ\tr$ on $v_{1,2}$,
$\mathcal{H}(h)=0$ by \eqref{GL psi U nontrivial using only w}, unless $w^{-1}$ permutes the last $c-{l}$ columns of $v_{1,2}$, columns
\begin{align*}
(k+2)c+{l}+1,\ldots,(k+3)c,
\end{align*}
outside of $U_P$, so that $w$ permutes rows
\begin{align*}
(k+2)c+{l}+1,\ldots,(k+3)c
\end{align*}
into $U_P$, $w^2_3=(*^{l},1^{c-{l}})$, and these are the bottom $c-{l}$ rows of $v_{2,3}$. Proceeding similarly ($\psi_U$ is $\psi\circ\tr$ on $v_{j,j+1}$) we obtain $w^2_i=(*^{l},1^{c-{l}})$ for all $1<i\leq k$.

In addition, $w^1_1=(0^{l},1^{c-{l}})$ implies $w^{-1}$ permutes the first $l$ columns of $u^{1,1}$, columns
\begin{align*}
(k-1)c+1,\ldots,(k-1)c+l,
\end{align*}
into $U_P$. Since $\psi_U$ restricts to $\psi\circ\tr$ on $u^{1,1}$, $\mathcal{H}(h)=0$ by \eqref{GL psi U nontrivial using only w}, unless $w$ permutes the first $l$ rows of $u^{1,1}$ outside of $U_P$, these are rows
\begin{align*}
(k-2)c+1,\ldots,(k-2)c+l,
\end{align*}
and we obtain $w^1_2=(0^l,*^{c-l})$. Then $w^{-1}$ permutes the first $l$ columns of the block $v_{k-1,k}$ of the top left copy of $V_{(c^{k-1})}$ in $U$, columns
\begin{align*}
(k-2)c+1,\ldots,(k-2)c+l
\end{align*}
into $U_P$, so that $\mathcal{H}(h)=0$ by \eqref{GL psi U nontrivial using only w},
unless $w$ permutes the first $l$ rows of $v_{k-1,k}$, rows
\begin{align*}
(k-3)c+1,\ldots,(k-3)c+l,
\end{align*}
outside of $U_P$, i.e., $w^1_3=(0^l,*^{c-l})$. Similarly, we deduce $w^1_i=(0^l,*^{c-l})$ for all $1<i\leq k$.
\end{proof}

For each $1<i\leq k$, let $0\leq d^1_{i-1}\leq c-l$ and $0\leq d^2_{i-1}\leq l$
be maximal such that for all $1<i\leq k$,
\begin{align*}
w^1_i=(0^{l+d^1_{i-1}},*^{c-l-d^1_{i-1}}),\qquad w^2_i=(*^{l-d^2_{i-1}},1^{c-l+d^2_{i-1}}).
\end{align*}
The integer $d^j_{i-1}$ is defined since $w^j_i$ takes the form \eqref{eq:GL w_i first reduction}.
\begin{proposition}\label{proposition:GL 2nd reduction of w}
We have $\mathcal{H}(h)=0$ unless $h\sim \hat{w}\hat{u}$, $p_0(\hat{w})=(0^l,1^{c-l},0^{c-l'},1^{l'})$, for each $1<i\leq k$,
\begin{align}\label{eq:GL w_i second reduction}
w^1_i=(0^{l+d^1_{i-1}},1^{c-l-d^1_{i-1}}),\quad w^2_i=(0^{l-d^2_{i-1}},1^{c-l+d^2_{i-1}}),\quad d^1_1\leq\ldots\leq d^1_{k-1},
\quad d^2_1\leq\ldots\leq d^2_{k-1},
\end{align}
and $\hat{u}$ satisfies the conditions of Proposition~\ref{proposition:GL structure of w u}, in particular
$\ell_0(\hat{u})$ takes the form \eqref{eq:GL second form u final}.
\end{proposition}
\begin{proof}
For each $1<i\leq k$, put $w^2_i=((w^2)'_i,1^{c-l})$ with $(w^2)'_i\in\{0,1\}^l$. Let $1\leq j\leq l$ and assume $1<i_0\leq k$ is minimal such that $(w^2)'_{i_0}[j]=1$. Assume $i>i_0$ is minimal with $(w^2)'_i[j]=0$. Since
$(w^2)'_{i-1}[j]=1$, $w$ permutes row $(k+i-2)c+j$ into $U_P$. This row contains coordinates of a row from $v_{i-2,i-1}$, and $\psi_U$ is $\psi\circ\tr$ on $v_{i-2,i-1}$, so that on row $(k+i-2)c+j$ it is nontrivial on the column $(k+i-1)c+j$. Thus \eqref{GL psi U nontrivial using only w} implies
$\mathcal{H}(h)=0$, unless $w^{-1}$ permutes column $(k+i-1)c+j$ outside of $U_P$, which means that $w$ permutes row $(k+i-1)c+j$ into $U_P$, contradicting the assumption $(w^2)'_i[j]=0$. Therefore $(w^2)'_i[j]=1$ for all $i\geq i_0$ (or $\mathcal{H}(h)=0$).

Now we are in a situation similar to the proof of Proposition~\ref{proposition:2nd reduction of w}. If $i_0$ is minimal with
$(w^2)'_{i_0}[j]=1$ and $(w^2)'_{i_0}[j+1]=0$, then for each $i>i_0$,
either $(w^2)'_{i}[j]=1,(w^2)'_{i}[j+1]=0$ or $(w^2)'_{i}[j]=(w^2)'_{i}[j+1]=1$. Using transpositions from
$(\diag(W(\GL_{l}),I_{c-l}),1)$, one can sort the coordinates of the blocks $w^2_i$ so that $d^2_1\leq\ldots\leq d^2_{k-1}$ holds.
Any  $b\in(\diag(W(\GL_{l}),I_{c-l}),1)$ fixes the last $c-l$ rows of $w^1_i$ and keeps the first $l$ rows of
$w^1_{i}$ in $w^1_{i}$, for each $1<i\leq k$, and since $w^1_i$ starts with $(0^l)$, $p_1({}^{b}w)=p_1(w)$ (for brevity, we identify $w^1_i$ with the rows it is affecting: $(k-i)c+1,\ldots,(k-i)c+c$). Additionally $b$ fixes the last $2c-l$ rows of $p_0(w)$ while keeping the first $l$ rows in $p_0(w)$, thus $p_0({}^{b}w)=p_0(w)$.

Similarly denote $w^1_i=(0^l,(w^1)'_i)$ with $(w^1)'_i\in\{0,1\}^{c-l}$, and consider $1\leq j\leq c-l$.
Suppose $i_0>1$ is minimal with $(w^1)'_{i_0}[j]=0$ and $i>i_0$ is minimal with $(w^1)'_{i}[j]=1$. On the one hand
$(w^1)'_{i-1}[j]=0$ hence $w$ permutes row $(k-i+1)c+j$ outside of $U_P$, so that
$w^{-1}$ permutes column $(k-i+1)c+j$ into $U_P$. On the other hand
$(w^1)'_{i}[j]=1$, whence $w$ permutes row $(k-i)c+j$ into $U_P$. Again $\mathcal{H}(h)=0$ because of \eqref{GL psi U nontrivial using only w}, otherwise we deduce that if $i_0$ exists, $(w^1)'_{i}[j]=0$ for all $i\geq i_0$.

This means that unless $\mathcal{H}(h)=0$, if $i_0$ is minimal with
$(w^1)'_{i_0}[j]=1$ and $(w^2)'_{i_0}[j+1]=0$, then for each $i>i_0$,
either $(w^2)'_{i}[j]=1,(w^2)'_{i}[j+1]=0$ or $(w^2)'_{i}[j]=(w^2)'_{i}[j+1]=0$.
Again we use transpositions, now from $(\diag(I_l,W(\GL_{c-l})),1)$, to rearrange the coordinates of the blocks $w^1_i$ and obtain $d^1_1\leq\ldots\leq d^1_{k-1}$. For $b\in(\diag(I_l,W(\GL_{c-l})),1)$, $b$ fixes the first $l$ rows of $w^2_{i}$ and leaves the last $c-l$ rows in $w^2_{i}$ for $1<i\leq k$, and because $w^2_{i}$ (still) ends with $(1^{c-l})$, $p_2({}^bw)=w$. Also
$p_0({}^bw)=p_0(w)$.

Now condition \eqref{eq:GL w_i second reduction} holds, and note that the conjugations affect $u$, but it still satisfies the conditions of
Proposition~\ref{proposition:GL structure of w u}. As opposed to the proof of Proposition~\ref{proposition:2nd reduction of w}, we do not claim $\ell_0(\hat{u})=\ell_0(u)$, it might not hold because $(\diag(W(\GL_{l}),I_{c-l}),1)$ does not commute with \eqref{eq:GL second form u final}.
\end{proof}
Re-denote $h=wu$, with $w$ and $u$ given by Proposition~\ref{proposition:GL 2nd reduction of w}.
Since the total sum of coordinates of $w$ must be $kc$, we have
\begin{align*}
c-l+l'+\sum_{i=1}^{k-1}(c-l-d^1_i)+\sum_{i=1}^{k-1}(c-l+d^2_i)=kc,
\end{align*}
hence
\begin{align}\label{eq:GL compatibility for w}
\sum_{i=1}^{k-1}d^2_i-d^1_i=(k-1)(2l-c)+l-l'.
\end{align}
We can multiple $w$ on the left by an element of $W(M_P)$, so that the matrix corresponding to $w$ takes the form
\begin{align*}
&\left(\begin{smallmatrix}
0&&&&&&&&I_{c-l+d^2_{k-1}}\\
&I_{c-l-d^1_{k-1}}&&&&&&0&\\
&&0&&&&\iddots &&\\
&&&\ddots&&0&&&\\
&&&&\left(\begin{smallmatrix}&&&I_{l}\\&I_{2c-l-l'}\\I_{l'}\end{smallmatrix}\right)&&&&\\
&&&0&&I_{l-d^2_1}&&&\\
&&\iddots&&&&0&\\
&0&&&&&&\ddots&\\
I_{l+d^1_{k-1}}&&&&&&&&0
\end{smallmatrix}\right).
\end{align*}
For example if $k=2$,
\begin{align*}
w=\left(\begin{smallmatrix}
          0 & 0 & 0 & 0 & 0 & 0 & I_{c-l+d^2_1} \\
          0 & I_{c-l-d^1_1} & 0 & 0 & 0 & 0 & 0 \\
          0 & 0 & 0 & 0 & I_{l} & 0 & 0 \\
          0 & 0 & 0 & I_{2c-l-l'} & 0 & 0 & 0 \\
          0 & 0 & I_{l'} & 0 & 0 & 0 & 0 \\
          0 & 0 & 0 & 0 & 0 & I_{l-d^2_1} & 0 \\
          I_{l+d^1_1} & 0 & 0 & 0 & 0 & 0 & 0
        \end{smallmatrix}\right),
\end{align*}
and for $k=3$,
\begin{align*}
w=\left(\begin{smallmatrix}
          0 & 0 & 0 & 0 & 0 & 0 & 0 &  0 & 0 & 0 & I_{c-l+d^2_2} \\
          0 & I_{c-l-d^1_2} & 0 & 0 &  0 & 0 & 0 & 0 & 0 & 0 & 0 \\
          0 & 0 & 0 & 0 & 0 & 0 & 0 &  0 & I_{c-l+d^2_1} & 0 & 0 \\
          0 & 0 & 0 & I_{c-l-d^1_1} &  0 & 0 & 0 & 0 & 0 & 0 & 0 \\
          0 & 0 & 0 & 0 & 0 & 0 &  I_{l'} & 0 & 0 & 0 & 0 \\
          0 & 0 & 0 & 0 & 0 & I_{2c-l-l'} & 0 & 0 & 0 & 0 & 0 \\
          0 & 0 & 0 & 0 & I_l & 0  & 0 & 0 & 0 & 0 & 0 \\
          0 & 0 & 0 & 0 & 0 & 0 &  0 & I_{l-d^2_1} & 0 & 0 & 0 \\
          0 & 0 & I_{l+d^1_1} & 0  & 0 & 0 & 0 & 0 & 0 & 0 & 0 \\
          0 & 0 & 0 & 0 & 0 & 0 &  0 & 0 & 0 & I_{l-d^2_2} & 0 \\
          I_{l+d^1_2} & 0 & 0 & 0  & 0 & 0 & 0 & 0 & 0 & 0 & 0
        \end{smallmatrix}\right).
\end{align*}

For $1\leq j\leq k-1$, denote
\begin{align*}
\gamma_j=&\diag(I_{\sum_{i=1}^{j-1}c-l-d^1_{k-i}},\left(\begin{smallmatrix}&I_{kc-(2j-1)(c-l)-d^2_{k-j}
+\sum_{i=1}^{j-1}d^1_{k-i}-d^2_{k-i}}\\I_{c-l+d^2_{k-j}}\end{smallmatrix}\right),I_{\sum_{i=1}^{j-1}c-l+d^2_{k-i}})\in W(\GL_{kc}),\\
\gamma_j'=&\diag(I_{\sum_{i=1}^{j-1}l+d^1_{k-i}},\left(\begin{smallmatrix}&I_{l+d^1_{k-j}}\\
I_{kc-(2j-1)l-d^1_{k-j}
+\sum_{i=1}^{j-1}d^2_{k-i}-d^1_{k-i}}\end{smallmatrix}\right),I_{\sum_{i=1}^{j-1}l-d^2_{k-i}})\in W(\GL_{kc}),
\end{align*}
and multiply $w$ on the left by $\diag(\gamma_{k-1}\cdot\ldots \cdot\gamma_1,\gamma_{k-1}'\cdot\ldots \cdot\gamma_1')$.
Now we see that ${}^hU\cap M_P=V_{\beta}\times V_{\beta'}$, where $\beta$ and $\beta'$ are the compositions of $kc$ given by
\begin{align*}
&\beta=(c-l-d^1_{k-1},\ldots,c-l-d^1_{1},l'+c-l,c-l+d^2_{1},\ldots,c-l+d^2_{k-1}),\\
&\beta'=(l+d^1_{k-1},\ldots,l+d^1_{1},c-l'+l,l-d^2_{1},\ldots,l-d^2_{k-1}).
\end{align*}
Both $\beta$ and $\beta'$ are indeed compositions of $kc$, by \eqref{eq:GL compatibility for w}.
Put $\psi_{V_{\beta}\times V_{\beta'}}={}^{h}\psi_U|_{V_{\beta}\times V_{\beta'}}$,
$\psi_{V_{\beta}}={}^{h}\psi_U|_{V_{\beta}\times I_{kc}}$ and $\psi_{V_{\beta'}}={}^{h}\psi_U|_{I_{kc}\times V_{\beta'}}$, and note that
\begin{align}\label{eq:GL psi beta factors through tensor}
J_{V_{\beta}\times V_{\beta'},\psi_{V_{\beta}\times V_{\beta'}}}(\rho)=
J_{V_{\beta},\psi_{V_{\beta}}}(\rho_1)\otimes J_{V_{\beta'},\psi_{V_{\beta'}}}(\rho_2).
\end{align}
We start with describing ${}^{w\ell_0(u)}\psi_U|_{V_{\beta}\times V_{\beta'}}$.
For $v\in V_{\beta}$ and $v'\in V_{\beta'}$, write
\begin{align*}
&v=\left(\begin{smallmatrix}I_{c-l-d^1_{k-1}}&b_1&\cdots\\&\ddots&\ddots&\\
&&I_{c-l-d^1_2}&b_{k-2}&\cdots\\
&&&I_{c-l-d^1_1}&e&b_{k-1}&\cdots\\
&&&&I_{l'}&0&b_{k}&\cdots\\
&&&&&I_{c-l}&b_{k+1}&\cdots\\
&&&&&&\ddots&\ddots\\
&&&&&&&I_{c-l+d^2_{k-2}}&b_{2k-1}\\
&&&&&&&&I_{c-l+d^2_{k-1}}
\end{smallmatrix}\right),\\
&v'=\left(\begin{smallmatrix}I_{l+d^1_{k-1}}&b'_1&\cdots\\&\ddots&\ddots&\\
&&I_{l+d^1_2}&b'_{k-2}&\cdots\\
&&&I_{l+d^1_1}&e'&b'_{k-1}&\cdots\\
&&&&I_{c-l'}&0&f'&\cdots\\
&&&&&I_{l}&b'_{k}&\cdots\\
&&&&&&\ddots&\ddots\\
&&&&&&&I_{l-d^2_{k-2}}&b'_{2k-2}\\
&&&&&&&&I_{l-d^2_{k-1}}
\end{smallmatrix}\right).
\end{align*}
The dimensions of the blocks $b_i$ and $b'_i$ are clear, note that $b_k$ and $b_{k+1}$ (resp., $b'_k$) have $c-l+d^2_{1}$ (resp., $l-d^2_{1}$) columns. Then
\begin{align}\label{GL psi_U on V beta 0}
&{}^{w\ell_0(u)}\psi_U(\diag(v,I_{kc}))=\psi(-\sum_{j=k-1}^{2}\tr(b_{k-j}\left(\begin{smallmatrix}0_{d^1_j-d^1_{j-1} \times  c-l-d^1_j} \\ I_{c-l-d^1_{j}}\end{smallmatrix}\right))
-\tr(b_{k-1}\left(\begin{smallmatrix}0_{d^1_1\times c-l-d^1_1}\\I_{c-l-d^1_1}\end{smallmatrix}\right))
\\&\qquad
-\tr(b_{k}\left(\begin{smallmatrix}A(X)\\0_{c-l\times l'}\end{smallmatrix}\right))
+\tr(b_{k+1}\left(
\begin{smallmatrix}0_{d^2_1\times c-l}\\I_{c-l}\end{smallmatrix}\right))
-\sum_{j=1}^{k-2}\tr(b_{k+1+j}\left(\begin{smallmatrix}0_{d^2_{j+1}-d^2_j\times c-l+d^2_{j}}\\I_{c-l+d^2_j}\end{smallmatrix}\right))),\nonumber
\end{align}
\begin{align}\label{GL psi_U on V beta' 0}
{}^{w\ell_0(u)}\psi_U(\diag(I_{kc},v'))=&\psi(-\sum_{j=k-1}^{2}\tr(b'_{k-j}\left(\begin{smallmatrix}
I_{l+d^1_{j-1}}&0_{l+d^1_{j-1}\times d^1_{j}-d^1_{j-1}}\end{smallmatrix}\right))
-\tr(b'_{k-1}\left(\begin{smallmatrix}I_l&0_{l\times d^1_1}\end{smallmatrix}\right))
\\&\quad
+\tr(b'_{k}\left(
\begin{smallmatrix}I_{l-d^2_1}&0_{l-d^2_1\times d^2_1}\end{smallmatrix}\right))
-\sum_{j=1}^{k-2}\tr(b'_{k+j}\left(\begin{smallmatrix}I_{l-d^2_{j+1}}&0_{l-d^2_{j+1}\times d^2_{j+1}-d^2_{j}}\end{smallmatrix}\right)))\nonumber.
\end{align}
In both formulas, the sums $\sum_{j=k-1}^{2}$ are omitted when $k=2$.
The matrix $A(X)\in\Mat_{d^2_1\times l'}$ in \eqref{GL psi_U on V beta 0} is uniquely defined given the block $X$ of $\ell_0(u)$, and in particular
$A(0)=0$ and when $d^2_1=l=l'$, $A(I_l)=I_l$. For $l=c$, we have $d^1_i=0$ for all $i$, and since $d^2_i\leq c$ and $l'\leq c$,
\eqref{eq:GL compatibility for w} implies $d^2_i=c$ for all $i$ and $l=l'$, then
\eqref{GL psi_U on V beta' 0} becomes $\psi_k^{-1}$, and when $X=I_c$, \eqref{GL psi_U on V beta 0} also becomes $\psi_k^{-1}$.

\begin{proposition}\label{proposition:GL wu_0 nontrivial implies h nontrivial orbit}
Assume $k>1$ and $\mathcal{H}(h)\ne0$.
The character $\psi_{V_{\beta}}$ belongs to the orbit of
\begin{align}\label{GL psi_U on V beta}
v\mapsto&\psi(-\sum_{j=k-1}^{2}\tr(b_{k-j}*)
-\tr(b_{k-1}\left(\begin{smallmatrix}*_{d^1_1\times c-l-d^1_1}\\I_{c-l-d^1_1}\end{smallmatrix}\right))
\\&\quad
-\tr(b_{k}\left(\begin{smallmatrix}{*}_{d_1^2\times l'}\\0_{c-l\times l'}\end{smallmatrix}\right))
+\tr(b_{k+1}\left(
\begin{smallmatrix}0_{d^2_1\times c-l}\\I_{c-l}\end{smallmatrix}\right))
-\sum_{j=1}^{k-2}\tr(b_{k+1+j}\left(\begin{smallmatrix}*_{d^2_{j+1}-d^2_j\times c-l+d^2_{j}}\\I_{c-l+d^2_j}\end{smallmatrix}\right))),\nonumber
\end{align}
and $\psi_{V_{\beta'}}$ belongs to the orbit of
\begin{align}\label{GL psi_U on V beta'}
v'\mapsto &\psi(-\sum_{j=k-1}^{2}\tr(b'_{k-j}\left(\begin{smallmatrix}
I_{l+d^1_{j-1}}&{*}_{l+d^1_{j-1}\times d^1_{j}-d^1_{j-1}}\end{smallmatrix}\right))
-\tr(b'_{k-1}\left(\begin{smallmatrix}I_l&0_{l\times d^1_1}\end{smallmatrix}\right))
\\&\quad
+\tr(b'_{k}\left(
\begin{smallmatrix}I_{l-d^2_1}&{*}_{l-d^2_1\times d^2_1}\end{smallmatrix}\right))
-\sum_{j=1}^{k-2}\tr(b'_{k+j}*))\nonumber.
\end{align}
Here $*$ means undetermined block entries. When $\ell(u)$ is the identity, \eqref{GL psi_U on V beta}--\eqref{GL psi_U on V beta'} coincide with \eqref{GL psi_U on V beta 0}--\eqref{GL psi_U on V beta' 0}.
\end{proposition}
\begin{proof}
The proof is similar to the proof of Proposition~\ref{proposition:wu_0 nontrivial implies h nontrivial orbit}.
Now $\psi_U$ is defined by $2(k-1)$ blocks in $\Mat_c$.
Let $B_{1,k-2}$ (resp., $B_{2,0}$) be the block corresponding to $u^{1,1}$ (resp., $u^{2,2}$), and $B_{1,0},\ldots,B_{1,k-3}$
(resp., $B_{2,1},\ldots,B_{2,k-2}$) be the blocks corresponding to the top left (resp., bottom right) embedding of $V_{(c^{k-1})}<M_P$
(see \S~\ref{Doubling setup}).

Set $d^1_0=d^2_0=0$. For $0\leq i\leq k-2$, write $B_{1,i}$ as the upper right block of
\begin{align*}
\left(\begin{smallmatrix}
  I_{l+d^1_{k-i-2}}&&&B_{1,i}^{1,1}&B_{1,i}^{1,2}&B_{1,i}^{1,3}\\
  &I_{d^1_{k-i-1}-d^1_{k-i-2}}&&B_{1,i}^{2,1}&B_{1,i}^{2,2}&B_{1,i}^{2,3}\\
  &&I_{c-l-d^1_{k-i-1}}&B_{1,i}^{3,1}&B_{1,i}^{3,2}&B_{1,i}^{3,3}\\
  &&&I_{l+d^1_{k-i-2}}\\
  &&&&I_{d^1_{k-i-1}-d^1_{k-i-2}}\\
  &&&&&I_{c-l-d^1_{k-i-1}}
\end{smallmatrix}\right)
\end{align*}
and $B_{2,i}$ as the upper right block of
\begin{align*}
\left(\begin{smallmatrix}
  I_{l-d^2_{i+1}}&&&B_{2,i}^{1,1}&B_{2,i}^{1,2}&B_{2,i}^{1,3}\\
  &I_{d^2_{i+1}-d^2_i}&&B_{2,i}^{2,1}&B_{2,i}^{2,2}&B_{2,i}^{2,3}\\
  &&I_{c-l+d^2_i}&B_{2,i}^{3,1}&B_{2,i}^{3,2}&B_{2,i}^{3,3}\\
  &&&I_{l-d^2_{i+1}}\\
  &&&&I_{d^2_{i+1}-d^2_i}\\
  &&&&&I_{c-l+d^2_i}
\end{smallmatrix}\right).
\end{align*}
Recall $\psi_U$ is given by
\begin{align}\label{eq:GL blocks of psi_U}
\psi(-\sum_{i=0}^{k-2}\sum_{t=1}^{3}\tr(B_{1,i}^{t,t})+\sum_{t=1}^{3}\tr(B_{2,0}^{t,t})
-\sum_{i=1}^{k-2}\sum_{t=1}^{3}\tr(B_{2,i}^{t,t})).
\end{align}
Let $\mathscr{M}_P$ (resp., $\mathscr{U}_P$, $\mathscr{U}_P^-$) denote the list of blocks conjugated by $w$ into $M_P$ (resp., $U_P$, $U_P^-$):
\begin{align*}
\mathscr{M}_P=&\{B_{j,i}^{1,1},B_{j,i}^{2,1},B_{j,i}^{3,2},B_{j,i}^{3,3}:1\leq j\leq 2,0\leq i\leq k-2\},\\
\mathscr{U}_P=&\{B_{j,i}^{3,1}:1\leq j\leq 2,0\leq i\leq k-2\},\\
\mathscr{U}_P^-=&\{B_{j,i}^{1,2},B_{j,i}^{1,3},B_{j,i}^{2,2},B_{j,i}^{2,3}:1\leq j\leq 2,0\leq i\leq  k-2\}.
\end{align*}
We can assume $\ell_j(u)=\diag(z_{j,0},\ldots,z_{j,k-2})$, $j=1,2$, with
$z_{j,i}\in{}^{w_{\sigma}}N_{\GL_c}$. Here $w_{\sigma}$ is the projection of $(\sigma,1)^{-1}$ into the $i$-th copy of $\GL_c$ ($(1,\sigma)$
commutes with $\ell(u)$). We can then write $z_{j,i}=z_{j,i}'m_{j,i}$ where
\begin{align*}
{}^w\diag(m_{1,0},\ldots,m_{1,k-2},I_{2c},m_{2,0},\ldots,m_{2,k-2})\in M_P
\end{align*}
and for all $0\leq i\leq k-2$,
\begin{align*}
&m_{1,i}=
\left(\begin{smallmatrix}
I_{l+d^1_{k-i-1}}&M_{1,i}^{1}\\
M_{1,i}^{2}&I_{c-l-d^1_{k-i-1}}+M_{1,i}^{2}M_{1,i}^{1}
\end{smallmatrix}\right)\in\GL_c,\qquad
m_{2,i}=
\left(\begin{smallmatrix}
I_{l-d^2_{i+1}}&M_{2,i}^{1}\\
M_{2,i}^{2}&I_{c-l+d^2_{i+1}}+M_{2,i}^{2}M_{2,i}^{1}
\end{smallmatrix}\right)\in\GL_c,\\
&I_{c-l-d^1_{k-i-1}}+M_{1,i}^{2}M_{1,i}^{1}\in\GL_{l+d^1_{k-i-1}},\qquad
I_{c-l+d^2_{i+1}}+M_{2,i}^{2}M_{2,i}^{1}\in\GL_{c-l+d^2_{i+1}}.
\end{align*}
Then
\begin{align*}
&m_{1,i}^{-1}=
\left(\begin{smallmatrix}
I_{l+d^1_{k-i-1}}+M_{1,i}^{1}M_{1,i}^{2}&-M_{1,i}^{1}\\
-M_{1,i}^{2}&I_{c-l-d^1_{k-i-1}}
\end{smallmatrix}\right)\in\GL_c,\qquad
m_{2,i}^{-1}=
\left(\begin{smallmatrix}
I_{l-d^2_{i+1}}+M_{2,i}^{1}M_{2,i}^{2}&-M_{2,i}^{1}\\
-M_{2,i}^{2}&I_{c-l+d^2_{i+1}}
\end{smallmatrix}\right)\in\GL_c.
\end{align*}
Henceforth we assume $z_{j,i}=m_{j,i}$.
To compute ${}^{\ell(u)}\psi_U$ we calculate
\begin{align*}
m_{1,k-2}^{-1}B_{1,k-2},\quad m_{1,i}^{-1}B_{1,i}m_{1,i+1},\quad
B_{2,0}m_{2,0},\quad m_{2,i}^{-1}B_{2,i+1}m_{2,i+1},\quad\forall 0\leq i\leq k-3.
\end{align*}

We start with $\psi_{V_{\beta}}$ and show that it belongs to the orbit of \eqref{GL psi_U on V beta}, otherwise $\mathcal{H}(h)=0$.  This amounts to the description of its restriction to $b_{k-1},\ldots,b_{2k-1}$ and $e$.
The rightmost $c-l$ columns of $b_{k-1}$ consist of $B_{1,k-2}^{3,2}$ and $B_{1,k-2}^{3,3}$ (the leftmost $l'$ columns are conjugated from the $c\times c$ block to the right of $u^{1,1}$). Looking at $m_{1,k-2}^{-1}B_{1,k-2}$, if the top $l$ rows of $M_{1,k-2}^1$ are nonzero,
${}^u\psi_U$ is nontrivial on $B_{1,k-2}^{3,1}\in\mathscr{U}_P$. Hence $\mathcal{H}(h)=0$ by \eqref{psi U nontrivial} in this case.
Also ${}^u\psi_U$ restricts to $\psi\circ\tr$ on $B_{1,k-2}^{3,3}$. Hence
$\psi_{V_{\beta}}$ agrees with \eqref{GL psi_U on V beta} on $b_{k-1}$.

The block $b_{k+1}$ consists of $(B_{2,0}^{3,2},B_{2,0}^{3,3})$ and we consider $B_{2,0}m_{2,0}$. If the last $c-l$ columns of $M_{2,0}^1$ are nonzero, ${}^u\psi_U$ is nontrivial on $B_{2,0}^{3,1}\in\mathscr{U}_P$. Unless $\mathcal{H}(h)=0$, we obtain that the
last $c-l$ columns of $M_{2,0}^1$ are $0$, then it follows that ${}^u\psi_U$ and $\psi_U$ coincide on $(B_{2,0}^{3,2},B_{2,0}^{3,3})$.

Regarding $b_k$, it is conjugated by $w$ from the $c\times c$ block below $u^{2,2}$. Denote this block by $B_0$, which we further
divide by writing it as the upper right block of
\begin{align*}
&\left(\begin{smallmatrix}
  I_{c-l'}&&&B_0^{1,1}&B_0^{1,2}&B_0^{1,3}\\
  &I_{l'-d_1^2}&&B_0^{2,1}&B_0^{2,2}&B_0^{2,3}\\
  &&I_{d_1^2}&B_0^{3,1}&B_0^{3,2}&B_0^{3,3}\\
  &&&I_{l-d_1^2}\\
  &&&&I_{d_1^2}\\
  &&&&&I_{c-l}
\end{smallmatrix}\right).
\end{align*}
Here
\begin{align*}
&B_0^{1,1},B_0^{2,2},B_0^{2,3},B_0^{3,2},B_0^{3,3}\in \mathscr{M}_P,\quad
B_0^{2,1},B_0^{3,1}\in \mathscr{U}_P,\quad
B_0^{1,2},B_0^{1,3}\in \mathscr{U}_P^-.
\end{align*}
The blocks conjugated into $b_k$ are $B_0^{2,2},B_0^{2,3},B_0^{3,2}$ and $B_0^{3,3}$.
The conjugation of $U$ by $\ell(u)$ multiplies $B_0$ on the right by $m_{2,0}^{-1}$.
The restriction of ${}^{\ell_0(u)}\psi_U$ to $B_0^{2,2}$ and $B_0^{3,2}$ is defined by $A(X)$, but
${}^{\ell_0(u)}\psi_U$ can also be nontrivial on $B_0^{2,1}$ or $B_0^{2,2}$ (or we could have, e.g., $d^2_1=0,l$).
We can assume ${}^{\ell_0(u)}\psi_U$ is given on $B_0$ by
\begin{align}\label{eq:psi u0 on B0}
\psi(\tr(\varphi_k\left(\begin{smallmatrix}B_0^{1,1}&B_0^{1,2}&B_0^{1,3}\\
B_0^{2,1}&B_0^{2,2}&B_0^{2,3}\\B_0^{3,1}&B_0^{3,2}&B_0^{3,3}\end{smallmatrix}\right))),\qquad \varphi_k=\left(\begin{smallmatrix}0_{l-d_1^2\times c-l'}&A_1(X)\\
0_{d_1^2\times c-l'}&A(X)\\0_{c-l\times c-l'}&0_{c-l\times l'}\end{smallmatrix}\right),\qquad A_1(X)\in\Mat_{l-d_1^2\times l'}.
\end{align}
Here $A_1(X)$ defines the restriction of ${}^{\ell_0(u)}\psi_U$ to $B_0^{2,1}$ and $B_0^{3,1}$.
When we consider $m_{2,0}\varphi_k$ we see that the restriction of ${}^u\psi_U$ to
$B_0^{2,1},B_0^{3,1}\in\mathscr{U}_P$ is given by the first $l-d_1^2$ rows of $m_{2,0}$ multiplied by the last $l'$ columns of $\varphi_k$, this restriction should vanish, and the restriction to the blocks conjugated into
$b_k$ corresponds to the last $c-l+d_1^2$ rows of $m_{2,0}$ multiplied by the last $l'$ columns of $\varphi_k$. Since the last $c-l$ columns of $M_{2,0}^1$ are $0$, we can denote $M_{2,0}^1=\left(\begin{smallmatrix}\alpha&0_{l-d_1^2\times c-l}\end{smallmatrix}\right)$. Also put
$M_{2,0}^2=\left(\begin{smallmatrix}\beta^1\\\beta^2\end{smallmatrix}\right)$ with $\beta^1\in\Mat_{d_1^2\times l-d_1^2}$, then
\begin{align*}
m_{2,0}\left(\begin{smallmatrix}A_1(X)\\A(X)\\0_{c-l\times l'}\end{smallmatrix}\right)
&=\left(\begin{smallmatrix}I_{l-d_1^2}&\alpha&0\\\beta^1&I_{d_1^2}+\beta^1\alpha&0\\\beta^2&\beta^2\alpha&I_{c-l}\end{smallmatrix}\right)
\left(\begin{smallmatrix}A_1(X)\\A(X)\\0_{c-l\times l'}\end{smallmatrix}\right)=
\left(\begin{smallmatrix}A_1(X)+\alpha A(X)\\\beta^1A_1(X)+I_{d_1^2}+\beta^1\alpha A(X)\\
\beta^2A_1(X)+\beta^2\alpha A(X)\end{smallmatrix}\right).
\end{align*}
Now if $\mathcal{H}(h)\ne0$, we must have $A_1(X)+\alpha A(X)=0$ thereby $\beta^2A_1(X)+\beta^2\alpha A(X)=0$, in which case
$\psi_{V_{\beta}}$ agrees with \eqref{GL psi_U on V beta} on $b_k$. Thus both characters agree on $b_{k-1}$, $b_k$ and $b_{k+1}$.

Consider $b_{k+i}$, $2\leq i\leq k-1$. We multiply $m_{2,i-2}^{-1}B_{2,i-1}m_{2,i-1}$. The block $b_{k+i}$ consists of
$B_{2,i-1}^{3,2}$ and $B_{2,i-1}^{3,3}$. Recall $B_{2,i-1}^{3,1}\in\mathscr{U}_P$. Then $\mathcal{H}(h)=0$ unless the top right
$l-d_i^2\times c-l+d_{i-1}^2$ block of $m_{2,i-1}m_{2,i-2}^{-1}$ is $0$:
\begin{align*}
\left(\begin{smallmatrix}I_{l-d_i^2}&M_{2,i-1}^1\end{smallmatrix}\right)\left(\begin{smallmatrix}-M_{2,i-2}^1\\
I_{c-l+d_{i-1}^2}\end{smallmatrix}\right)=0.
\end{align*}
In this case the restriction of ${}^u\psi_U$ to $(B_{2,i-1}^{3,2},B_{2,i-1}^{3,3})$, which corresponds to the bottom right $c-l+d_i^2\times c-l+d_{i-1}^2$ block of $m_{2,i-1}m_{2,i-2}^{-1}$, is defined by
\begin{align*}
\left(\begin{smallmatrix}M_{2,i-1}^{2}&I_{c-l+d^2_{i}}+M_{2,i-1}^{2}M_{2,i-1}^{1}\end{smallmatrix}\right)\left(\begin{smallmatrix}-M_{2,i-2}^1\\
I_{c-l+d_{i-1}^2}\end{smallmatrix}\right)&=
\left(\begin{smallmatrix}0_{c-l+d_{i}^2\times l-d_i^2}&I_{c-l+d^2_{i}}\end{smallmatrix}\right)\left(\begin{smallmatrix}-M_{2,i-2}^1\\
I_{c-l+d_{i-1}^2}\end{smallmatrix}\right)\\&=
\left(\begin{smallmatrix}*_{d_i^2-d_{i-1}^2\times c-l+d_{i-1}^2}\\I_{c-l+d_{i-1}^2}\end{smallmatrix}\right).
\end{align*}
Therefore $\psi_{V_{\beta}}$ agrees with \eqref{GL psi_U on V beta} on $b_{k+i}$.

Also $\psi_{V_{\beta}}|_e=1$, because $e$ is conjugated from the $c\times c$ block to the right of $u^{1,1}$.
This completes the proof for $\psi_{V_{\beta}}$.

We turn to the restriction of $\psi_{V_{\beta'}}$ to $b'_{1},\ldots,b'_{k}$, $e'$ and $f'$, and prove that unless $\mathcal{H}(h)=0$,
$\psi_{V_{\beta'}}$ and \eqref{GL psi_U on V beta'} coincide. The block $b'_k$ corresponds to $B_{2,0}^{1,1}$ and $B_{2,0}^{2,1}$. Considering $B_{2,0}m_{2,0}$, the restriction of ${}^u\psi_U$ to these blocks is given by the top left $l-d_1^2\times l$ block of
$m_{2,0}$, namely
\begin{align*}
\psi(\tr(\left(\begin{smallmatrix}
           I_{l-d_1^2} & {*}_{l-d_1^2\times d_1^2}
         \end{smallmatrix}\right)\left(\begin{smallmatrix}
           B_{2,0}^{1,1}\\B_{2,0}^{2,1}\end{smallmatrix}\right))).
\end{align*}
Hence $\psi_{V_{\beta'}}$ and \eqref{GL psi_U on V beta'} coincide on $b'_k$.

The block $b'_{k-1}$ is conjugated from $B_{1,k-2}^{1,1}$ and $B_{1,k-2}^{2,1}$. This is similar to $b_{k+1}$.
We multiply $m_{1,k-2}^{-1}B_{1,k-2}$ and if the first $l$ rows of $M_{1,k-2}^1$ are nonzero, ${}^u\psi_U$ is nontrivial on
$B_{1,k-2}^{3,1}\in\mathscr{U}_P$ whence $\mathcal{H}(h)=0$. Henceforth we can assume the first $l$ rows of
$M_{1,k-2}^1$ are $0$, then the top left $l\times l+d_1^1$ block of $m_{1,k-2}^{-1}$ equals $\left(\begin{smallmatrix}I_l&0_{l\times l+d_1^1}\end{smallmatrix}\right)$, so that the restriction of ${}^u\psi_U$ to $B_{1,k-2}^{1,1}$ and $B_{1,k-2}^{2,1}$ coincides with the restriction of $\psi_U$ ($\psi\circ\tr$ on the former, trivial on the latter).

Consider $b'_{i}$, $1\leq i\leq k-2$. We multiply $m_{1,i-1}^{-1}B_{1,i-1}m_{1,i}$.
The block $b'_{i}$ is conjugated from
$(B_{1,i-1}^{1,1},B_{1,i-1}^{2,1})$. This is similar to the case of $b_{k+i}$. Since $B_{1,i-1}^{3,1}\in\mathscr{U}_P$,
$\mathcal{H}(h)=0$ unless the top right $l+d_{k-i-1}^1\times c-l-d_{k-i}^1$ block of $m_{1,i}m_{1,i-1}^{-1}$ is $0$, i.e.,
\begin{align*}
\left(\begin{smallmatrix}
I_{l+d^1_{k-i-1}}&M_{1,i}^{1}
\end{smallmatrix}\right)
\left(\begin{smallmatrix}
-M_{1,i-1}^{1}\\
I_{c-l-d^1_{k-i}}
\end{smallmatrix}\right)=0.
\end{align*}
Then the restriction of ${}^u\psi_U$ to $(B_{1,i-1}^{1,1},B_{1,i-1}^{2,1})$, which corresponds to the top left $l+d_{k-i-1}^1\times l+d_{k-i}^1$ block of $m_{1,i}m_{1,i-1}^{-1}$ becomes
\begin{align*}
\left(\begin{smallmatrix}
I_{l+d^1_{k-i-1}}&M_{1,i}^{1}
\end{smallmatrix}\right)
\left(\begin{smallmatrix}
I_{l+d^1_{k-i}}+M_{1,i-1}^{1}M_{1,i-1}^{2}\\
-M_{1,i-1}^{2}&
\end{smallmatrix}\right)&=
\left(\begin{smallmatrix}
I_{l+d^1_{k-i-1}}&M_{1,i}^{1}
\end{smallmatrix}\right)
\left(\begin{smallmatrix}
I_{l+d^1_{k-i}}\\
0_{c-l-d^1_{k-i}\times l+d^1_{k-i}}&
\end{smallmatrix}\right)
\\&=\left(\begin{smallmatrix}
I_{l+d^1_{k-i-1}}&
{*}_{l+d^1_{k-i-1}\times d^1_{k-i}-d^1_{k-i-1}}
\end{smallmatrix}\right),
\end{align*}
hence $\psi_{V_{\beta'}}$ agrees with \eqref{GL psi_U on V beta'} on $b'_{i}$.

The character $\psi_{V_{\beta'}}$ is trivial on $f'$, because $f'$ is conjugated from $B_0^{1,1}$ (see \eqref{eq:psi u0 on B0}, the top left $l-d_1^2\times c-l'$ block of $\varphi_k$). It is also trivial on $e'$ since it is conjugated from the $c\times c$ block to the right of $u^{1,1}$
(this is similar to $e$).
\end{proof}

\begin{proposition}\label{proposition:GL d_1 < n-l}
Consider $k>1$. Assume $d^1_1<c-l$ (in particular $l<c$) or $d^2_1<l$ (in particular $0<l$). Then $J_{V_{\beta}\times V_{\beta'},\psi_{V_{\beta}\times V_{\beta'}}^{-1}}(\rho)=0$ and $\mathcal{H}(h)=0$.
\end{proposition}
\begin{proof}
We argue as in the proof of Proposition~\ref{proposition:d_1 < n-l}.
By Proposition~\ref{proposition:GL wu_0 nontrivial implies h nontrivial orbit}, we can assume
$\psi_{V_{\beta}}$ (resp., $\psi_{V_{\beta'}}$) is given by \eqref{GL psi_U on V beta} (resp., \eqref{GL psi_U on V beta'}). Let $\varphi$ be the transpose of the nilpotent element defined by $\psi_{V_{\beta}}^{-1}$ (resp., $\psi_{V_{\beta'}}^{-1}$).
By \eqref{eq:GL psi beta factors through tensor} and \cite[Theorems~A, E]{GGS}, because $\rho_1$ (resp., $\rho_2$) is $(k,c)$, it is enough to show that $\varphi$ is nilpotent of order at least $k+1$.

Consider $d^1_1<c-l$.
Looking at \eqref{GL psi_U on V beta}, we have $k$ nontrivial blocks
$b_{k-1},b_{k+1},\ldots,b_{2k-1}$, and for each block, the bottom right coordinate is nontrivial and the other coordinates on its column in $\varphi$ are $0$. This does not depend on the undetermined coordinates of the character. To see this use the assumption $c-l-d^1_1>0$ for $b_{k-1}$ and $l<c$ for $b_{k+1}$, and the bottom right coordinate of $b_{k+1}$ is the only nonzero coordinate on its column in $\varphi$ because on the $l'\times c-l+d^2_1$ block $b_k$ above $b_{k+1}$, $\varphi$ is $0$ on the last $c-l$ columns (if $l'=0$, this is trivially true). It follows that $\varphi$ is nilpotent of order at least $k+1$.

For the case $d^2_1<l$, the blocks $b'_1,\ldots,b'_{k}$ are $k$ nontrivial blocks, the top left coordinate of each block is nontrivial (use $l>0$ and $d^2_1<l$) independently of undetermined coordinates, and is the only nonzero coordinate on its row (for $b'_{k-1}$ use the fact that \eqref{GL psi_U on V beta'} is trivial on $e'$!). Again $\varphi$ is nilpotent of order at least $k+1$.
\end{proof}
\begin{remark}
If $l=l'$, the conditions $d^1_1=c-l$ and \eqref{eq:GL compatibility for w} already imply $d^2_i=l$ for all $i$.
\end{remark}
For the remaining cases $k=1$ or both $d^1_1=c-l$ and $d^2_1=l$, in which case $d^1_i=c-l$ and $d^2_i=l$ for all $i$, whence by \eqref{eq:GL compatibility for w} we have, for all $k\geq1$, $l'=l$. Up to left multiplication by an element of $W(M_P)$, $w$ equals
\begin{align*}
\left(\begin{smallmatrix}
   &  &  &  &  &  & I_c \\
   &  &  &  &  & \iddots &  \\
   &  &  &  & I_c &  &  \\
   &  &  & \left(\begin{smallmatrix}&&I_{l}\\&I_{2(c-l)}\\I_{l}\end{smallmatrix}\right) &  &  &  \\
   &  & I_c &  &  &  &  \\
   & \iddots &  &  &  &  & \\
  I_c &  &  &  &  &  &
\end{smallmatrix}\right),
\end{align*}
so that ${}^w\ell(u)\in P$ and $h\sim w\ell_0(u)$ (we still do not change $w$, in order to use $\beta,\beta'$ and the formulas for the characters given above). Therefore $\psi_{V_{\beta}}$ and $\psi_{V_{\beta'}}$ are already given by
\eqref{GL psi_U on V beta 0} and \eqref{GL psi_U on V beta' 0}.
Considering the action of $(\GL_l,\GL_l)$, where
$\GL_l$ is the natural subgroup of $M_{(l,c-l)}$, we can already write $X=A_{l,j}=\left(\begin{smallmatrix}&I_{j}\\0_{l-j}\end{smallmatrix}\right)$ with $0\leq j\leq l$. We deduce there are only finitely many more representatives to analyze, but as opposed to
\S~\ref{section not GL}, we must handle each $0\leq j\leq l$ separately (i.e., we can not easily reduce to $j=l$).
The form of representatives is finally similar to the case $k=1$. For the representative $h$ such that $j=l=c$ we have $h\sim\delta$.

\begin{proposition}\label{proposition:GL d1 = n-l l < n}
Assume $d^1_1=c-l$ or $k=1$, and $0\leq l<c$. Then $\mathcal{H}(h)=0$ outside a discrete subset of $s$. Furthermore, if $l>0$ (forcing $c>1$) and $\pi_2$ is supercuspidal, or $ck>1$, $l=0$, $\pi_1$ and $\pi_2$ are supercuspidal and $\rho_2=\rho_c(\tau_2)$ for an irreducible supercuspidal representation $\tau_2$ of $\GL_k$, then $\mathcal{H}(h)=0$ for all $s$.
\end{proposition}
\begin{proof}
Now $V_{\beta}=V_{\beta'}=V_{(c^k)}$.
Consider the parabolic subgroup $R<G$ with $M_R=M_{(c-l,l)}$ and $U_R=V_{(c-l,l)}^-$. Note that
$V_{(c-l,l)}^-$ is trivial if $l=0$. Identify the group $\GL_{c-l}$ (nontrivial for $0\leq l<c$) with its natural image in $M_R$.

For convenience, we multiply $w$ on the left by
$\diag(I_{2kc-c},\left(\begin{smallmatrix}&I_{l}\\I_{c-l}\end{smallmatrix}\right))$. This permutation normalizes $V_{\beta}\times V_{\beta'}$, fixes $\psi_{V_{\beta}}$ and conjugates $\psi_{V_{\beta'}}$ into
\begin{align}\label{GL psi_U on V beta''}
&\left(\begin{smallmatrix}
I_c&b'_1\\&\ddots&\ddots\\
&&I_c&b'_{k-1}&e'\\
&&&I_{l}\\&&&&I_{c-l}
\end{smallmatrix}\right)\mapsto\psi(-\sum_{j=k-1}^{2}\tr(b'_{k-j})
-\tr(b'_{k-1}\left(\begin{smallmatrix}I_l&0_{l\times c-l}\end{smallmatrix}\right))).
\end{align}
Now ${}^h(1,\GL_{c-l})=\diag(I_{2kc-(c-l)},\GL_{c-l})$, and since the character \eqref{GL psi_U on V beta''} is trivial on $e'$,
$J_{V_{\beta}\times V_{\beta'},\psi_{V_{\beta}\times V_{\beta'}}^{-1}}(\rho)$ is a well defined representation of ${}^h(1,\GL_{c-l})$.

Over non-archimedean fields, we simultaneously prove that for all $s$, $J_{V_{\beta}\times V_{\beta'},\psi_{V_{\beta}\times V_{\beta'}}^{-1}}(\rho)$ is a trivial representation of ${}^h(1,U_R)$, and admits a finite length filtration as a representation of
${}^h(1,\GL_{n-l})$, where ${}^h(1,C_{\GL_{n-l}})$ acts by a character on each constituent. For archimedean fields we prove that ${}^h(1,\mathfrak{u}_R)$ acts locally nilpotently on $J_{V_{\beta}\times V_{\beta'},\psi_{V_{\beta}\times V_{\beta'}}^{-1}}(\rho)^*$,
and the Lie algebra $\mathfrak{v}_{((k-1)c+l,c-l)}$ of $\diag(I_{kc},V_{((k-1)c+l,c-l)})$ acts locally nilpotently on $J_{V_{\beta}\times V_{\beta'},\psi_{V_{\beta}\times V_{\beta'}}^{-1}}(\rho)^*$. Note that ${}^h(1,\GL_{n-l})$ is a direct factor of
$\diag(I_{kc},M_{((k-1)c+l,c-l)})$. Cf. Lemmas~\ref{lemma:Jacquet module is a trivial rep of U_R} and \ref{lemma:Jacquet module is a finite length}.

Granted that, since ${}^{h^{-1}}(|\det|^{s-1/2})(1,aI_{c-l})=|a|^{-(c-l)(s-1/2)}$, one can apply \eqref{eq:relation for T with s} to deduce $\mathcal{H}(h)=0$ outside a discrete subset of $s$. For $l>0$, if $\pi_2$ is supercuspidal,
$\mathcal{H}(h)=0$ for all $s$ (because $J_{U_R}(\pi_2^{\vee})=0$).

Henceforth we identify $\GL_{kc}$ with the bottom right block of $M_P$.
For $u\in U_R$, ${}^h(1,u)=m_uu'$ with  $u'\in U_P$ and $m_u=\diag(I_{(k-1)c},\left(\begin{smallmatrix}I_{l}&A_{l,j}u\\&I_{c-l}\end{smallmatrix}\right))$.
Let $Z=\diag(I_{(k-1)c},V_{(l,c-l)})$. This (abelian) group stabilizes \eqref{GL psi_U on V beta''}. In addition, the subgroups $\diag(I_{kc-(c-l)},\GL_{c-l})$ and $\diag(\GL_l,I_{c-l})^{\triangle}<\GL_{kc}$ stabilize \eqref{GL psi_U on V beta''}, and act on the characters of $Z$ with $2$ orbits.

Over non-archimedean fields, we show that for any nontrivial character $\chi$ of $Z$,
\begin{align}\label{eq:GL Jacquet functor V beta U_R}
J_{V_{\beta'}\ltimes Z,\psi_{V_{\beta'}}^{-1}\otimes\chi}(\rho_2)=0,
\end{align}
which implies (by \cite[5.9--5.12]{BZ1})
\begin{align}\label{eq:GL Jacquet functor V beta U_R implies}
J_{V_{\beta'},\psi_{V_{\beta'}}^{-1}}(\rho_2)=
J_{V_{\beta'}\ltimes Z,\psi_{V_{\beta'}}^{-1}}(\rho_2).
\end{align}
Thus $J_{V_{\beta'},\psi_{V_{\beta'}}^{-1}}(\rho_2)$ is a trivial representation of ${}^h(1,U_R)$ and this Jacquet module factors through $J_{V_{((k-1)c+l,c-l)}}(\rho_2)$, which is an admissible finite length representation of $M_{((k-1)c+l,c-l)}$. By exactness $J_{V_{\beta'},\psi_{V_{\beta'}}^{-1}}(\rho_2)$ admits a finite length filtration such that on each constituent,
${}^h(1,C_{\GL_{c-l}})$ acts by character. Now by \eqref{eq:GL psi beta factors through tensor}, $J_{V_{\beta}\times V_{\beta'},\psi_{V_{\beta}\times V_{\beta'}}^{-1}}(\rho)$ is a trivial representation of ${}^h(1,U_R)$ and admits a finite filtration with
${}^h(1,C_{\GL_{c-l}})$ acting by a character on each constituent.

For the proof of \eqref{eq:GL Jacquet functor V beta U_R} we can assume $l>0$, otherwise \eqref{eq:GL Jacquet functor V beta U_R implies} is trivial. Identifying $Z$ with $\Mat_{l\times c-l}$, we can assume $\chi$ is nontrivial on the
bottom right coordinate of $Z$, and trivial on the other coordinates of the rightmost column.
Let $\varphi$ be the transpose of the nilpotent element defined by the inverse of \eqref{GL psi_U on V beta''} and by $\chi$. Note that $\varphi$ is independent of $A_{l,j}$. After conjugating $\varphi$ by
$\diag(\left(\begin{smallmatrix}&I_{c-l}\\I_l\end{smallmatrix}\right)^{\triangle'},I_c)$ ($\GL_c^{\triangle'}$ is the diagonal embedding of $\GL_c$ in $\GL_{(k-1)c}$), it has nontrivial entries on the bottom right coordinates of $k$ blocks:
$b'_1,\ldots,b'_{k-1}$ and $Z$, and in each block there is only one nontrivial coordinate on the rightmost column. Therefore $\varphi$ is nilpotent of order at least $k+1$ and \eqref{eq:GL Jacquet functor V beta U_R} holds, because $\rho_2$ is $(k,c)$.

Over archimedean fields we repeat the proof of \eqref{eq:GL Jacquet functor V beta U_R} and apply \cite[Proposition~3.0.1]{GGS2} for each coordinate of $Z$ separately, exactly as in the proofs of
Lemmas~\ref{lemma:Jacquet module is a trivial rep of U_R} and \ref{lemma:Jacquet module is a finite length}.

It remains to prove the stronger assertion when $ck>1$, $l=0$, both $\pi_1$ and $\pi_2$ are supercuspidal and $\rho_2=\rho_c(\tau_2)$ for an irreducible supercuspidal $\tau_2$. Since $\tau_2$ is supercuspidal and $J_{V_{\beta'},\psi_{V_{\beta'}}^{-1}}(\rho_2)$ factors through $J_{V_{((k-1)c,c)}}(\rho_2)=J_{V_{((k-1)c,c)}}(\rho_c(\tau_2))$, we obtain
$\mathcal{H}(h)=0$ for all $s$, unless $c=tk$ for some integer $t\geq1$ (use \cite[2.13 (a)]{BZ2}).

If $t>1$, in particular $c>k$, and we claim $J_{V_{((k-1)c,c)}}(\rho_2)$ is trivial on
${}^h(1,V_{\delta})$ for some composition $\delta$ of $c$, then because $\pi_2$ is supercuspidal, $\mathcal{H}(h)=0$ for all $s$. This follows by repeatedly applying the derivatives of Bernstein and Zelevinsky \cite{BZ1,BZ2} to $J_{V_{((k-1)c,c)}}(\rho_2)$. Indeed for $1\leq i\leq c$,
let $\chi_i$ be the character of $V_{((k-1)c,c-i,1^i)}$ given by $\chi_i(z)=\psi(\sum_{i'=1}^{i}z_{kc-i',kc-i'+1})$. Then either
\begin{align*}
J_{V_{((k-1)c,c)}}(\rho_2)=J_{V_{((k-1)c,c-1,1)}}(\rho_2),
\end{align*}
in which case our claim is proved with $\delta=(c-1,1)$, or
\begin{align*}
J_{V_{((k-1)c,c)}}(\rho_2)=J_{V_{((k-1)c,c-1,1)},\chi_1}(\rho_2).
\end{align*}
We proceed with $i=2$. After $i$ steps, our claim is either proved with $\delta=(c-i,1^i)$, or
\begin{align*}
J_{V_{((k-1)c,c)}}(\rho_2)=J_{V_{((k-1)c,c-i,1^i)},\chi_i}(\rho_2).
\end{align*}
However, since $c>k$, for $i=c$ we already obtain $J_{V_{((k-1)c,1^c)},\chi_c}(\rho_2)=0$, because the highest derivative of $\rho_2$ is $k$ (put differently, the suitable $\varphi$ is nilpotent of order at least $k+1$).

Lastly for $c=k>1$, $J_{V_{((k-1)c,c)}}(\rho_2)=|\det|^{\alpha_1}\rho_{c-1}(\tau_2)\otimes|\det|^{\alpha_2}\tau_2$, where
$\alpha_1,\alpha_2\in\tfrac12\Z$ (see \cite[3.4]{Z3}) and $\rho_{c-1}(\tau_2)$ is $(k,c-1)$, then (because $l=0$)
\begin{align*}
J_{V_{\beta'},\psi_{V_{\beta'}}^{-1}}(\rho_2)=|\det|^{\alpha_1}J_{V_{(c^{k-1})},\psi_{k-1}}(\rho_2)\otimes|\det|^{\alpha_2}\tau_2.
\end{align*}
Hence $\diag(\SL_c^{\triangle'},I_c)$ acts trivially on $J_{V_{\beta'},\psi_{V_{\beta'}}^{-1}}(\rho_2)$. Additionally because
$\psi_{V_{\beta}}^{-1}$ belongs to the orbit of $\psi_k$ ($l=0$, see \eqref{GL psi_U on V beta 0}), $\SL_c^{\triangle}$ acts trivially on
$J_{V_{\beta},\psi_{V_{\beta}}^{-1}}(\rho_1)$. Thus ${}^h(\SL_c,1)$ acts trivially on $J_{V_{\beta}\times V_{\beta'},\psi_{V_{\beta}\times V_{\beta'}}^{-1}}(\rho)$, in particular $\mathcal{H}(h)=0$ for all $s$, because $\pi_1$ is supercuspidal.
\end{proof}
For the remaining cases $l=c$ and $0\leq j\leq c$ (recall $j$ is the rank of $A_{l,j}$). The cases $j<c$ are similar to $l<c$, but involve $V_{\beta}$ and $\psi_{V_{\beta}}$.
\begin{proposition}\label{proposition:GL d1 = n-l j < l = c}
Assume $0\leq j<l=c$ or $k=1$. Then $\mathcal{H}(h)=0$ outside a discrete subset of $s$. Furthermore, if $j>0$ and $\pi_2$ is supercuspidal, or $ck>1$, $j=0$, $\pi_1$ and $\pi_2$ are supercuspidal and $\rho_1=\rho_c(\tau_1)$ for an irreducible supercuspidal representation $\tau_1$ of $\GL_k$, then $\mathcal{H}(h)=0$ for all $s$.
\end{proposition}
\begin{proof}
In this case $X=A_{c,j}=\left(\begin{smallmatrix}&I_{j}\\0_{c-j}\end{smallmatrix}\right)$, so that if we consider
$R<G$ with $M_R=M_{(c-j,j)}$ and $U_R=V_{(c-j,j)}^-$, we can repeat most of the proof of Proposition~\ref{proposition:GL d1 = n-l l < n} (with $j$ instead of $l$), except we use $V_{\beta}$ instead of $V_{\beta'}$ (hence, e.g., $\tau_1$ instead of $\tau_2$).

Identify $\GL_{c-j}$ with its natural image in $M_R$. We have ${}^h(1,\GL_{c-j})=\diag(\GL_{c-j},I_{(2k-1)c})$.
The character $\psi_{V_{\beta}}$ is now given by
\begin{align*}
&\left(\begin{smallmatrix}
I_c&b_k\\
&I_c&b_{k+2}\\
&&\ddots&\ddots\\
&&&I_{c}&b_{2k-1}\\
&&&&I_{c}
\end{smallmatrix}\right)\mapsto\psi(-\tr(b_kA_{c,j})-\sum_{j=2}^{k-1}\tr(b_{k+j})),
\end{align*}
and $\psi_{V_{\beta}'}=\psi_k^{-1}$. Then $J_{V_{\beta}\times V_{\beta'},\psi_{V_{\beta}\times V_{\beta'}}^{-1}}(\rho)$ is a well defined representation of ${}^h(1,\GL_{c-j})$.

We proceed over non-archimedean fields, and prove that for all $s$, $J_{V_{\beta}\times V_{\beta'},\psi_{V_{\beta}\times V_{\beta'}}^{-1}}(\rho)$ is a trivial representation of ${}^h(1,U_R)$ and factors through $J_{V_{(c-j,j+(k-1)c)}}(\rho_1)$.

Since ${}^{h^{-1}}(|\det|^{s-1/2})(1,aI_{c-j})=|a|^{(c-j)(s-1/2)}$, $\mathcal{H}(h)=0$ outside a discrete subset of $s$ by
\eqref{eq:relation for T with s}. For $j>0$ (then $l=c>1$), if $\pi_2$ is supercuspidal,
$\mathcal{H}(h)=0$ for all $s$. For more details and the archimedean case see the proof of Proposition~\ref{proposition:GL d1 = n-l l < n}.

Identify $\GL_{kc}$ with the top left block of $M_P$.
Let $Z=\diag(V_{(c-j,j)},I_{(k-1)c})\cong\Mat_{c-j\times j}$. For $u\in U_R$ we have ${}^h(1,u)=m_uu'$ with $m_u\in Z$ and $u'\in U_P$. The group $Z$ stabilizes $\psi_{V_{\beta}}$, and the set of characters of $Z$ is partitioned into $2$ orbits with respect to the action of
$\diag(\GL_{c-j},I_{kc-(c-j)})$ and $\{\diag(I_{c-j},g,\diag(g,I_{c-j})^{\triangle'}):g\in\GL_j\}$. We show that
for any character $\chi\ne1$ of $Z$,
\begin{align}\label{eq:GL Jacquet functor V beta U_R j<l=c}
J_{V_{\beta}\ltimes Z,\psi_{V_{\beta}}^{-1}\otimes\chi}(\rho_1)=0.
\end{align}
This implies that $J_{V_{\beta}\times V_{\beta'},\psi_{V_{\beta}\times V_{\beta'}}^{-1}}(\rho)$ is a trivial representation of ${}^h(1,U_R)$ and factors through $J_{V_{(c-j,j+(k-1)c)}}(\rho_1)$.

For the proof of \eqref{eq:GL Jacquet functor V beta U_R j<l=c} assume $j>0$. We can assume $\chi$ is nontrivial on the
bottom right coordinate of $Z$, and trivial on the other coordinates on the rightmost column.
Let $\varphi$ be the transpose of the nilpotent element defined by $\psi_{V_{\beta}}^{-1}\otimes\chi$, which now depends on
$A_{l,j}$ (as opposed to the proof of Proposition~\ref{proposition:GL d1 = n-l l < n}). Using a conjugation by
$\diag(I_c,\left(\begin{smallmatrix}&I_{c-j}\\I_j\end{smallmatrix}\right)^{\triangle'})$, we obtain $\varphi$ which has nontrivial entries on the bottom right coordinates of $b_k,b_{k+2},\ldots,b_{2k-1}$ and $Z$ ($k$ blocks). This proves \eqref{eq:GL Jacquet functor V beta U_R j<l=c}, because $\rho_1$ is $(k,c)$.

The proof of the assertion for the case $ck>1$, $j=0$, supercuspidal representations $\pi_1$ and $\pi_2$, and $\rho_1=\rho_c(\tau_1)$ for an irreducible supercuspidal $\tau_1$, proceeds as in the proof of Proposition~\ref{proposition:GL d1 = n-l l < n}.
Since now $J_{V_{(c-j,j+(k-1)c)}}(\rho_1)=J_{V_{(c,(k-1)c)}}(\rho_1)$, $\mathcal{H}(h)=0$ for all $s$ unless $c=tk$, $t\geq1$.
For $t>1$ we use derivatives along $V_{(1^i,c-i,(k-1)c)}$, $1\leq i\leq c$. For $c=k>1$,
$J_{V_{(c,(k-1)c)}}(\rho_1)=|\det|^{\alpha_3}\tau_1\otimes|\det|^{\alpha_4}\rho_{c-1}(\tau_1)$ and $\rho_{c-1}(\tau_1)$ is $(k,c-1)$, hence ${}^h(\SL_c,1)$ acts trivially on $J_{V_{\beta}\times V_{\beta'},\psi_{V_{\beta}\times V_{\beta'}}^{-1}}(\rho)$.
\end{proof}

Propositions~\ref{proposition:GL structure of w u}--\ref{proposition:GL d1 = n-l j < l = c} imply $\mathcal{H}(h)=0$ for all $h$ unless $h\sim\delta$. We prove $\dim\mathcal{H}(\delta)=\dim\Hom_{G}(\chi_0\pi_1^{\vee},\pi_2)$, for all $s$. Now
$P_{\delta}=G^{\iota}\ltimes (V_{(c^k)}\times V_{(c^k)})$ with $G^{\iota}=\{(g,g):g\in G\}$ (for $H=\GL_{2kc}$ one can take $\iota=I_c$, we keep the notation $G^{\iota}$ for uniformity) and
any morphism in $\mathcal{H}(\delta)$ factors through $J_{V_{(c^k)}\times V_{(c^k)},\psi_k\otimes\psi_k}(\rho)$. Hence
\begin{align*}
\mathcal{H}(\delta)&=\Hom_{G^{\iota}}({}^{\delta^{-1}}J_{V_{(c^k)}\times V_{(c^k)},\psi_k\otimes\psi_k}(\rho)\otimes\pi_1^{\vee}\otimes\pi_2^{\vee},1).
\end{align*}
Note that $|\det|^{s-1/2}\otimes|\det|^{-s+1/2}$ and $\theta_h$ are trivial on $G^{\iota}$.

We can assume $\delta$ commutes with $G^{\iota}$ ($G^{\iota}$ is simply the diagonal embedding of $G$ in $H$). Then as a representation of $G^{\iota}$,
\begin{align*}
{}^{\delta^{-1}}J_{V_{(c^k)}\times V_{(c^k)},\psi_k\otimes\psi_k}(\rho)=
J_{V_{(c^k)},\psi_k}(\rho_1)\otimes J_{V_{(c^k)},\psi_k}(\rho_2).
\end{align*}
Recall that the action of $G^{\iota}$ on $J_{V_{(c^k)},\psi_k}(\rho_1)\otimes J_{V_{(c^k)},\psi_k}(\rho_2)$ is given by $g\mapsto \chi_0(\det g)$ ($g\in G$) for some
quasi-character $\chi_0$ of $F^*$ (see \S~\ref{outline}), therefore
\begin{align*}
\Hom_{G^{\iota}}({}^{\delta^{-1}}J_{V_{(c^k)}\times V_{(c^k)},\psi_k\otimes\psi_k}(\rho)\otimes\pi_1^{\vee}\otimes\pi_2^{\vee},1)
=\Hom_{G}(\chi_0\pi_1^{\vee},\pi_2).
\end{align*}
The remaining parts of the proof now follow as in \S~\ref{section not GL}, and note that when
$ck>1$, $\pi_1$ and $\pi_2$ are supercuspidal and $\rho_i=\rho_c(\tau_i)$ for irreducible supercuspidal representations $\tau_i$ of $\GL_k$, $i=1,2$, we do not need to exclude any $s$.

\section{Applications}\label{Applications}
\subsection{Covering groups}\label{Covering groups}
In this section we describe the extension of Theorem~\ref{theorem:uniqueness} to certain covering groups. We proceed with the definitions and notation of \S~\ref{Doubling setup}. Let $m\geq1$. Assume $F^*$ contains the full group of $m$-th roots of unity $\mu_m$. A topological central extension of $G(F)$ by $\mu_m$ is an exact sequence of topological groups
\begin{align*}
1\rightarrow \mu_m\xrightarrow{i} G^{(m)}\xrightarrow{p} G(F)\rightarrow 1,
\end{align*}
where $i$ and $p$ are continuous, $i(\mu_m)$ is closed and belongs to the center of $G^{(m)}$, and $p$ induces an isomorphism $i(\mu_m)\backslash G^{(m)}\cong G(F)$ as topological groups. We call $G^{(m)}$ an $m$-fold covering group of $G(F)$; it is in general not unique, but for $G(F)=\Sp_c(F)$ is it uniquely defined given a Steinberg symbol (e.g., a Hilbert $m$-th order symbol). The covering groups under consideration here were constructed, in increasing level of generality, through a series of works including \cite{Weil2,Kubota,Moore,Stein,Kubota2,Mats,KP,BD}. For further reference see \cite{BLS,McNamara}.

In this section we assume the field is non-archimedean, then $G^{(m)}$ is an $l$-group in the sense of \cite{BZ1}. For $m>2$, an archimedean field is already complex, then the cover is split over the group so that the results in this case are immediate from the linear case. As above, we identify $F$-groups with their $F$-points. Of course this only applies to $G$ and its subgroups; $G^{(m)}$ is not an algebraic group.

In general if $X<G$, $\widetilde{X}$ denotes the covering of $X$ (precisely: of $X(F)$) defined by restriction from $G^{(m)}$. This covering depends on the embedding of $X$ inside $G$. We say that $\widetilde{X}$ is split over $X$ if there is a group embedding $X\rightarrow\widetilde{X}$. If $X$ is perfect (as an $F$-group), such a splitting, if exists, is unique. Note that since $F$ is of characteristic $0$, $\Sp_c$ and $\SL_c$ are perfect.
The coverings under consideration are split canonically over unipotent subgroups, hence the notions of Jacquet functors and unipotent orbits extend to the covering in the obvious way. If $Y$ is a unipotent subgroup of $G$, denote by $\varphi_Y:Y\rightarrow\widetilde{Y}$ the splitting of $Y$. Since $\varphi_Y$ is canonical, we usually omit it from the notation, e.g., if $R<G$ is a parabolic subgroup and we consider a genuine representation $\sigma$ of $\widetilde{M}_R$, for the induced representation $\Ind_{\widetilde{R}}^{G^{(m)}}(\sigma)$ we extend $\sigma$ trivially on $U_R$, which more precisely means on $\varphi_Y(U_R)$. Since we are considering central coverings, $G$ acts on $G^{(m)}$ by conjugation. In particular
\begin{align}\label{eq:covering conjugating canonical splitting}
{}^h\varphi_Y(y)=\varphi_{{}^hY}({}^hy),\qquad\forall y\in Y.
\end{align}

We describe a general system of assumptions for covering groups, under which the doubling construction is well defined, then
state the analog of Theorem~\ref{theorem:uniqueness}. For the particular cases of the covering $\Sp_c^{(m)}$ of \cite{Mats} and
the covering $\widetilde{\GL}_c$ obtained by restriction from $\Sp_{2c}^{(m)}$, these assumptions were verified in \cite{me12}. More details are given below, see also Corollary~\ref{corollary:covering integral props and gamma}.

Fix a covering group $G^{(m)}$. Assume there is a covering $\widetilde{H}$ of $H$ (typically $\widetilde{H}=H^{(m)}$) with the following properties.
\begin{enumerate}[leftmargin=*]
\item \label{covering:GSpin center}For $H=\GSpin_{2kc}$, the preimage $\widetilde{C}_H^{\circ}$ of $C_H^{\circ}$ in $\widetilde{H}$ belongs to the center of $\widetilde{H}$, and $\widetilde{C}_H^{\circ}$ is split over $C_H^{\circ}$. The same properties are satisfied by
    the preimage $\widetilde{C}_G^{\circ}$ of $C_G^{\circ}$ in $\widetilde{G}$.
\item \label{covering:G G e1, e_2}Let $\mathfrak{e}_1(g)=(g,1)$ and $\mathfrak{e}_2(g)=(1,g)$. These are the embeddings of $G$ into $M_Q$ in the linear case. Assume they extend to embeddings $\widetilde{\mathfrak{e}}_i:G^{(m)}\rightarrow\widetilde{\mathfrak{e}_i(G)}$.
\item \label{covering:G G e_2}The restriction of $\widetilde{\mathfrak{e}}_2$ to $\mu_m$ is the identity. (Here we regard $\mu_m$ as a subgroup of $G^{(m)}$.)
\item\label{covering:G G}The images of $\widetilde{\mathfrak{e}}_1$ and $\widetilde{\mathfrak{e}}_2$ commute in $\widetilde{H}$, and give rise to a homomorphism
\begin{align}\label{eq:covering G X G in H}
\{(\epsilon_1,\epsilon_2)\in\mu_m\times\mu_m:\epsilon_1=\epsilon_2\}
\backslash G^{(m)}\times G^{(m)}\rightarrow \widetilde{M}_Q.
\end{align}
This is (automatically) an embedding unless $H=\GSpin_{2kc}$, in which case we further assume that $\widetilde{\mathfrak{e}}_1(z)\widetilde{\mathfrak{e}}_2(z)$ is the identity for $z\in C_G^{\circ}$, then the left hand side of
\eqref{eq:covering G X G in H} is further divided by the subgroup $\{(z,z):z\in C_G^{\circ}\}$ (a subgroup by \eqref{covering:GSpin center}).
Cf. \eqref{embedding G G in H GSpin}. Denote the left hand side of \eqref{eq:covering G X G in H} by $(G,G)^{(m)}$.
\item \label{covering:GL Levi}For $H=\GL_{2kc}$, the preimages of the direct factors $\GL_{kc}$ of $M_P$ commute in $\widetilde{H}$, and the coverings $\widetilde{\GL}_{kc}$ of each copy of $\GL_{kc}$ are isomorphic.
\item \label{covering:GLkc def}Identify $\widetilde{\GL}_{kc}$ with $\widetilde{M}_P$ if $H\ne\GL_{2kc}$, or with the covering of one of the copies of $\GL_{kc}$ in $M_P$ for $H=\GL_{2kc}$.
Assume $\widetilde{\GL}_{kc}$ is split over $\SL_c^{\triangle}$.
\item \label{covering:MP split over GL GL}For $H=\GL_{2kc}$, assume $\widetilde{M}_P$ is split over
$\{\diag(g^{\triangle},g^{\triangle}):g\in\GL_c\}$.
\item \label{covering:inv iota}The involution $\iota$ extends to an involution of $G^{(m)}$ and for a genuine representation $\pi$ of $G^{(m)}$, $(\pi^{\vee})^{\iota}=(\pi^{\iota})^{\vee}$.
\item\label{covering:center of GLl}For any maximal parabolic subgroup $R<G$ whose Levi part contains $\GL_l$, the covering $\widetilde{\GL}_l$ has the property that for a sufficiently large integer $d$, the preimage of $C_{\GL_l}^d=\{x^d:x\in C_{\GL_l}\}$ belongs to the center of $\widetilde{\GL}_l$.
\end{enumerate}

First we use these properties to construct the basic data for the doubling method.
Define $\widetilde{\GL}_{kc}$ by \eqref{covering:GLkc def}.
Let $\rho$ be a genuine representation of $\widetilde{\GL}_{kc}$. We say that $\rho$ is a $(k,c)$ representation if
$\Hom_{V(\sigma)}(\rho,\psi')=0$ for all $\sigma\succsim(k^c)$ and $\psi'\in\widehat{V}(\sigma)_{\mathrm{gen}}$, and
$\dim\Hom_{V_{(c^k)}}(\rho,\psi_{k})=1$. By \eqref{covering:GLkc def}, the action of $\SL_c^{\triangle}$ on
$J_{V_{(c^k)},\psi_k}(\rho)$ is well defined, then it is trivial.

If $H\ne\GL_{2kc}$, let $\rho$ be a genuine admissible finite length $(k,c)$ representation of $\widetilde{\GL}_{kc}$.
For $H=\GSpin_{2kc}$, by \eqref{covering:GSpin center} the irreducible representations of $\widetilde{C}_H^{\circ}$ are the lifts of quasi-characters of $F^*$ to genuine characters, therefore if $\eta$ is a quasi-character of $F^*$ which we regard also as a character of $\widetilde{C}_H^{\circ}$, the representation $\rho\otimes\eta$ is well defined.
For $H=\GL_{2kc}$, by \eqref{covering:GL Levi} we have
\begin{align*}
\{(\epsilon_1,\epsilon_2)\in\mu_m\times\mu_m:\epsilon_1\epsilon_2=1\}\backslash \widetilde{\GL}_{kc}\times \widetilde{\GL}_{kc} \cong\widetilde{M}_P.
\end{align*}
Hence $\rho=\rho_1\otimes\rho_2$ is defined for genuine representations $\rho_1$ and $\rho_2$, which we take to be admissible finite length and $(k,c)$. The space $V(s,\rho\otimes\eta)$ is now defined as in \S~\ref{Doubling setup}, with induction from $\widetilde{P}$ to $\widetilde{H}$.

If $H=\GL_{2kc}$, according to \eqref{covering:MP split over GL GL} there is a quasi-character $\chi_0$ of $F^*$, such that the action of $\{\diag(g^{\triangle},g^{\triangle}):g\in\GL_c\}$ on ${}^{\delta^{-1}}(J_{V_{(c^k)},\psi_k}(\rho_1)\otimes J_{V_{(c^k)},\psi_k}(\rho_2))$ is given by $g\mapsto\chi_0(\det g)$.

Let $\pi_1$ (resp., $\pi_2$) be an anti-genuine (resp., genuine) finite length admissible representation of $G^{(m)}$. If $H=\GSpin_{2kc}$,
we assume $\pi_1$ and $\pi_2$ admit central characters. Then by \eqref{covering:GSpin center} these characters restrict to genuine characters
of $\widetilde{C}_H^{\circ}$, denoted $\chi_{\pi_1}$ and $\chi_{\pi_2}$, which we can identify with quasi-characters of $F^*$. We assume $\chi_{\pi_1}^{-1}=\chi_{\pi_2}$ and put $\eta=\chi_{\pi_1}^{-1}$. Consider the space
\begin{align}\label{eq:covering Hom factor 1}
\Hom_{(G,G)^{(m)}}(J_{U,\psi_U^{-1}}(V(s,\rho\otimes\eta)),\pi_1\otimes\pi_2).
\end{align}
The representation $V(s,\rho\otimes\eta)$ is a priori a representation of $(G,G)^{(m)}$ by \eqref{covering:G G}.
Since $\pi_1$ is anti-genuine and $\pi_2$ is genuine, $\pi_1\otimes\pi_2$ factors through $(G,G)^{(m)}$, and it follows that \eqref{eq:covering Hom factor 1} is well defined.

Recall $D=U\rtimes (G,G)$ and denote $D^{(m)}=U\rtimes (G,G)^{(m)}$. Then \eqref{eq:covering Hom factor 1} is isomorphic to
\begin{align}\label{eq:covering Hom L 1}
\Hom_{D^{(m)}}(V(s,\rho\otimes\eta),\psi_U^{-1}\otimes\pi_1\otimes\pi_2).
\end{align}

By \eqref{covering:G G e_2}, $V(s,\rho\otimes\eta)$ is a genuine representation of the right copy of $G^{(m)}$, and so is $\pi_2$. Combining
\eqref{covering:G G e_2} with \eqref{covering:G G}, $\epsilon_1\in\mu_m$ is mapped to $\epsilon_1^{-1}$ under $\widetilde{\mathfrak{e}}_1$, whence $V(s,\rho\otimes\eta)$ is an anti-genuine representation of the left copy of $G^{(m)}$, as is $\pi_1$. Therefore the representation
\begin{align*}
V(s,\rho\otimes\eta)\otimes(\psi_U\otimes\pi_1^{\vee}\otimes\pi_2^{\vee})
\end{align*}
of $D^{(m)}$ factors through $D$. Hence \eqref{eq:covering Hom L 1} equals
\begin{align*}
&\Hom_{D}(V(s,\rho\otimes\eta)\otimes(\psi_U\otimes\pi_1^{\vee}\otimes\pi_2^{\vee}),1)
\\&=\Hom_{D}(\Ind_{\widetilde{P}\times D^{(m)}}^{\widetilde{H}\times D^{(m)}}\left((|\det|^{s-1/2}\rho\otimes\eta)\otimes(\psi_U\otimes\pi_1^{\vee}\otimes\pi_2^{\vee})\right),1).
\end{align*}
(Cf. \eqref{eq:Hom L}.)
Recall $P_{h}={}^{h^{-1}}P\cap D$. The covering $\widetilde{P}_{h}$ obtained by restriction from $\widetilde{H}$ coincides with the covering
restricted from $D^{(m)}$, by \eqref{covering:G G}. The space of distributions on $\widetilde{P}hD^{(m)}$ corresponds to
\begin{align*}
\Hom_{D}(\ind_{\widetilde{P}_{h}}^{D^{(m)}}\left({}^{h^{-1}}((|\det|^{s-1/2}\rho\otimes\eta)\delta_P^{1/2})\otimes(\psi_U\otimes \pi_1^{\vee}\otimes\pi_2^{\vee})\right),1),
\end{align*}
which by the Frobenius reciprocity is equal to
\begin{align}\label{covering H(h)}
\mathcal{H}(h)=\Hom_{P_{h}}({}^{h^{-1}}(|\det|^{s-1/2}\rho\otimes\eta)\otimes(\psi_U\otimes \pi_1^{\vee}\otimes\pi_2^{\vee}),\theta_h).
\end{align}
(Cf. \eqref{H(h)}.)
We can now use the theory of distributions on $l$-sheafs of \cite{BZ1}.
Recall that the right action of $D$ on $P\backslash H$ is constructive, i.e., the graph of the action is a finite union of locally closed sets (see \cite[6.1--6.6]{BZ1} for more details on these notions).
This follows from \cite[Theorem~A]{BZ1},
because $P\backslash H$ is an algebraic $F$-variety. Since
\begin{align*}
(\widetilde{P}\times D^{(m)})\backslash(\widetilde{H}\times D^{(m)})\cong \widetilde{P}\backslash\widetilde{H}\cong P\backslash H
\end{align*}
(as topological spaces), the right action of $D$ on
$(\widetilde{P}\times D^{(m)})\backslash(\widetilde{H}\times D^{(m)})$ is also constructive, justifying the application of \cite[Theorem~6.9]{BZ1} (note that the action of $D^{(m)}$ on the quotient factors through $D$).

The arguments of \S~\ref{outline} for showing the vanishing of $\mathcal{H}(h)$ also remain valid. We explain this in more detail.
First, if $Y<{}^hU\cap M_P$, then by \eqref{eq:covering conjugating canonical splitting},
${}^{h^{-1}}\varphi_Y(Y)=\varphi_{{}^{h^{-1}}Y}({}^{h^{-1}}Y)=\varphi_{U}({}^{h^{-1}}Y)$. Hence
\eqref{eq:T on Jacquet} holds and any morphism in $\mathcal{H}(h)$ factors through $J_{Y,{}^{h}\psi_U^{-1}}(\rho)$.

Condition \eqref{psi U nontrivial} is independent of the covering. Since
$\rho\otimes\eta$ is trivial on $\varphi_{N_H}(U_P)$, the condition ${}^hY<U_P$ and \eqref{eq:covering conjugating canonical splitting} imply ${}^{h^{-1}}(|\det|^{s-1/2}\rho\otimes\eta)$ is trivial on $\varphi_Y(Y)$, then we can deduce $\mathcal{H}(h)=0$.
The second method, where we show that any morphism in $\mathcal{H}(h)$ factors through $J_{V(\sigma),\psi'}(\rho)$ with $\sigma\succsim(k^c)$ and $\psi'\in\widehat{V}(\sigma)_{\mathrm{gen}}$, also implies $\mathcal{H}(h)=0$, as in the linear case.

The only change concerns the third condition, where it is not necessarily true that the preimage of
${}^h(1,C_{\GL_l})$ in $\widetilde{H}$ acts by a character, because this preimage might not be abelian. However,
we can instead use the preimage $\widetilde{C}_{\GL_l}^d$ of $C_{\GL_l}^d$ (for a large integer $d$), which is abelian and belongs to the center of $\widetilde{\GL}_l$, by assumption \eqref{covering:center of GLl}. Then $\widetilde{C}_{\GL_l}^d$ acts by a character on each irreducible constituent of $J_{U_R}(\pi_2^{\vee})$, and the preimage of ${}^h(1,C_{\GL_l}^d)$ in $\widetilde{H}$ acts by a character on each of finitely many constituents. The only change to \eqref{eq:relation for T with s} is that now we replace $a\in F^*$ with $a^d$, but this still implies the vanishing outside a discrete subset of $s$.

Define
\begin{align*}
d(s,\rho,\eta,\pi_1,\pi_2)=\dim\Hom_{(G,G)^{(m)}}(J_{U,\psi_U^{-1}}(V(s,\rho\otimes\eta)),\pi_1\otimes\pi_2).
\end{align*}
We are ready to prove Theorem~\ref{theorem:uniqueness} for covering groups.
\begin{theorem}\label{theorem:covering uniqueness}
Let $\pi_1$, $\pi_2$ and $\rho$ be as above.
\begin{enumerate}[leftmargin=*]
\item\label{part1}Outside a discrete subset of $s$, $d(s,\rho,\eta,\pi_1,\pi_2)\leq\dim\Hom_{G^{(m)}}(\chi_0\pi_1^{\vee},\pi_2^{\iota})$.
\item\label{part2}If $\pi_1$ and $\pi_2$ are irreducible, outside a discrete subset of $s$,
$d(s,\rho,\eta,\pi_1,\pi_2)=0$ unless $\pi_1=\chi_0(\pi_2^{\iota})^{\vee}$ in which case $d(s,\rho,\eta,\pi_1,\pi_2)\leq1$.
\end{enumerate}
Furthermore, assume $\pi_2$ is supercuspidal and $\rho$ is not necessarily of finite length. Then
the assertions of \eqref{part1} and \eqref{part2} hold for all $s$, granted either
$H\ne\GL_{2kc}$ and $c>2$, or $H=\Sp_{4k}$.
\end{theorem}
\begin{remark}
Evidently, there is no essential difference between the statements in the linear setting and the covering (for $m=1$, $G^{(m)}=G$), except the
supercuspidal cases, where we excluded the conditions depending on $\rho$. This is because we are not discussing $\rho_c(\tau)$ for covering groups here; the definition of this representation is thus far clear only when $\tau$ is a genuine unramified principal series (see \cite{me12}). Once the details are worked out, the arguments here are expected to extend to these cases as well.
\end{remark}
\begin{proof}
Since $\widetilde{P}\backslash \widetilde{H}/D^{(m)}=P\backslash H/D$, we can use the same description for the representatives $h$.
The arguments of Propositions~\ref{proposition:structure of w u}--\ref{proposition:wu_0 nontrivial implies h nontrivial orbit} and
Propositions~\ref{proposition:GL structure of w u}--\ref{proposition:GL wu_0 nontrivial implies h nontrivial orbit} extend to the covering.

For Propositions~\ref{proposition:d_1 < n-l}, \ref{proposition:d1 = n-l l < n}, \ref{proposition:GL d_1 < n-l},
\ref{proposition:GL d1 = n-l l < n}--\ref{proposition:GL d1 = n-l j < l = c} we used two types of arguments. First, we showed that the Jacquet module $J_{V_{\beta},\psi_{V_{\beta}}^{-1}}(\rho)$ vanishes, because the order of nilpotency of $\varphi$ is at least $k+1$. The arguments involving the action of a normalizer on the set of characters of an abelian unipotent subgroup carry over to the covering. This is because in general, if a subgroup $A<H$ normalizes a unipotent subgroup $Y<H$, thereby acts on its set of characters, then $\widetilde{A}$ also acts on the set of characters of $Y$ with the same orbits (because $\widetilde{H}$ is split canonically over $Y$). The arguments of \cite[5.9--5.12]{BZ1} still apply. Then we used \cite[Theorems~A, E]{GGS}, in which strictly speaking covering groups were not discussed. However, one can still use conjugations and \cite[5.9--5.12]{BZ1} to show that $J_{V_{\beta},\psi_{V_{\beta}}^{-1}}(\rho)$ factors through a Jacquet module with respect to a unipotent orbit which is greater than or non-comparable with $(k^c)$. See Example~\ref{example:covering GGS} below. Second, we used \eqref{eq:relation for T with s}, which is still applicable with the minor change explained above.

It remains to consider $\mathcal{H}(h)$ where $h\sim\delta$. Consider $H\ne\GL_{2kc}$. Then
\begin{align*}
\mathcal{H}(\delta)&=\Hom_{G^{\iota}\times C_H^{\circ}}({}^{\delta^{-1}}J_{V_{(c^k)},\psi_k}(\rho)\otimes\eta\otimes\pi_1^{\vee}\otimes\pi_2^{\vee},1).
\end{align*}
For $H=\GSpin_{2kc}$ the assumption $\eta=\chi_{\pi_1}^{-1}=\chi_{\pi_2}$ implies that this space equals
\begin{align*}
\Hom_{G^{\iota}}({}^{\delta^{-1}}J_{V_{(c^k)},\psi_k}(\rho)\otimes\pi_1^{\vee}\otimes\pi_2^{\vee},1).
\end{align*}
Then since $J_{V_{(c^k)},\psi_k}(\rho)$ is a trivial representation of $\SL_c^{\triangle}$ (see \eqref{covering:GLkc def}) and by virtue of
\eqref{covering:inv iota},
\begin{align*}
\mathcal{H}(\delta)=\Hom_{G^{\iota}}(\pi_1^{\vee}\otimes\pi_2^{\vee},1)=\Hom_{G}(\pi_1^{\vee}\otimes(\pi_2^{\vee}){}^{\iota},1)=
\Hom_{G^{(m)}}(\pi_1^{\vee},\pi_2^{\iota}).
\end{align*}

For $H=\GL_{2kc}$ we first have
\begin{align*}
\mathcal{H}(\delta)&=\Hom_{G^{\iota}}({}^{\delta^{-1}}J_{V_{(c^k)}\times V_{(c^k)},\psi_k\otimes\psi_k}(\rho)\otimes\pi_1^{\vee}\otimes\pi_2^{\vee},1).
\end{align*}
The action of $G^{\iota}$ on ${}^{\delta^{-1}}(J_{V_{(c^k)},\psi_k}(\rho_1)\otimes J_{V_{(c^k)},\psi_k}(\rho_2))$ is given by $g\mapsto\chi_0(\det g)$, and we obtain $\Hom_{G^{(m)}}(\chi_0\pi_1^{\vee},\pi_2)$.
The remaining parts of the proof now follow as in the linear case.
\end{proof}

\begin{example}\label{example:covering GGS}
Consider a Jacquet module of a $(2,2)$ representation $\rho$ with respect to the unipotent subgroup $Y<\GL_4$ and character $\psi$ given by
\begin{align*}
Y=\left\{y=\left(\begin{smallmatrix}
  1 & x_1 & x_2 & x_3 \\
   & 1 & x_4 & x_5 \\
   &  & 1 &  \\
   &  &  & 1
\end{smallmatrix}\right)\right\},\qquad \psi(y)=\psi(x_1+x_5).
\end{align*}
It suffices to show the vanishing with respect to the subgroup of $Y$ with $x_4=0$. Using conjugation by $\diag(1,\left(\begin{smallmatrix}&1\\1\end{smallmatrix}\right),1)$, we obtain
\begin{align*}
Y'=\left\{y=\left(\begin{smallmatrix}
  1 & x_2 & x_1 & x_3 \\
   & 1 &  &  \\
   &  & 1 & x_5 \\
   &  &  & 1
\end{smallmatrix}\right)\right\}.
\end{align*}
The Jacquet module $J_{Y',\psi}(\rho)$ ($\psi$ does not change) is a representation of
\begin{align*}
X=\left\{\left(\begin{smallmatrix}
  1 &  &  &  \\
   & 1 &  & x_6 \\
   &  & 1 &  \\
   &  &  & 1
\end{smallmatrix}\right)\right\}.
\end{align*}
The preimage of the subgroup $\{\diag(1,t,I_2):t\in F^*\}$, which also acts on $J_{Y',\psi}(\rho)$, acts on the set of characters of $X$ with $2$ orbits (for an action of $T_{\GL_2}$ use $\diag(t',t,t',t')$). Both orbits can be conjugated into $J_{Y'\rtimes X,\psi}(\rho)$ with still the same $\psi$
using $\diag(1,\left(\begin{smallmatrix}1&\\z&1\end{smallmatrix}\right),1)$. It remains to prove
$J_{Y'\rtimes X,\psi}(\rho)=0$. Passing to the subgroup with $x_2=0$ and conjugating by
$\diag(\left(\begin{smallmatrix}&1\\1\end{smallmatrix}\right),I_2)$, it is enough to prove
$J_{Y'',\psi}(\rho)=0$ with
\begin{align*}
Y''=\left\{y=\left(\begin{smallmatrix}
  1 &  &  & x_6 \\
   & 1 & x_1 & x_3 \\
   &  & 1 & x_5 \\
   &  &  & 1
\end{smallmatrix}\right)\right\},\qquad \psi(y)=\psi(x_1+x_5).
\end{align*}
As with $x_6$, one can fill in the missing coordinate above $x_1$ using
\begin{align*}
X'=\left\{\left(\begin{smallmatrix}
  1 &  & x_0 &  \\
   & 1 &  &  \\
   &  & 1 &  \\
   &  &  & 1
\end{smallmatrix}\right)\right\},\qquad \{\diag(t,I_3):t\in F^*\}, \qquad
\diag(\left(\begin{smallmatrix}1&\\z&1\end{smallmatrix}\right),I_2).
\end{align*}
We have shown that $J_{Y,\psi}(\rho)$ factors through $J_{V_{(2,1,1)},\psi}(\rho)$. This module is filtered by the third and fourth derivatives of $\rho$ (in the sense of \cite{BZ1}), both of which vanish because $\rho$ is $(2,2)$.
\end{example}

We briefly describe the applicability of Theorem~\ref{theorem:covering uniqueness} to the construction of \cite{me12}.
Henceforth assume $-1$ is an $m$-th root of unity in $F^*$ (this is a technical assumption, used in \cite{me12} and several other works, to simplify some of the computations).
For any integer $l$, let $\Sp_{2l}^{(m)}$ denote the covering of \cite{Mats} defined using the $m$-th order Hilbert symbol $(,)_m$. For $\GL_l$, let $\GL_l^{(m)}$ denote the covering obtained by restriction from $\Sp_{2l}^{(m)}$, when we identify $\GL_l$ with the standard Siegel Levi subgroup of $\Sp_{2l}$ by $g\mapsto \diag(g,g^*)$.

Let $r=m$ if $m$ is odd, otherwise $r=m/2$. Let $k_0$ be a positive integer, and put $k=rk_0$.
The above list of properties were verified in \cite{me12} when $G=\Sp_c$ or $\GL_c$, for the covering $G^{(m)}$, with $\widetilde{H}=H^{(m)}$.

\begin{remark}
The group $\GL_l^{(m)}$ was denoted $\GL_l^{(m,r)}$ in \cite{me12}, to underline the difference between this covering and the coverings of \cite{KP}, and $k$ of \cite{me12} is $k_0$ here.
\end{remark}

Assume we have a $(k,c)=(rk_0,c)$ representation $\rho$ (admissible of finite length). It is at present not clear how to construct such representations in general (e.g., from representations of a covering of $\GL_k$), but in the unramified setting this was obtained in \cite{me12} (following \cite{Suzuki1998}). Note that here the ``unramified setting" includes the assumptions $|m|=1$ and $q>4$. Briefly, given a genuine unramified principal series representation $\tau$ of $\GL_{k_0}^{(m)}$, one can choose an unramified character of $T_{\GL_{k_0}}$ associated with the inducing data of $\tau$ (the correspondence is not unique). Using this character and an exceptional representation of $\GL_{rc}^{(m)}$ (exceptional in the sense of \cite{KP}, see \cite{Gao5}), the prescribed $\rho$ was constructed in \cite[\S~2.2]{me12}.

For $H=\GL_{2kc}$, $\chi_0$ was taken to be trivial (see \cite[(3.34)]{me12}).

Also let $\pi$ be a genuine irreducible representation of $G^{(m)}$.
The integral $Z(s,\omega,f)$, with a holomorphic or rational section $f$, was defined in \cite{me12} (using notation similar to \S~\ref{Doubling setup}). Formally, it belongs to
\eqref{eq:covering Hom factor 1} with $\pi_1=\pi^{\vee}$ and $\pi_2=\pi^{\iota}$. This was proved in
\cite[Propositions~68, 75]{me12} (in \textit{loc. cit.} (3.21) and (3.36), $G^{(m)}\times G^{(m)}$ should be replaced with $(G,G)^{(m)}$).
\begin{corollary}\label{corollary:covering integral props and gamma}
$Z(s,\omega,f)$ admits meromorphic continuation to a function in $q^{-s}$.
\end{corollary}
\begin{proof}
This follows from Theorem~\ref{theorem:covering uniqueness} and Bernstein's continuation principle (\cite{Banks}), see
\cite[Remark~72]{me12} and \cite[\S~3.3]{me12} (cf. Corollary~\ref{coro:meromorphic continuation for doubling 1} here).
\end{proof}
\begin{corollary}
One can define a local $\gamma$-factor $\gamma(s,\pi\times\tau,\psi)$ by virtue of \eqref{eq:gamma factor}.
\end{corollary}
Note that the additional normalization of the intertwining operator appearing in \eqref{eq:gamma factor} can be applied to the covering case as well; but we are not proving the multiplicativity properties of the $\gamma$-factor here, and at any rate, we are still limited to the unramified setting. The point here is that the proportionality factor exists.

\subsection{Global unfolding}\label{Global unfolding}
The global doubling construction in the linear case for arbitrary $k$ was first described in \cite{CFGK2} mainly for the symplectic group, with some details also for the special even orthogonal group,
then briefly explained in \cite{CFK} for the other cases appearing here. The covering version for the symplectic group was described in
\cite{me12}.

Let $F_0$ be a number field with a ring of adeles $\A$.
Let $\tau$ be an irreducible cuspidal automorphic representation of $\GL_k(\A)$, and $\mathcal{E}_{\tau}$ denote the generalized Speh representation of $\GL_{kc}(\A)$ corresponding to $c$ copies of $\tau$, constructed by Jacquet \cite{Jac4}. According to \cite{G2,JL2013,CFGK2,CFK}, the representation $\mathcal{E}_{\tau}$ is a global $(k,c)$ representation: it does not support any Fourier coefficient along an orbit greater than or non-comparable with $(k^c)$, it supports a Fourier coefficient along $(k^c)$, and all of its local components are $(k,c)$. See \cite[\S~2.2]{CFGK2} and the references therein for more details on the global notions. Moreover, if $\tau=\otimes_{\nu}'\tau_{\nu}$ as a restricted tensor product,
$(\mathcal{E}_{\tau})_{\nu}=\rho_c(\tau_{\nu})$ for any place $\nu$ of $F_0$.

One can readily globalize our arguments used for the proof of Theorem~\ref{theorem:uniqueness} to obtain a proof of the
unfolding of the global doubling integral, for all of the groups under consideration here (and in \cite{CFK}). At the same time,
since local vanishing of Jacquet modules implies global vanishing of the corresponding Fourier coefficients (even one local $(k,c)$ component suffices for this), the proof of Theorem~\ref{theorem:uniqueness} also provides a proof of the global unfolding. In addition we obtain the following corollary, which for brevity, is stated here in the symplectic or special orthogonal cases, but the other cases are evident as well.

We use the notation and definitions of \S~\ref{Doubling setup}, in the global context. Let $K_H$ be a
standard maximal compact subgroup, in a ``good position" with respect to $T_H$.
Let $f$ be a $K_H$-finite section of $\Ind_{P(\A)}^{H(\A)}(|\det|^{s-1/2}\mathcal{E}_{\tau})$, whose restriction to
$K_H$ is independent of $s$. We regard
$f$ as a complex-valued function.

Recall the definition \eqref{psi_U on V beta d1 n-l} of a character $\psi_{V_{\beta}}$ when $\beta=(c^k)$, defined with respect to $0\leq l\leq n$, which we re-denote here by $\psi_{(c^k),l}$ (in the context of \eqref{psi_U on V beta d1 n-l}, $l$ was fixed). In particular
$\psi_{(c^k),n}$ is in the orbit of $\psi_k^{-1}$ ($\psi_{(c^k),n}=\psi_k^{-1}$ when $c$ is even).
For $k=1$, $\psi_{(c^k),l}$ is trivial. Then we have the Fourier coefficients of
$f$ along $(V_{(c^k)},\psi_{(c^k),l})$, defined by
\begin{align*}
&f^{V_{(c^k)},\psi_{(c^k),l}}(s,x)=\int\limits_{V_{(c^k)}(F_0)\backslash V_{(c^k)}(\A)}
f(s,vx)\,\psi_{(c^k),l}(v)\,dv.
\end{align*}
In particular $f^{V_{(c^k)},\psi_{(c^k),n}}$ is the coefficient $f_{W_{\psi}(\mathcal{E}_{\tau})}$ appearing in \cite[Theorem~1]{CFGK2}, i.e., the composition of $f$ with the global $(k,c)$ functional on the space of $\mathcal{E}_{\tau}$ given by a Fourier coefficient (if $c$ is odd, this is true up to a conjugation which identifies $\psi_{(c^k),n}$ with $\psi_k^{-1}$).

The Eisenstein series corresponding to $f$ is defined by
\begin{align}\label{eq:Eisenstein series main}
E(x;s,f)=\sum\limits_{\gamma\in P(F_0)\backslash H(F_0)}f(s,\gamma x),\qquad x\in H(\A).
\end{align}
The series is absolutely convergent for $\Real(s)\gg0$ and admits meromorphic continuation to $\C$.
Consider the Fourier coefficient of $E(x;s,f)$ along $(U,\psi_U)$, given by
\begin{align}\label{eq:U psi U coefficient of the series}
E^{U,\psi_U}(x;s,f)=\int\limits_{U(F_0)\backslash U(\A)}E(ux;s,f)\psi_U(u)\,du.
\end{align}
The definitions imply that $E^{U,\psi}(\cdot;s,f)$ is an automorphic form on $G(\A)\times G(\A)$.

For $0\leq l\leq n$, let $w_l$ be the representative $w$ chosen after the proof of
Proposition~\ref{proposition:2nd reduction of w} (used for the computation of \eqref{eq:beta}), but with $d_1=\ldots=d_{k-1}=n-l$.
Using Example~\ref{eq:example k=2,3} we see that
\begin{align*}
w_l=\left(\begin{smallmatrix}0&0&0&0&I_l&0\\
0&0&I_{c-l}&0&0&0\\0&0&0&0&0&I_{(k-1)c}\\\epsilon_0I_{(k-1)c}&0&0&0&0&0\\0&0&0&I_{c-l}&0&0\\0&\epsilon_0I_l&0&0&0&0\end{smallmatrix}\right)
\jmath_{(k-1)c+l}.
\end{align*}
A quick computation implies ${}^{w_l^{-1}}V_{(c^k)}={}^{w''}V_{(c^k)}$, where
\begin{align*}
w''=\jmath_{(k-1)c+l}\diag(I_{(k-1)c+l},\left(\begin{smallmatrix}& I_{c-l} \\\epsilon_0I_{c-l} & \end{smallmatrix}\right),I_{(k-1)c+l}).
\end{align*}
Then $U={}^{w_l^{-1}}V_{(c^k)}\ltimes U_{n-l}$ for the subgroup
\begin{align*}
U_{n-l}={}^{w''}(U\cap U_P)=
{}^{{}^{\jmath_{(k-1)c+l}}}\left(\begin{smallmatrix}
  I_{(k-1)c} & 0& u_1 & 0 & u_2 & u_3\\
   & I_l &&&& u_2' \\
   & & I_{c-l} &&& 0\\
   & & & I_{c-l}&&u_1' \\
    &&&& I_l &0\\
  &&&&&I_{(k-1)c}
\end{smallmatrix}\right)
\end{align*}
(if we replace $w_l$ by $\delta_0$, $U_0=U\cap {}^{\jmath_{kc}}U_P$),
and $P_{w_l}\cap U={}^{w_l^{-1}}V_{(c^k)}$. Denote $P_{w_l}'=P_{w_l}\cap (G,G)$.
\begin{corollary}\label{corollary:unfolding of the series}
In $\Real(s)\gg0$,
\begin{align*}
E^{U,\psi_U}(x;s,f)=
\sum_{l=0}^n
\sum\limits_{y\in P_{w_l}'(F_0)\backslash (G(F_0),G(F_0))}\,
\int\limits_{U_{n-l}(\A)}f^{V_{(c^k)},\psi_{(c^k),l}}(s,w_l({}^{\jmath_l}u_l)uyx)\psi_U(u)\,du.
\end{align*}
\end{corollary}
\begin{proof}
We can assume $x=I_{2kc}$. Write the sum \eqref{eq:Eisenstein series main} over $P(F_0)\backslash H(F_0)/D(F_0)\times P_h(F_0)\backslash D(F_0)$. In a right half plane we can exchange summation and integration. Thus
\begin{align*}
E^{U,\psi_U}(I_{2kc};s,f)=\sum\limits_{h\in P(F_0)\backslash H(F_0)/D(F_0)}\,\int\limits_{U(F_0)\backslash U(\A)}\,
\sum\limits_{y\in P_h(F_0)\backslash D(F_0)}f(s,hyu)\psi_U(u)\,du.
\end{align*}
Next because $u\in M_Q$ and $P \cap {}^wQ=(P \cap {}^wM_Q)\ltimes (P \cap {}^wU)$, we have
\begin{align*}
{}^{h^{-1}}P \cap Q={}^{h^{-1}}(P \cap {}^wQ)=
({}^{h^{-1}}P \cap M_Q)\ltimes({}^{h^{-1}}P \cap U).
\end{align*}
Since $P_h<Q$ and $(G,G)<M_Q$, we deduce
\begin{align*}
P_h=(P_h\cap (G,G))\ltimes (P_h\cap U)=P'_h\ltimes P''_h,
\end{align*}
whence we can collapse the $du$-integral, exchange $yu\mapsto uy$ and take the integral inside:
\begin{align*}
E^{U,\psi_U}(I_{2kc};s,f)=\sum\limits_{h\in P(F_0)\backslash H(F_0)/D(F_0)}\,
\sum\limits_{y\in P'_h(F_0)\backslash(G(F_0),G(F_0))}\,
\int\limits_{P''_h(F_0)\backslash U(\A)}
f(s,huy)\psi_U(u)\,du.
\end{align*}

Now the proof of Theorem~\ref{theorem:uniqueness}, more specifically
Propositions~\ref{proposition:structure of w u}, \ref{proposition:1st reduction of w}--\ref{proposition:d_1 < n-l}, imply the inner $du$-integral vanishes unless
$h\sim w_l({}^{\jmath_l}u_l)$, $0\leq l\leq n$. The corresponding summand is
\begin{align*}
\sum\limits_{y\in P_{w_l}'(F_0)\backslash (G(F_0),G(F_0))}\,
\int\limits_{U_{n-l}(\A)}f^{V_{(c^k)},\psi_{(c^k),l}}(s,w_l({}^{\jmath_l}u_l)uy)\psi_U(u)\,du.
\end{align*}
This completes the proof.
\end{proof}
Now let $\pi_1$ and $\pi_2$ be irreducible cuspidal automorphic representations of $G(\A)$, and $\varphi_1$ and $\varphi_2$ be two cusp forms in the corresponding spaces. Assume $G$ admits nontrivial unipotent subgroups (i.e., exclude some low rank cases). Denote ${}^{\iota}\varphi_2(g)=\varphi_2({}^{\iota}g)$ and
\begin{align*}
\langle\varphi_1,\varphi_2\rangle=\int\limits_{G(F_0)\backslash G(\A)}\varphi_1(g)\overline{\varphi_2(g)}\,dg.
\end{align*}
Then by Corollary~\ref{corollary:unfolding of the series} and Lemma~\ref{lemma:Jacquet module is a trivial rep of U_R}, \eqref{eq:U psi U coefficient of the series} pairs with $\varphi_1$ and $\varphi_2$, in the sense that for $\Real(s)\gg0$,
\begin{align}\label{eq:main identity}
&\int\limits_{G(F_0)\backslash G(\A)\times G(F_0)\backslash G(\A)}\varphi_1(g_1)\overline{{}^{\iota}\varphi_2(g_2)}
E^{U,\psi}((g_1,g_2);s,f)\,dg_1\,dg_2
\\&=\int\limits_{G(\A)}\int\limits_{U_0(\A)}<\varphi_1,\pi_2(g)\varphi_2>
f^{V_{(c^k)},\psi_{(c^k),n}}(s,\delta u_0(1,{}^{\iota}g))\psi_U(u_0)\,du_0\,dg.\notag
\end{align}
Indeed consider one of the summands appearing in
Corollary~\ref{corollary:unfolding of the series} with $l<n$. Set $U_R^l={}^{w_l({}^{\jmath_l}u_l)}(1,{}^{\jmath_{(c+1)l}}U_R)$, with the notation of the proof of Proposition~\ref{proposition:d1 = n-l l < n}. The Fourier coefficient of $f$ along $(V_{(c^k)},\psi_{(c^k),l})$
is left invariant under $U_R^l(\A)$. To see this consider a second Fourier expansion of this coefficient, along $U_R^l$. All terms but the constant one vanish, because by Lemma~\ref{lemma:Jacquet module is a trivial rep of U_R}, at the non-archimedean places $v$ of $F_0$ the action of $U_R^l((F_0)_v)$ on $J_{V_{(c^k)},\psi_{(c^k),l}^{-1}}((\mathcal{E}_{\tau})_{v})$ is trivial. Since $\pi_2$ is cuspidal, the summand itself vanishes.

Of course \eqref{eq:main identity} is plainly the main global identity of \cite{CFGK2}: the left hand side is the global doubling integral
$Z(s,\varphi_1,\varphi_2,f)$, and it is nontrivial when $\pi_1=\pi_2$ according to the computations of the local integrals appearing in the Euler product on the right hand side.

One can include low rank cases, e.g., $c=2$ and $G=\SO_2$, by globalizing the argument from Lemma~\ref{lemma:Jacquet module is a finite length} (the constant term of the Eisenstein series defining $\mathcal{E}_{\tau}$ along $V_{(1,kc-1)}$ vanishes when $k>1$).
The low rank arguments of Propositions~\ref{proposition:GL d1 = n-l l < n} and \ref{proposition:GL d1 = n-l j < l = c} can be globalized
using the constant term computation of $\mathcal{E}_{\tau}$ given by \cite[Lemma~4.1]{JngLiu}.

The results of this section also hold in the covering case of \cite{me12}, but to formulate them properly one must check the validity of certain properties of the global covering, which are the analogs of the list from \S~\ref{Covering groups} (this was carried out in \cite{me12}).

\appendix
\section{Vanishing of vector-valued distributions on smooth manifolds, \\ by Avraham Aizenbud and Dmitry Gourevitch}\label{Appendix}

Let a Lie group $C$ act on a smooth manifold $X$. Let $Z\subset X$ be a locally closed  $C$-invariant subset. Let $\cF$ be a possibly infinite-dimensional $C$-equivariant bundle on $X$ (see \S \ref{subsec:Bundles} below for this notion). Assume that for any $z\in Z$, and  $k \in \Z_{\geq 0}$ we have
\begin{align}\label{=loc0}
((\cF|_z \otimes \Sym^k(CN_{z,Cz}^X) \otimes ((\Delta_C)|_{C_z}/\Delta_{C_z}))^*)^{C_z} =0 \, .
\end{align}
In this appendix we show that
\begin{align}\label{=dist0}
\cD'_Z(X,\cF)^C=0\, ,
\end{align}
under certain additional finiteness conditions, generalizing Theorem \ref{thm:AG}.

In \S \ref{sec:prel} we will explain the notation of (\ref{=loc0},\ref{=dist0}), and give the definitions of the main notions used in this appendix, as well as some basic properties of these objects. In particular, we use the theory of infinite-dimensional bundles developed in \cite{KM1997}, define generalized sections of such bundles, and construct pullbacks and pushforwards for such sections.

In the case when $C$ has finitely many orbits on $Z$, and the bundle is trivial, the implication \eqref{=loc0}$\Rightarrow$\eqref{=dist0} is classical, see \cite{Bru}.
In \cite{ChenSun} a cohomological version of the implication is proven, in a semi-algebraic setting and assuming finitely many orbits. In \cite{AG_ST} a similar implication is proven in the semi-algebraic setting, with $\cF$ finite-dimensional. In \cite{KV} several special cases of the implication are proven, in particular the case in which $\cF$ has the form $\cE\otimes V$, where $\cE$ is finite-dimensional and $V$ is a fixed representation in a \Fre space, and the action of $C$ on $X$, $\cE$ and $\cF$ can be extended to an action of a group $G$ that includes $C$ as a normal subgroup, preserves $Z$ and acts on it with finitely many orbits, each orbit locally closed.

We prove that \eqref{=loc0}$\Rightarrow$\eqref{=dist0}  in a similar generality, but with two essential differences. First, we allow $\cE$ to be a general  \Fre bundle (which makes $V$ obsolete). Second, we allow to twist the action by an additional $C$-equivariant line bundle $L$ on which the $C$-action does not necessarily extend to $G$. However, we  put an additional finiteness condition on the pullbacks of $L$ under the action of $G$ on $X$. The twist by a line bundle $L$ is crucial for our application. We use both the result and the method of \cite{KV} in our proof.

We do not know whether the vanishing \eqref{=loc0} implies the vanishing \eqref{=dist0} in general. This is probably a very difficult analytic question.

\textbf{Structure of our proof.} Let us first describe the finiteness condition we require.
For any $z\in Z$ denote by $C_z$ and $G_z$ the stabilizers of $z$ in $C$ and $G$ respectively. Let $L_z$ denote the character of $C_z$ defined by $L$. For any $g\in G$ denote by $L_z^g$ the character of $C_z$ given by  $L_z^g(c):=L_{gz}(gcg^{-1})$. We require that for any $z\in Z$, the set $\{L_z^g \, | \, g\in G\}$ is a finite union of locally-closed orbits of $G_z$.

We first solve the case in which $G$ acts trivially on $X$, and $\cE$ is constant as a $G$-equivariant bundle. For this  case we remove the assumption that $Z$ lies in a finite union of $G$-orbits.
Localizing the problem, we assume $X$ to be $\R^n$, and let it act on itself by translations, and on $\cE$ trivially. We translate the problem to a problem on $X\times C$-equivariant distributions on $X\times \hat{C}_{\C}$, where $\hat{C}_{\C}$ denotes the set of all (continuous) characters of $C$. In this new problem, the $X\times C$-equivariant structure on the line bundle extends to an action of $X\times G$, and thus the space of equivariant distributions vanishes by \cite{KV}.

The next case we resolve is when $G$ acts transitively on $X$. Then we have $X=G/H$. We construct a bundle $\cF_1$ on $X_1:=C\backslash G$ such that the space of $C$-invariant $\cF$-distributions on $X$ is isomorphic to the space of $H$-invariant $\cF_1$-distributions on $X_1$. We show that already the space of $H\cap C$-invariant distributions vanishes, using the previous case. The argument here is somewhat similar to the argument in \cite{KV}. However, it is complicated by the presence of the line bundle $L$, and by $\cF$ not being constant.

The next case we treat is the case of $Z$ being a single $G$-orbit. As in \cite{KV}, it reduces to the previous one using the transverse symbol of distributions.

Finally, we prove the general case by induction on the number of $G$-orbits on $Z$.

\textbf{Acknowledgements.} We thank Alexander Shamov for fruitful discussions. We were partially supported by ISF grant number 249/17.

\subsection{Preliminaries}\label{sec:prel}

\subsubsection{Topological vector spaces}
All topological vector spaces considered in this appendix will be
complete, Hausdorff, and locally convex.
For such a   space $V$, $V^*$ will denote the strong dual, and for two spaces $V$ and $W$,  $V\hot W$ will denote the completed projective tensor product (this is the same tensor product denoted $\otimes$ in the body of the paper, for convenience). The projective topology on $V\hot W$ is generated by the family of seminorms which are the largest cross-norms corresponding to pairs of generating semi-norms on $V$ and $W$, see \cite[\S 43]{Tre}. In particular, if $V$ and $W$ are \Fre  spaces, then so is $V\hot W$. If $V$ (or $W$) is nuclear then the projective tensor product is naturally isomorphic to the injective one, see \cite[Theorem 50.1]{Tre}. This is the case in all our theorems. The tensor product of nuclear spaces is nuclear. A \Fre space is nuclear if and only if its dual space is.  For more information on nuclear spaces we refer the reader to \cite[\S 50]{Tre} or \cite[Appendix A]{CHM2000}.

\subsubsection{General topology}

We will use the following elementary lemma.
\begin{lemma}
Let $X$ be a topological space, and let $\{X_i\}_{i=1}^k$ be disjoint locally closed subsets such that $X=\bigcup_{i=1}^k X_i$. Then there exists $i$ such that the interior of $X_i$ is non-empty.
\end{lemma}
\begin{corollary}\label{cor:OpOrb}
Let a topological group $G$ act continuously on a topological space $X$ with finitely many locally closed orbits. Then one of the orbits is open.
\end{corollary}

\subsubsection{Infinite-dimensional smooth bundles over smooth manifolds}\label{subsec:Bundles}

We will use the results and terminology from \cite{KM1997}, which considers infinite-dimensional smooth manifolds and bundles over them. In our case the base manifolds will be finite-dimensional, but we will consider infinite-dimensional bundles. All the vector spaces we consider are complete, thus  sequentially complete, and thus $c^{\infty}$-complete, see \cite[Lemma 2.2 and Theorem 2.14]{KM1997}. Therefore all the results of \cite{KM1997} are applicable to them.
We use the notion of infinite-dimensional vector bundles and their spaces of smooth sections and spaces of compactly supported sections, see \cite[\S \S 29, 30]{KM1997}.

Let $X$ be a smooth manifold and $\cE$ be a vector bundle over $X$, possibly infinite-dimensional. We define the space of $\cE$-distributions $\cD'(X,\cE)$ to be the  continuous dual space $C_c^{\infty}(X,\cE)^*$ equipped with the strong topology. For a closed subset $Z\subset X$, we denote by $\cD'_Z(X,\cE)\subset \cD'(X,\cE)$ the subspace of distributions supported on $Z$.
For a locally closed subset $Z\subset X$, we denote
\begin{align}
\cD'_Z(X,\cE):=\cD'_Z(U,\cE), \text{ where }U:=X\setminus (\overline{Z}\setminus Z).
\end{align}

By \cite[II.\S 3.3]{Gro}, if $\cE$ is a trivial bundle with \Fre fiber $V$, then $C_c^{\infty}(X,\cE)\cong C_c^{\infty}(X)\hot V$.

For any smooth map $\nu:Y\to X$ and a bundle $\cE$ over $X$, the pullback bundle $\nu^*\cE$ over $Y$ is defined in \cite[29.6]{KM1997}. For a smooth section $f$ of $\cE$ we denote the corresponding section of $\nu^*\cE$ by $\nu^*f$.

We define
 $\nu^!\cE:=\nu^*\cE\otimes D_Y\otimes D^{-1}_X$.
In the case $\nu$ is a submersion and $\cE$ has \Fre fibers, we can also define $\nu_*:C_c^{\infty}(Y,\nu^!\cE)\to C_c^{\infty}(X,\cE)$ in the following way. For a trivial bundle $\cE$ with \Fre fiber $V$, we use the identification $$C_c^{\infty}(X,\cE)\cong C_c^{\infty}(X,V)\cong C_c^{\infty}(X)\hot V,$$ and the classical pushforward $\nu_*: C_c^{\infty}(Y,D_Y\otimes \nu^*D^{-1}_X)\to C_c^{\infty}(X)$. For a locally trivial bundle we use the partition of unity to trivialize $\cE$.

We denote the map dual to $\nu_*:C_c^{\infty}(Y,\nu^!\cE)\to C_c^{\infty}(X,\cE)$ by $$\nu^*:\cD'(X,\cE)\to \cD'(Y,\nu^!\cE).$$

If $Z$ is a smooth submanifold regularly embedded in $X$, and $V$ is a \Fre space, then for any $\xi\in \cD_Z(X,V)$ and $z\in X$,
\cite[\S 2]{KV} defines a \emph{transversal degree} $d\in \Z_{\geq 0}$ and a transverse symbol $\sigma_d(\xi)\in V^*\otimes \Sym^d(N_{z,Z}^X)$, where $\Sym^d$ denotes symmetric power,  $N_{z,Z}^X$ denotes the normal space to $Z$ in $X$ at $z$, and $CN_{z,Z}^X:=(N_{z,Z}^X)^*$ is the conormal bundle.
Denote by ${\cD'_Z}^{\leq d}(X,V)$ the space of distributions that have transversal degree at most $d$ for any $z\in Z$. By \cite[Theorem 2.1]{KV}, $\sigma_d$ defines a natural embedding $$\sigma_d:{\cD'_Z}^{\leq d}(X,V)/{\cD'_Z}^{\leq d-1}(X,V)\hookrightarrow \cD'(X,V\otimes \Sym^d(CN_{z,Z}^X))\,.$$
Using partition of unity, this construction extends to any bundle $\cE$ with \Fre fibers.

Let a Lie group $G$ act on $X$,  let $a:G\times X\to X$ denote the action map and $p:G\times X\to X$ denote the projection. A $G$-equivariant  bundle $\cE$ on $X$ is a  bundle $\cE$ on $X$ together with an isomorphism $a^*\cE\simeq p^*\cE$ satisfying the usual cocycle condition. Note that this structure also defines an isomorphism $a^!\cE\simeq p^!\cE$. Note also that the dual of an equivariant bundle has a canonical equivariant structure.

We denote by $\Delta_G$ the modular function of $G$.

We define a smooth representation of $G$ to be a $G$-equivariant bundle on a point. Note that a smooth \Fre representation of moderate growth is also a smooth representation according to this definition.

Note that if $G$ is a Lie group, $X$ is a $G$-manifold,  $p:X\to Y$ is a $G$-invariant map ({\it i.e.} $p(gx)=p(x)$), and $\cF$ is a bundle on $Y$, then $p^!\cF$ has a natural $G$-equivariant structure.
\begin{lemma}\label{lem:descent}
Let $G$ be a Lie group, $Y$ be a smooth manifold, and $p:X\to Y$ be a $G$-principal space over $Y$. Let $\cE$ be a $G$-equivariant bundle on $X$.
Then there exists a natural bundle $\cF$ over $Y$ and an isomorphism of $G$-equivariant bundles $p^!\cF\cong \cE$ such that
$p^*$ defines an isomorphism $\cD'(Y,\cF)\cong \cD'(X,\cE)^G$.
\end{lemma}
The proof is standard, but we will include it here since our bundles are infinite-dimensional.

We will need the following notation and lemmas.

\begin{notn}
For any continuous representation $A$ of $G$, denote $A(G):=\Span\{v-gv\, \vert \, g\in G,\, v\in A\}$, and $A_G:=A/\overline{A(G)}$.
\end{notn}

\begin{lemma}
Let $A$ be a continuous representation of $G$, and $B$ be a nuclear space.
Let $G$ act on $A\hot B$ by acting on $A$. Then the natural map $\alp:(A_G\hot B)^*\to ((A\hot B)^*)^G$ is an isomorphism.
\end{lemma}
\begin{proof}
We have $(A_G\hot B)^*\cong \Bil(A_G,B)$ and $(A\hot B)^*)\cong \Bil(A,B)$, where $\Bil(A,B)$ denotes the space of continuous bilinear maps $A\times B \to \C$ (see {\it e.g.} \cite[Ch. 41]{Tre}).
The map $\alp$ is defined by the map $\alp':\Bil(A_G,B)\to \Bil(A,B)^G$
which in turn is given by the projection $pr:A\times B\to A_G\times B$. Since $pr$ is onto, $\alp'$ is injective. To show that $\alp'$ is onto, choose $\omega\in \Bil(A,B)^G$. Since the left kernel of $\omega$ includes $A(G)$, $\omega$ factors through a bilinear map $\omega':A_G\times B\to \C$. Since  $pr:A\times B\to A_G\times B$ is open and surjective, $\omega'$ is continuous.
\end{proof}

\begin{lemma}\label{lem:coinv}
Let $G$ be a Lie group, and $Y$ be a smooth manifold. Let $G$ act on $Y\times G$ by left shifts on $G$, and let $p:Y\times G\to Y$ denote the projection.
Then $p_*$ defines an isomorphism of topological vector spaces $$C_c^{\infty}(Y\times G, D_{G})_G \cong C_c^{\infty}(Y) $$
\end{lemma}
\begin{proof}
Denote by   $C_c^{\infty}(Y\times G, D_{G})_0$ the kernel of $p_*$. Let us first show that
\begin{equation}\label{=0Int}
\overline{ C_c^{\infty}(Y\times G, D_{G})(G)} =C_c^{\infty}(Y\times G, D_{G})_0 \,\,.
\end{equation}
Since $p_*$ is a $G$-invariant morphism, the inclusion $\subset$ follows.
For the other inclusion, let $f\in C_c^{\infty}(Y\times G, D_{G})_0$, and approximate it by a sequence $f^j$ of the form $f^j=\sum_{j=1}^{n_j}q^j_i\otimes h^j_i$, with $q^j_i\in C_c^{\infty}(Y)$ and $h^j_i\in  C_c^{\infty}( G, D_{G})$.
Fix $\rho\in C_c^{\infty}( G, D_{G})$ with $\int_{G} \rho =1$ and let $F^j:=f^j-p_*(f^{j})\otimes \rho$. Then we have
$$F^j=f^j-\left(\sum_{i=1}^{n_j} \left (\int_{G} h^j_i\right )q^j_i\right)\otimes \rho = \sum_{i=1}^{n_j}q^j_i\otimes \left (h^j_i- \left (\int_{G} h^j_i \right)\rho\right)\in C_c^{\infty}(Y)\otimes C_c^{\infty}(G, D_{G})_0, $$
where $C_c^{\infty}( G,D_{G})_0$ denotes smooth compactly supported measures on $G$ with zero integral.

By \cite[Theorem 1]{BW} we have $$C_c^{\infty}( G,D_{G})_0=C_c^{\infty}( G,D_{G})(G) \, ,$$ and thus $F^j\in C_c^{\infty}(Y\times G, D_{G})(G)$. Now, since $p_*(f)=0$, we have $p_*(f^{j})\to 0$,  thus $F^j-f^j\to 0$, and thus $F^j\to f$. Thus $f\in \overline{ C_c^{\infty}(Y\times G, D_{G})(G)}$ and \eqref{=0Int} holds. This shows that $p_*$ defines a continuous linear isomorphism between $C_c^{\infty}(Y\times G, D_{G})_G$ and $C_c^{\infty}(Y) $. To see that its inverse is continuous, it is enough to construct a continuous section of $p_*$. One such section is given by $f\mapsto f\otimes \rho$.
\end{proof}

\begin{proof}[Proof of Lemma \ref{lem:descent}]
$\,$
\begin{enumerate}
\item Case 1: $X=Y\times G$.

Let $i:Y\to X$ be defined by $i(y):=(y,1)$ and $\cF:=i^!\cE$. The isomorphism $p^!\cF\cong \cE$ is given by the $G$-equivariant structure of $\cE$.
By partition of unity, we can assume that $\cF$ is a constant bundle and denote its fiber by $V$. We have to show that $p^*$ defines an isomorphism
$\cD'(Y,V)\cong \cD'(Y\times G,V\otimes D_{G,1})^G$, where $D_{G,1}$ is the fiber of the bundle $D_G$ at $1$. By Lemma \ref{lem:coinv} we have
\begin{align*}
C_c^{\infty}(Y\times G, D_{G,1})_G \cong C_c^{\infty}(Y).
\end{align*}

Now, the map $p^*$ decomposes as
\begin{multline*}
$$\cD'(Y,V) = (C_c^{\infty}(Y, V))^* \cong (C_c^{\infty}(Y)\hot V)^* \cong (C_c^{\infty}(Y\times G,D_{G,1})_G\hot V)^* \cong\\\cong ((C_c^{\infty}(Y\times G,D_{G,1})\hot V)^*)^G \cong
(C_c^{\infty}(Y\times G,V\otimes D_{G,1})^*)^G=\cD'(Y\times G,V\otimes D_{G,1})^G.$$ \end{multline*}

\item The general case.

Note that if there exists $\cF$ and an isomorphism $\nu: p^!\cF\cong \cE$ of $G$-equivariant bundles then such $\cF$ and $\nu$ are unique, in the sense that for any other such pair $(\cF',\nu')$ there exists an isomorphism $\mu:\cF\cong \cF'$ such that $\nu=\nu'\circ p^!(\mu)$, and such that $\mu$ is unique. Thus it is enough to construct $\cF$ locally, which is done in Case 1. Now, $p^*:\cD'(Y,\cF)\cong \cD'(X,\cE)^G$ is an isomorphism by partition of unity and Case 1.\qedhere
\end{enumerate}
\end{proof}

\subsection{Distribution vanishing theorems and their proofs}\label{sec:thm}

\begin{theorem}\label{thm:TrivAct}
Let $X$ be a smooth manifold and $Z\subset X$ be a locally closed subset.  Let $C$ be a  Lie group with finitely many connected components. Let $C$ act trivially on $X$. Let $L$ be a $C$-equivariant line bundle on $X$.  Let   $H$ be a Lie group  with a smooth action on $C$. For any $z\in Z$ let $L_z$ denote the character by which $C$ acts on the fiber $L|_z$.
Let $\widehat{C}_{\C}$ denote the manifold of all characters of $C$.
Assume that the set $\{L_z\, \vert \, z\in Z\}$ lies in a finite union of  locally closed $H$-orbits in $\widehat{C}_{\C}$.
  Let
 $V$ be a smooth representation of $H \ltimes C$ in a \Fre space.
Assume that for any $z\in Z$ we have $((V\otimes L|_z)^{*})^{C}=0$.
Then
$$\cD'_Z(X,V\otimes L)^C=0.$$
\end{theorem}
\begin{proof}
By partition of unity, we may assume that $L$ is trivial as a line bundle,  that $X=\R^n$ and $Z$ is compact. Let $L_0$ denote the natural $C$-equivariant line bundle on $\widehat{C}_{\C}$.
There exists a unique smooth map $\psi:X\to \widehat{C}_{\C}$ such that $L\cong \psi^*(L_0)$. Define $\Gamma \subset Z\times \widehat{C}_{\C}\subset X\times \widehat{C}_{\C}$ to be the graph of the restriction $\psi|_Z$. Let $\tilde L:=\C_X \boxtimes L_0$. It is enough to show that  $\cD'_{\Gamma}(X\times \widehat{C}_{\C},V\otimes \tilde L)^C=0.$

Assume the contrary. Let $0\neq  \xi \in \cD'_{\Gamma}(X\times \widehat{C}_{\C},V\otimes \tilde L)^C$. Let $G:=X\times H \ltimes C$ act on $X\times \widehat{C}_{\C}$ by
$$(x,h,c)(y,\chi):= (x+y,\chi\circ a(h^{-1})) $$
where $a(h)$ denotes the action of $H$ on $C$. Define a structure of a $G$-equivariant bundle on $\tilde L$ through the action on the total space by
$$(x,h,c)(y,\chi,\alp):= (x+y,\chi\circ a(h^{-1}), \chi(c)\alp).$$
Define a representation of $G$ on $V$ by letting $X$ act trivially.
By the conditions, $\Supp(\xi)$ lies in the union of finitely many locally closed $G$-orbits on $X\times \widehat{C}_{\C}$. By Corollary \ref{cor:OpOrb}, for one of those orbits $\cO$,  the intersection $\cO\cap  \Supp(\xi)$ is  open and non-empty in $\Supp(\xi)$. Let $x\in \cO\cap  \Supp(\xi)$. There exists a cutoff function $\rho$ such that $x\in \Supp( \rho\xi) \subset \cO$. This leads to a contradiction since by \cite[Theorem 3.15(i)]{KV}, $ \rho\xi =0$.
\end{proof}
By partition of unity, we obtain the following corollary.

\begin{corollary}\label{cor:TrivAct}
Let $X,Z,C,H,$ and $L$ be as in Theorem \ref{thm:TrivAct}. Let $H \ltimes C$ act trivially on $X$, and let
 $\cE$ be a locally constant  $H \ltimes C$-equivariant bundle on $X$, {\it i.e.} an equivariant bundle that is locally given by a single representation of $H \ltimes C$.
Assume that for any $z\in Z$ we have $(((\cE\otimes L)|_z)^{*})^{C}=0$.
Then
$$\cD'_Z(X,V\otimes L)^C=0.$$
\end{corollary}

We will need the following corollary of Lemma \ref{lem:descent}.

\begin{corollary}
Let $G$ be a Lie group and $H_1,H_2$ be closed Lie subgroups.
Consider the two-sided action of $H_1\times H_2$ on $G$, and let $\cE$ be an $H_1\times H_2$-equivariant bundle on $G$. Let $p:G\to G/H_2$ denote the natural projection. Then there exists a natural $H_1$-equivariant bundle $\cF$ on $G/H_2$ and a natural isomorphism $p^!\cF\cong \cE$ such that  $p^*$ defines an isomorphism $$\cD'(G/H_2,\cF)^{H_1}\cong \cD'(G,\cE)^{H_1\times H_2}.$$
\end{corollary}

\begin{corollary}\label{cor:twoPulls}
Let $G$ be a Lie group and $H_1,H_2$ be closed Lie subgroups.
Let $\cF_1$ be an $H_1$-equivariant bundle on $G/H_2$.
Let $p_1:G\to H_1\backslash G$ and $p_2:G\to G/H_2$ denote the natural projections.

Then there exists a natural $H_2$-equivariant bundle $\cF_2$ on $H_1\backslash G$ such that
$p_1^!\cF_2\cong p_2^!\cF_1$ as $H_1\times H_2$-equivariant bundles, and  $$\cD'(G/H_2,\cF_1)^{H_1}\cong\cD'(H_1\backslash G,\cF_2)^{H_2}.$$
\end{corollary}

\begin{notn}
Let a Lie group $G$ act on a smooth manifold $X$.  Let $C\subset G$ be a subgroup.
Let $L$ be a $C$-equivariant line bundle on $X$. Then for any $x\in X$ and $g\in G$ we define a character $L_x^g:C_x\to \C^{\times}$ by letting $L_x^g(c)$ be the scalar by which $gcg^{-1}$ acts on the fiber ${L_{gx}}$.
\end{notn}

\begin{definition}\label{def:conv}
Let a Lie group $C$ act on a smooth manifold $X$. Let $\cF$ be a $C$-equivariant \Fre bundle over $X$. Let $Z\subset X$ be a locally closed $C$-invariant subset. We call the quadruple $(C,X,Z,\cF)$  \emph{convenient} if there exist
\begin{enumerate}[leftmargin=*,label=\alph*)]
\item A Lie group $G \supset C$ acting smoothly on $X$ extending the action of $C$,
\item A  $G$-equivariant \Fre bundle $\cE$ on $X$,
\item and a $C$-equivariant line bundle $L$ on $X$,
\end{enumerate}
such that the following holds:
\begin{enumerate}[leftmargin=*,label=\roman*)]
\item $\cF\cong \cE \otimes L$ as a $C$-equivariant bundle.
\item $C$ is a normal subgroup of $G$.
\item \label{it:conv:CharFin}For any $z\in Z$, the collection $\{L_z^g\in \widehat{(C_z)}_{\C}\, |\, g\in G \} $ lies in a finite union of locally closed $G_z$-orbits.
\item $Z$ is contained in a union of finitely many locally closed $G$ orbits.
\end{enumerate}
\end{definition}

\begin{theorem}\label{theorem:convenient vanishing}
Let $(C,X,Z,\cF)$ be a convenient quadruple.
Suppose that for any $z\in Z$
and any $k\geq 0$ we have
\begin{align}\label{=PointwiseVanishing}
(\cF|_z^* \otimes \Sym^k(N_{Cz,z}^X) \otimes (\Delta^{-1}_C)|_{C_z}\otimes\Delta_{C_z})^{C_z} =0.
\end{align}
Then $\cD'_Z(X,\cF)^C=0$.
\end{theorem}
\begin{proof} We divide the proof into several cases.
\begin{enumerate}

\item $G$ acts transitively on $X$.

Fix $x_0\in X$ and let $\nu:G\to X$ be the corresponding action map.
Let $G':=X_1:=C\backslash G$. By Corollary \ref{cor:twoPulls}, there exists a $G_{x_0}$-equivariant bundle $\cF_1$ on $X_1$ such that $\cD'(X,\cF)^C\cong \cD'(X_1,\cF_1)^{G_{x_0}}$
and $\nu^!\cF\cong p^!\cF_1$, where  $p:G \to X_{1}$ is the projection.
We construct bundles $\cE_1$ and $L_1$ on $X_1$ in a similar way, and have $\cF_1=\cE_1\otimes L_1$.

Let $H_1:=G_{x_0}$ and $C_{1}:=H_1\cap C$. Then it is enough to show that $\cD'(X_1,\cF_1)^{C_{1}}=0$.
We  deduce this from Corollary \ref{cor:TrivAct}.
For this we need to show that
\begin{enumerate}
\item $\cE_1$ is locally constant as a $C_1$-equivariant bundle.
\item The action of $C_1$ on the fibers of $\cE_1$ extends to  $H_1$.
\item The set $\{(L_1)_z\, \vert \, z\in X_1\}$ lies in a finite union of  locally closed $H_1$-orbits in $\widehat{(C_1)}_{\C}$.
\item For any $z\in X_{1}$ we have $(((\cE_1\otimes L_1)|_z)^*)^{C_1}=0$.
\end{enumerate}

Proof of (a). It is enough to show that $\cE_1':=\cE_1\otimes D^{-1}_{X_1}$ is locally constant as a $C_1$-equivariant bundle.
 Since $p$ locally has a section, it is enough to show that $p^*(\cE'_1)$ is constant as an $H_1$-equivariant bundle with respect to the action of $H_1$ on $G$ by right multiplications. We have $p^*(\cE'_1)=\nu^*(\cE')$, where $\cE':=\cE\otimes D^{-1}_{X} $. This gives a structure of a $G\times H_{1}$-equivariant bundle on $p^*(\cE'_1)$ with respect to the two-sided action. This implies that $p^*(\cE'_1)$ is constant as an $H_1$-equivariant bundle.

Proof of (b). The fiber of $\cE_1$ at $[1]$ is isomorphic to $\cE|_{x_0}\otimes \Delta_{H_1}|_{C_1}\otimes \Delta_C^{-1}|_{C_1}$ as a representation of $C_1$.
Since $C$ is normal in $G$, we have $\Delta_C=\Delta_G|_C$, and thus the representation $\cE_{x_0}\otimes \Delta_{H_1}|_{C_1}\otimes \Delta_C^{-1}|_{C_1}$ extends to $H_1$.

Proof of (c). $L$ satisfies condition \ref{it:conv:CharFin} of Definition \ref{def:conv}. Thus so does $L':=L\otimes D_X^{-1}$, since the action of $C$ on $D_X^{-1}$ can be extended to $G$. It is enough to show (c) with $L_1$ replaced by $L_1':=L\otimes D_{X_1}$. Now, we have $p^*(L_1')=\nu^*(L')$. Thus for any $g\in G$ we have $(L')_{x_0}^g=(L_1')_{[g]}$ and thus
$$\{(L_1')_{y}\in \widehat{C}_{\C}\, |\, y\in X_1 \} =\{(L'_{x_0})^g\in \widehat{C}_{\C}\, |\, g\in G \}.$$

Statement (d) follows from \eqref{=PointwiseVanishing} by  a  straightforward computation.\\

\item $Z$ lies in a single  closed $G$-orbit $\cO$.

Suppose by way of contradiction that there exists $0\neq \xi\in \cD'_Z(X,\cF)^C$ and let $z\in \Supp(\xi)$. Let $d$ be the transversal degree of $\xi$ to $\cO$ at $z$. Let $X_1:=\{p\in \cO\, \vert \, \deg_{p,\cO}(\xi)=d\}.$
Consider  $$\sigma_d(\xi)|_{X_1}\in \cD'_{Z\cap X_1}(X_1,\cF\otimes \Sym^d(CN_{X_1}^X))^C.$$
By the previous case we obtain $\sigma_d(\xi)=0$ - contradiction!\\

\item The General case.

We prove this step by induction on the number $n$ of orbits of $G$ in $GZ$.
When $n=0$, $Z$ is empty and the statement is obvious. For $n\geq1$, Corollary \ref{cor:OpOrb} implies that there exists an open orbit $\cO\subset Z$. Let $Z':=Z\setminus \cO$, $X':=X\setminus Z'$. Then we have the exact sequence
$$0\to \cD'_{Z'}(X,\cF)^C\to \cD'_{Z}(X,\cF)^C\to \cD'_{\cO}(X',\cF)^C.$$
We have $\cD'_{\cO}(X',\cF)^C=0$ by the previous case, and $\cD'_{Z'}(X,\cF)^C=0$ by the induction hypothesis.\qedhere
\end{enumerate}
\end{proof}
\begin{remark}
Substituting for $\mathcal{E}$ the constant bundle with fiber $V$ and for $L$ a constant line bundle, we obtain Theorem~\ref{thm:AG}.
\end{remark}

\def\cprime{$'$} \def\cprime{$'$} \def\cprime{$'$}

\end{document}